\documentclass[12pt,reqno]{amsart}
\usepackage{amssymb}
\usepackage{relsize}
\usepackage[utf8]{inputenc}
\usepackage[french, english]{babel}
\usepackage{tikz}
\usepackage{tikz-cd}
\usetikzlibrary{matrix,arrows}
\usepackage[a4paper, total={6in, 8in}, centering]{geometry}
\usepackage{enumitem}
\usepackage[T1]{fontenc}
\usepackage{lmodern}
\usetikzlibrary{positioning}
\definecolor{processblue}{cmyk}{0.96,0,0,0}
\usepackage[colorlinks=true,hyperindex]{hyperref}
\hypersetup{
    colorlinks,
        linkcolor=violet,
        citecolor=cyan,
        urlcolor=orange
}
\usepackage{appendix}
\usepackage{stmaryrd}
\usepackage{mathrsfs}  
\newcommand{\CC}{\mathbb{C}}
\newcommand{\RR}{\mathbb{R}}
\newcommand{\QQ}{\mathbb{Q}}
\newcommand{\ZZ}{\mathbb{Z}}

\newcommand{\PP}{\mathbb{P}}
\newcommand{\GG}{\mathbb{G}}
\newcommand{\AAA}{\mathbb{A}}

\newcommand{\FF}{\mathbb{F}}

\newcommand{\pf}{\mathfrak{p}}

\newcommand{\mf}{\mathfrak{m}}
\newcommand{\nf}{\mathfrak{n}}
\newcommand{\af}{\mathfrak{a}}

\newcommand{\Sf}{\mathfrak{S}}
\newcommand{\gf}{\mathfrak{g}}

\newcommand{\tf}{\mathfrak{t}}
\newcommand{\bff}{\mathfrak{b}}
\newcommand{\wf}{\mathfrak{w}}
\newcommand{\uf}{\mathfrak{u}}

\newcommand{\Bh}{\mathcal{B}}

\newcommand{\Dh}{\mathcal{D}}
\newcommand{\Eh}{\mathcal{E}}
\newcommand{\Fh}{\mathcal{F}}
\newcommand{\Gh}{\mathcal{G}}

\newcommand{\Ih}{\mathcal{I}}

\newcommand{\Lh}{\mathcal{L}}
\newcommand{\Mh}{\mathcal{M}}

\newcommand{\Oh}{\mathcal{O}}
\newcommand{\Ph}{\mathcal{P}}

\newcommand{\Sh}{\mathcal{S}}

\newcommand{\Uh}{\mathcal{U}}

\newcommand{\Zh}{\mathcal{Z}}
\newcommand{\abf}{\mathbf{a}}
\newcommand{\fbf}{\mathbf{f}}
\newcommand{\gbf}{\mathbf{g}}
\newcommand{\Ascr}{\mathscr{A}}
\newcommand{\Bscr}{\mathscr{B}}
\newcommand{\Cscr}{\mathscr{C}}

\newcommand{\Gscr}{\mathscr{G}}
\newcommand{\Hscr}{\mathscr{H}}
\newcommand{\Iscr}{\mathscr{I}}

\newcommand{\Nscr}{\mathscr{N}}

\newcommand{\Sscr}{\mathscr{S}}
\newcommand{\Tscr}{\mathscr{T}}
\newcommand{\Uscr}{\mathscr{U}}
\newcommand{\Vscr}{\mathscr{V}}
\newcommand{\Wscr}{\mathscr{W}}

\newcommand{\Zscr}{\mathscr{Z}}

\newcommand{\Grm}{\mathrm{G}}

\newcommand{\Irm}{\mathrm{I}}

\newcommand{\Prm}{\mathrm{P}}

\newcommand{\Res}{\on{R}}

\newcommand{\spec}{\on{Spec}}
\newcommand{\spf}{\on{Spf}}

\newcommand{\Gr}{\on{Gr}}

\newcommand{\on}{\operatorname}
\theoremstyle{plain}
\newtheorem{thm}{Théorème}[subsection]
\newtheorem{lem}[thm]{Lemme}
\newtheorem{propn}[thm]{Proposition}
\newtheorem{cor}[thm]{Corollaire}

\theoremstyle{definition}
\newtheorem{defn}[thm]{Définition}
\newtheorem{exmp}[thm]{Exemple}
\newtheorem{rem}[thm]{Remarque}
\newtheorem{cons}[thm]{Construction}
\newtheorem{hypt}[thm]{Hypothèse}

\theoremstyle{plain}
\newtheorem{thmintro}{Théorème}[section]

\newtheorem{quesintro}[thmintro]{Question}

\newcommand{\pot}[1]{ [\hspace{-0,5mm}[ {#1} ]\hspace{-0,5mm}] }
\newcommand{\rpot}[1]{ (\hspace{-0,7mm}( {#1} )\hspace{-0,7mm}) }
\numberwithin{equation}{subsection}

\begin{document}
\title{Grassmanniennes affines tordues sur les entiers}
\author[J. Louren{\c c}o]{Jo\~ao Louren{\c c}o}
\address{Mathematisches Institut der Universität Bonn, Endenicher Allee 60, 53115 Bonn, Deutschland}
\email{lourenco@math.uni-bonn.de}
\address{Department of Mathematics, Imperial College London, London SW7 2AZ, UK}
\email{j.nuno-lourenco@imperial.ac.uk}
\dedicatory{À Jacques Tits pour son nonantième anniversaire}

\selectlanguage{french}

\begin{abstract}
On généralise les travaux de Pappas--Rapoport--Zhu sur les grassmanniennes affines tordues aux groupes sauvagement ramifiés, quasi-déployés et résiduellement déployés, à tores maximaux induits. Ceci repose sur la construction, inspirée des travaux de Tits, de certains $\ZZ[t]$-groupes lisses, affines et connexes de type parahorique. Les $\FF_p(t)$-groupes échéants sont pseudo-réductifs éventuellement non standards au sens de Conrad--Gabber--Prasad, et leurs $\FF_p\pot{t}$-modèles sont parahoriques à notre sens généralisé. On étudie leurs grassmanniennes affines, démontrant notamment la normalité des variétés de Schubert et un théorème de cohérence à la Zhu.
\end{abstract}

\maketitle
{\hypersetup{linkcolor=violet}

\tableofcontents
}
\section{Introduction}

Les grassmanniennes affines $\Gr_G$ de $k$-groupes réductifs $G$ sont des $k$-ind-schémas qui ont été introduits dans les années 90 en théorie géométrique des représentations pour étudier le champ des $G$-fibrés sur une courbe à l'aide de l'uniformisation de Beauville--Laszlo \cite{BLConfBlocks}. Cependant, leur importance n'a été mise en relief que dans les décennies suivantes, dans fondamentalement deux programmes : d'un côté, l'équivalence de Satake géométrique, et de l'autre, la théorie des modèles entiers de variétés de Shimura ou plutôt de ce qu'on appelle leurs modèles locaux. Pour cette dernière application, il faut comprendre vraiment la structure géométrique de certains sous-schémas fermés de $\Gr_G$. En fait, la stratégie adoptée à la suite de Faltings \cite{Fal} pour parvenir aux résultats désirés consiste à tirer profit du fait que l'on ait une réalisation entière de la grassmannienne affine. Lorsqu'on veut traiter les groupes non constants et tordus, la situation devient plus compliquée et les divers travaux monumentaux \cite{PR}, \cite{ZhuCoh}, \cite{PZh} de Pappas--Rapoport--Zhu consacrés à ce problème ne réussirent pas à éviter des hypothèses de ramification modérée - à savoir la grassmannienne affine tordue n'est réalisée qu'au-dessus de $\ZZ[1/e]$, où l'on note $e$ l'indice de ramification. 

Dans l'article présent, on se propose d'étendre la théorie existante pour les groupes quasi-déployés aux cas de ramification sauvage et d'obtenir ainsi des réalisations entières des grassmanniennes affines tordues, voir le théorème \ref{thintro2}. L'étape-clé est la construction d'un certain $\ZZ[t]$-groupe $\underline{\Gscr_{\fbf}}$ affine, lisse et connexe, qu'on regarde comme une famille entière de groupes parahoriques au-dessus de chaque corps premier et dont la grassmannienne affine est celle qu'on cherche, sous des hypothèses légères sur le $\QQ(t)$-groupe $G$ fibre générique de $\underline{\Gscr_{\fbf}}$, voir le théorème \ref{thintro 1}. Notre approche diffère de celle de \cite{PZh} en ce que les auteurs de ce dernier article choisissent relever leurs groupes en caractéristique mixte aux anneaux des polynômes, tandis que nous adoptons le point de vue de la réduction de la caractéristique nulle à la positive. Toutefois, on s'attend à ce que notre article fournisse les outils nécessaires, voir le théorème \ref{thintro3}, pour faire des avances dans la théorie des modèles locaux $p$-adiques en des cas de ramification sauvage, telle qu'elle fut exposée dans les travaux de Scholze--Weinstein \cite{SchBerk} et de He--Pappas--Rapoport \cite{HPR}, voir notre dissertation \cite{LouDiss} et le dernier paragraphe de cette introduction.

Signalons que, afin d'étudier l'arithmétique de nos groupes en les caractéristiques sauvages, il faut utiliser la théorie générale et classification des groupes pseudo-réductifs sur les corps développée par Conrad--Gabber--Prasad \cite{CGP}, ainsi que leur théorie schématique de Bruhat--Tits sur les corps discrètement valués, complets et à corps résiduels parfaits, qui fut développée en \cite{LouBT}. Cet outil-là a été développé pour qu'on puisse démontrer certaines propriétés générales concernant les groupes algébriques linéaires sur les corps de fonctions, comme en \cite{ConFin} ou \cite{Ros}. Néanmoins, le présent article est le seul à mettre au centre de l'attention les groupes pseudo-réductifs non standards, depuis les travaux de Tits consacrés aux groupes de Kac--Moody, \cite{TitsOber} et \cite{TitsBour}, où les $\ZZ[t^{\pm 1}]$-groupes $\underline{G}$ furent construits, voir la question \ref{quesintro1}. Cependant, ce fut originalement à travers l'article \cite{Lev} de Levin, consacré aux modèles locaux $p$-adiques pour les restrictions des scalaires de groupes modérés, que nous sommes conduit au cadre pseudo-réductif, car celui-là faisait implicitement usage des extensions purement inséparables.

Pour conclure en ce qui touche aux références historiques, nous aimerions rappeler au lecteur que, dans le cadre de Kac--Moody, la communauté mathématique avait déjà étudié à fond les variétés de drapeaux sur $\CC$ typiquement de dimension infinie associées aux algèbres de Kac--Moody $\gf$, qui admettent un prolongement naturel en des $\ZZ$-ind-schémas d'après Mathieu \cite{Mat}. Comme produit dérivé des résultats qu'on va présenter, se dégage un dictionnaire arithmétique complet entre les grassmanniennes affines tordues et les variétés de drapeaux de Kac--Moody affines, voir l'annexe \ref{comparaison avec les varietes de demazure de kac-moody}, ce qui était attendu en les caractéristiques petites depuis longtemps. Le théorème \ref{thintro2} sur la normalité des variétés de Schubert retient néanmoins son importance, parce que leurs définitions ne se ressemblent guère et qu'on y recourt pour établir la comparaison. Il ne doit être alors aucune surprise que Tits soit arrivé au milieu de sa recherche sur les groupes de Kac--Moody aux constructions sur lesquelles on s'appuie, et que cette théorie soit progressivement devenue une grosse source d'inspiration de cet article. 

\subsection{Groupes tordus sur $\ZZ[t]$.} Soient $k$ un corps algébriquement clos, $K=k\rpot{t}$ le corps de séries de Laurent à coefficients dans $k$ et $G$ un $K$-groupe réductif, connexe, quasi-déployé, simplement connexe et absolument presque simple. Fixons un épinglage $(H, T_H, B_H, X_H)$ au-dessus de $\ZZ$ de la $\ZZ$-forme déployée $H$ de $G$ et un isomorphisme $G \otimes_K K^{\on{sép}} \simeq H \otimes_K K^{\on{sép}}$ de telle sorte que le groupe de Galois $\on{Gal}(K^{\on{sép}}/K)$ opère sur le membre de droite par automorphismes qui préservent l'épinglage donné. Comme le groupe $\on{Aut}(H, T_H, B_H, X_H)$ est isomorphe au groupe d'automorphismes du diagramme de Dynkin de $H$, il s'identifie au groupe symétrique $\Sf_e$ d'indice $e=1$, $2$ ou $3$. Si $e$ est différent de la caractéristique $p$ de $k$, alors la plus petite extension de $K$ déployant $G$ est $L=k\rpot{t^{1/e}}$ et l'on a un isomorphisme $G =(\Res_{L/K} H)^{\Gamma}$, où $\Res_{L/K} H$ désigne la restriction des scalaires à la Weil de $H$ sur laquelle $\Gamma:=\on{Gal}(L/K)$ opère de façon galoisienne sur les coefficients et diagrammatique sur $H$, cf. constr. \ref{construction groupe quasi deploye}. Cependant, si $e=p$, alors le même type d'énoncé reste valide, mais il faut remplacer l'extension $k\rpot{t^{1/e}}$ de $k\rpot{t}$ par des classes d'isomorphisme en nombre infini de clôtures galoisiennes d'extensions séparables de degré $e$. 

Cette différence arithmétique se manifeste aussi dans la mesure où les groupes $G$ à ramification modérée et indice de Tits fixé, voir la classification de \cite{TitsClas}, qu'on vient de décrire peuvent être rassemblés dans un $\ZZ[1/e][t^{\pm 1}]$-groupe lisse, affine et connexe $\underline{G}$, en prenant la même définition comme invariants de la restriction des scalaires du groupe de Chevalley $H$. La question qui a déclenché ce travail fut la suivante :

\begin{quesintro}\label{quesintro1}
Y a-t-il un prolongement \og canonique \fg{} de $G$ en un schéma en groupes $\underline{G}$ au-dessus de $\ZZ[t^{\pm 1}]$ ? Si oui, quel groupe obtient-on en caractéristique $e$ ?
\end{quesintro}

La réponse est heureusement toujours affirmative et ce schéma en groupes fut originalement construit par Tits \cite{TitsOber}, en faisant usage des techniques de recollement de Bruhat--Tits \cite{BTII}, voir le th. \ref{theoreme d'existence et unicite des groupes attaches aux drs}, pour se ramener à construire une \og donnée radicielle schématique \fg{} $(\underline{T}, \underline{U_a})_{a \in \Phi^{\on{nd}}}$ constituée de certains $\ZZ[t^{\pm 1}]$-modèles naturels induits par le choix d'un système cohérent d'épinglages, voir la déf.~\ref{definition drs groupe de tits}. Il est toujours loisible d'identifier le $\ZZ[t^{\pm 1}]$-groupe $\underline{G}$ à un sous-schéma en groupes fermé de la restriction des scalaires de $H$ le long de $\ZZ[\zeta_e,t^{1/e}]/\ZZ[t]$, voir la prop. \ref{proposition plongement restrictions des scalaires groupe de tits}, mais on doit s'attendre au besoin de modifier à la Tits par quelques facteurs de $2$ le choix canonique de quasi-épinglages dû à Steinberg \cite{St1} et \cite{St2}, cf. défs. \ref{definition quasi-systemes de chevalley} et \ref{definition quasi-systeme de Tits}. Il s'ensuit que $\underline{G}$ est lisse, affine et à fibres connexes.

Quant au groupe algébrique que l'on obtient en caractéristique $e$, le lecteur aura peut-être soupçonné qu'il ne s'agit pas non plus d'un groupe réductif, à cause de la présence d'une extension de corps purement inséparable $\FF_e(t^{1/e})/\FF_e(t)$, et en fait $\underline{G}\otimes \FF_e(t)$ est un groupe pseudo-réductif non standard et pseudo-déployé au sens de \cite{CGP}, comme mentionné plus haut, cf. prop. \ref{proposition fibre generique quasi-reductive}. Ces groupes pseudo-réductifs non standards forment une classe assez exceptionnelle qui a été découverte par Tits et qui n'existe qu'en caractéristiques petites $p\leq 3$. La raison d'être de tels groupes, selon les auteurs de \cite{CGP}, est qu'il y ait des phénomènes de $p$-divisibilité dans le système de racines $\Phi_G$. Nous préférons les regarder comme le fruit de la volonté de prendre avec insistance des $\Gamma$-invariants galoisiens-diagrammatiques d'un groupe déployé, lorsque cela cesse d'être naïvement possible.

Venons-en à présent aux constructions de type parahorique. Soit $S$ le plus grand tore déployé de $G$ contenu dans le Cartan $T$. Admettons pour l'instant l'existence d'identifications canoniques entre les appartements $\Ascr(\underline{G},\underline{S},\FF_p\rpot{t})$ en caractéristiques différentes, voir la prop. \ref{proposition combinatoire caracteristiques differentes}, qu'on note plutôt $\Ascr(\underline{G}, \underline{S}, \ZZ\rpot{t})$ pour signaler que la combinatoire est indépendante de l'arithmétique. Ceci soulève le même type de questions que plus haut, comparer avec la question \ref{quesintro1} :

\begin{quesintro}\label{quesintro2}
Si on se donne une partie bornée et non vide $\Omega \subset \Ascr(\underline{G}, \underline{S}, \ZZ\rpot{t})$, existe-t-il un $\ZZ[t]$-groupe lisse, affine et connexe de type \og immobilier\fg{} $\underline{\Gscr_{\Omega}}$ modelant $\underline{G}$ ?
\end{quesintro}

Le mot \og immobilier \fg{} dans l'énoncé veut dire tout simplement que, au-dessus de $\FF_p\pot{t}$ pour tout $p \in \PP \cup \{0\}$, on obtient des modèles en groupes canoniquement associés à $\Omega$ au sens de la théorie de Bruhat--Tits généralisée de \cite{LouBT}. Il faut noter que, en dehors de $\{e=0\}$, soit dans \cite{PR}, soit dans \cite{PZh}, on peut essentiellement trouver une construction de ce schéma en groupes à l'aide des techniques des $\Gamma$-invariants, mais lorsqu'on insiste pour travailler à coefficients entiers, il n'existe pas à notre connaissance une référence dans la littérature où ce schéma en groupes soit construit. Pour construire cet objet, on se propose de trouver une donnée radicielle valuée et puis d'appliquer le théorème sur l'existence de solutions aux lois birationnelles.
 
Ceci étant, précisons le cadre un petit peu plus général dans lequel on va travailler. On se donne un $\QQ(t)$-groupe réductif connexe $G$ quasi-déployé, $\QQ(\zeta_e,t^{1/e})$-déployé, $\QQ\rpot{t}$-résiduellement déployé, dont le tore maximal $T$ est induit, ainsi qu'un système cohérent d'épinglages à la Tits des sous-groupes radiciels $U_a$. Le groupe de Tits $\underline{G}$, ainsi que ses modèles immobiliers $\underline{\Gscr_\Omega}$ associés à une partie bornée et non vide $\Omega\subseteq \Ascr(\underline{G}, \underline{S}, \ZZ\rpot{t})$, se construisent de manière analogue, cf. défs. \ref{definition groupe de Tits} et \ref{definition modeles immobiliers du groupe de tits}. Les conclusions de la première partie de l'article peuvent être résumées comme suit :

\begin{thmintro}[prop. \ref{proposition drs modele immobilier groupe de tits}, th. \ref{theoreme affinite modeles immobiliers}]\label{thintro 1}
Sous les hypothèses précédentes, il existe un et un seul $\ZZ[t]$-modèle en groupes lisse, affine, connexe et immobilier $\underline{\Gscr_\Omega}$ de $\underline{G}$ associé à $\Omega$.
\end{thmintro}

On dit un mot sur la preuve de l'affinité. Malheureusement, on ne peut pas suivre l'argument de \cite{PZh} verbatim, car celui utilise la cohomologie de groupes et l'interprétation en tant que $\Gamma$-invariants. Néanmoins, on réussit encore à démontrer l'affinité de la manière suivante. On observe que l'enveloppe affine, lequel est lisse d'après Raynaud \cite{Ray}, ne peut pas posséder des composantes connexes supplémentaires dans les fibres, en appliquant le fait que les groupes d'Iwahori--Weyl de $\underline{G}\otimes \FF_p\rpot{t}$ ne dépendent pas de la caractéristique.

\subsection{La géométrie des grassmanniennes affines tordues} Soient $A$ un anneau et $\Grm$ un $A\pot{t}$-groupe affine et lisse. La grassmannienne affine $\Gr_{\Grm}$ de $\Grm$ est le pré-faisceau sur la catégorie des $A$-algèbres $B$ qui classifie les $\Grm$-torseurs sur le disque $B\pot{t}$ munis d'une trivialisation sur le disque épointé $B\rpot{t}$, cf. déf. \ref{definition grassmannienne affine locale}. On dispose aussi des groupes d'arcs $L^+\Grm$ et de lacets $L\Grm$ donnés par $R \mapsto \Grm(R\pot{t})$ et $R\mapsto \Grm(R\rpot{t})$, tels que le faisceau quotient $L\Grm/L^+\Grm$ pour la topologie étale soit isomorphe à $\Gr_{\Grm}$. Sa représentabilité par un $A$-ind-schéma en est assurée dès que $A$ soit un anneau de Dedekind et $\Grm$ soit lisse, affine et connexe, cf. prop. \ref{proposition representabilite grassmannienne affine locale}.

On s'intéresse surtout au cas où $A=\ZZ$ et $\Grm=\underline{\Gscr_\fbf}$ est parahorique associé à une facette $\fbf \subset \Ascr(\underline{G},\underline{S},\ZZ\rpot{t})$, en reprenant les notations de la partie précédente concernant les groupes tordus, cf. question \ref{quesintro2} et théorème \ref{thintro 1}. La parahoricité est condition nécessaire et suffisante pour que la grassmannienne affine soit un ind-schéma projectif, cf. \ref{proposition projectivite grassmannienne affine locale}, grâce à un théorème de Richarz \cite{Rich} que nous avons étendu au cas pseudo-réductif dans \cite{LouBT}. 

Pour une alcôve fixe $\abf$ adhérente à $\fbf$ et chaque classe $w \in \widetilde{W}/W_{\textbf{f}}$ du groupe d'Iwahori--Weyl, on dispose du schéma de Schubert $\Gr_{\underline{\Gscr_\fbf},\leq w}$, défini comme l'orbite fermée sous $L^+\underline{\Gscr_\abf}$ d'un représentant convenable $n_w \in L\underline{N}(\ZZ)$ de $w$, cf. déf. \ref{definition schemas de schubert}. Pour étudier leur géométrie, on se sert des schémas de Demazure $\on{Dem}_\wf$ lisses et projectifs sur $\ZZ$, produit tordu de droites de Schubert, tel que le morphisme vers la grassmannienne se factorise en une application birationnelle et surjective $\on{Dem}_\wf \rightarrow \Gr_{\underline{\Gscr_\fbf},\leq w}$. Voici le résultat central :

\begin{thmintro}[th. \ref{theoreme normalite des schemas de Schubert}]\label{thintro2}
Supposons que $G^{\on{dér}}=G^{\on{sc}}$. Alors les schémas de Schubert $\Gr_{\underline{\Gscr_\fbf},\leq w}$ sont géométriquement normaux sur $\ZZ$, de Cohen--Macaulay et leurs fibres modulo $p>0$ sont scindées.
\end{thmintro}

La démonstration du théorème ci-dessus ressemble à celles de Faltings \cite{Fal} lorsque $G$ est déployé ou de Pappas--Rapoport \cite{PR} au-dessus de $\ZZ[1/e]$, mais il y a des différences importantes que nous allons expliquer dans la suite. Avant de commencer, nous voudrions aussi signaler que l'hypothèse $G^{\on{dér}}=G^{\on{sc}}$ est impérative, si l'on veut éviter des complications dont la découverte fut le sujet de notre collaboration récente avec Haines--Richarz \cite{HLR} : en effet, dans ce cas les schémas de Schubert possèdent des fibres spéciales non réduites à réduits non normaux, cf. cors. \ref{corollaire grassmannienne affine reduite} et \ref{corollaire non normalite varietes de schubert}.

La première étape est composée d'un traitement exhaustif de la structure du groupe d'arcs $L^+\underline{\Gscr_\fbf}$, voir les props. \ref{proposition decomposition de levi} et \ref{proposition engendrement produit direct sous-groupes radiciels}, telle que de son opposé le groupe de cordes $L^-\underline{\Gscr_\fbf}$. Notre approche de celui-ci est en fonction d'une extension convenable du $\AAA_{\ZZ}^1$-groupe $\underline{\Gscr_\fbf}$ à $\PP_\ZZ^1$, qui nous permet de poser $L^-\underline{\Gscr_\fbf}(R)=\underline{\Gscr_\fbf}(R[t^{-1}])$, cf. déf. \ref{definition groupe de cordes}. La méthode de preuve ici adoptée pour la plupart des propriétés fondamentales concernant le groupe de cordes se distingue ainsi complètement de celle de de Cataldo--Haines--Li dans \cite{dHL18} pour les groupes déployés (reprise dans \cite{HLR} pour les groupes tordus modérément ramifiés), en substituant aux calculs interminables une analyse raffinée n'impliquant que l'action de rotation, voir le th. \ref{theoreme proprietes groupe de cordes}. 

Puis on construit les $L^-\underline{\Gscr_\abf}$-orbites fermées $\Gr_{\underline{\Gscr_\fbf}}^{\geq w}$ dans la grassmannienne affine à la suite de \cite{Fal}, cf. th. \ref{theoreme cycles de schubert topologie et normalite}, dont des quotients locaux pour certains sous-groupes de congruence sont de type fini de codimension $l(w)$, et même géométriquement normaux et de Cohen--Macaulay, d'après Kashiwara--Shimozono \cite{KS09}. Ces résultats et constructions devraient être mis en relief, car ils sont incontournables pour qu'on parvienne à obtenir tous les faisceaux inversibles amples de la grassmannienne affine $\Gr_{\underline{\Gscr_\fbf}}$, cf. cor. \ref{corollaire groupe de picard}, s'agissant d'une lacune très sérieuse dans la preuve du théorème de normalité dans \cite{PR} en caractéristique $p=2$.

Pour la deuxième étape, on fait usage des fibrés en droites amples pour montrer que les variétés de Schubert normalisées en caractéristique $p$ sont scindées, cf. prop. \ref{proposition varietes de schubert proprietes}, donc les schémas de Schubert normalisés satisfont aux propriétés du théorème \ref{thintro2}, formant en outre un ind-schéma $\widetilde{\Gr}_{\underline{\Gscr_\fbf}}$ opéré par les groupes de lacets $L\underline{U_\pm}$ associés aux sous-groupes radiciels, cf. cor. \ref{corollaire normalises des schemas de schubert}. Alors que la normalité en caractéristique $0$ est due à Kumar \cite{KumDem}, il ne reste qu'à montrer que l'application d'espaces tangents induite par la normalisation est surjective, dès que $G=G^{\on{sc}}$. Malheureusement, l'algèbre de Lie de $L\underline{G}$ sur $\ZZ$ n'est topologiquement pas engendrée par celles associées aux sous-groupes radicielles, dès que $\Phi_G$ soit non réduit, cf. lem. \ref{lemme engendrement algebres de lie}, en vertu de quelques facteurs quadratiques méchants, ce qui n'était jamais le cas sous les hypothèses de \cite{Fal} ou de \cite{PR}. Pour remédier à ce problème, nous devons considérer les espaces de distributions, c'est-à-dire opérateurs différentiels supérieurs, cf. lem. \ref{lemme engendrement distributions}.

Ceci fournit en outre le lien entre nos schémas de Schubert et ceux de Mathieu \cite{Mat} encadrés dans la théorie de Kac--Moody, cf. \S\ref{subsection theorie de kac-moody}, dont les groupes d'arcs parahoriques sont définis par dualité en fonction de tels opérateurs différentiels supérieurs, cf. lem. \ref{lemme identification forme de tits et distributions}. Puis la comparaison des schémas de Schubert s'ensuit en appliquant les théorèmes de normalité, voir le th. \ref{theoreme comparaison kac-moody schubert}. Bien sûr, le lecteur cynique pourrait objecter à ce stade, nonobstant l'énorme travail géométrique que l'on a eu besoin de faire en avance, que les grassmanniennes affines ici construites ne sont alors que des objets bien connus depuis trente ans. 

En réponse à cela, nous objecterions que la déformation de Beilinson--Drinfeld \cite{BDr}, nous ne l'avions pas ! Par là, on entend le pré-faisceau $\on{GR}_{\Grm}$ associé à une courbe lisse affine relative $A \rightarrow B$ sur un anneau $A$ et à un $B$-groupe affine $\Grm$, qui fait correspondre à toute $B$-algèbre $C$ les classes d'isomorphisme de $\Grm$-torseurs sur $B\otimes_A C$ munis d'une trivialisation en dehors du graphe $B \otimes_A C \rightarrow C $, cf. déf. \ref{definition grassmannienne affine globale}. On a aussi des groupes de lacets $L \Grm$ et d'arcs $L^+ \Grm$ opérant sur $\on{GR}_{\Grm}$ et dont le quotient pour la topologie étale $L \Grm/L^+\Grm=\on{GR}_{\Grm}$ s'identifie à la grassmannienne affine. 

Soient alors $A=\ZZ$, $B=\ZZ[t]$ et $\Grm=\underline{\Gscr_\fbf}$ comme ci-dessus. La grassmannienne affine globale $\on{GR}_{\underline{\Gscr_\fbf}}$ est représentable par un ind-schéma projectif, voir la prop. \ref{proposition representabilite grassmannienne affine globale} et le th. \ref{theoreme grassmannienne globale projective}. Cet ind-schéma admet certains sous-schémas fermés invariants par le groupe d'arcs pour chaque copoids géométrique dominant $\mu$ de $G$, voir la déf. \ref{definition schemas de schubert globaux}, appelés les schémas globaux de Schubert $\on{GR}_{\underline{\Gscr_\fbf},\leq \mu}$ et définis sur $\ZZ[\zeta_\infty, t^{1/\infty}]$. Encore mieux, l'on a une description géométrique de ces variétés de Schubert globales.

\begin{thmintro}[th. \ref{theoreme de coherence cohomologie fibres amples}, th. \ref{theoreme normalite schema de schubert global et fibres geometriques admissibles}]\label{thintro3}
Supposons que $G^{\on{dér}}=G^{\on{sc}}$. Alors les schémas de Schubert globaux $\on{GR}_{\underline{\Gscr_\fbf},\leq \mu}$ sont normaux, plats, géométriquement réduits et leurs fibres sont égales au lieu admissible de $\Gr_{\underline{\Gscr_\fbf}\otimes k\pot{z_x}}$.
\end{thmintro}

Le lieu admissible est classiquement défini en termes de l'application de Kottwitz, mais ici, on se sert de la finitude au sens inductif de $\on{GR}_{\underline{\Tscr}}$, cf. lem. \ref{lemme grassmannienne globale tore}. Faisons aussi quelques commentaires relativement à la démonstration du théorème, dont l'affirmation reste nouvelle en ce degré de généralité même pour les groupes déployés ou tordus modérément ramifiés. L'empêchement le plus sérieux est qu'on ne sait plus sur une base bidimensionnelle si $\on{GR}_{\underline{\Gscr_\fbf},\leq \mu}$ est plat, contrairement à ce qui se passe sur les anneaux de dimension $1$, auquel cas l'assertion résulterait du théorème de cohérence de Zhu \cite{ZhuCoh}, cf. th. \ref{theoreme de coherence cohomologie fibres amples}. Toutefois, cette assertion plus faible sur un anneau de dimension $1$ permet d'entendre assez bien le normalisé de $\on{GR}_{\underline{\Gscr_\fbf},\leq \mu}$ et de montrer que l'application de normalisation est un isomorphisme, cf. th. \ref{theoreme normalite schema de schubert global et fibres geometriques admissibles}.

Enfin, venons-en bref aux applications à la théorie des modèles locaux. Ceux-ci sont étroitement liés à l'arithmétique des variétés de Shimura et plus récemment Scholze--Weinstein \cite{SchBerk} ont conjecturé qu'un certain v-faisceau de Schubert minuscule dans leur grassmannienne affine $p$-adique soit représentable par un schéma projectif, plat et géométriquement réduit. Dans une première version de cet article, nous avons généralisé les arguments de \cite{SchBerk} et de He--Pappas--Rapoport \cite{HPR}, afin de prouver la conjecture en presque tous les cas de type abélien. Vu que l'article présent est devenu trop grand et que l'encadrement perfectoïde $p$-adique ne ressemble guère aux méthodes employées, on s'est décidé de revenir sur la conjecture générale dans un futur travail. Les gens intéressés trouveront des informations supplémentaires dans \cite[IV]{LouDiss}.

\subsection{Leitfaden}
L'article est organisé à peu près comme l'introduction. Le \S\ref{rappels et preparatifs} est un rappel de divers concepts fondamentaux autour des épinglages de sous-groupes radiciels et de leurs systèmes cohérents. Au \S\ref{groupes sur Z[t]} il s'agit de construire les $\ZZ[t]$-groupes immobiliers tordus $\underline{\Gscr_\Omega}$ qui sont la base de tout ce travail. Le \S\ref{section grassmanniennes affines} concerne la géométrie des grassmanniennes affines tordues entières associées aux groupes parahoriques $\underline{\Gscr_\fbf}$, tandis que l'on étudie au \S\ref{section grassmanniennes beilinson-drinfeld} la déformation de Beilinson--Drinfeld et démontre notre variante du théorème de la cohérence. Enfin, on décrit dans l'annexe \ref{comparaison avec les varietes de demazure de kac-moody} le lien entre notre théorie et celle de Kac--Moody au sens de Mathieu. Maintenant, on va détailler le contenu de chaque sous-section.

Le \S\ref{epinglages et systemes de Chevalley-Steinberg} explique quelques notions préparatoires concernant les automorphismes de groupes épinglés et les systèmes de Chevalley--Steinberg. Le \S\ref{formes tordues quasi-deployees} concerne la construction de formes tordues quasi-déployées associées à ces données. Au \S\ref{facteurs de 2 dans le cas non reduit}, nous éclaircissons l'aspect technique qu'on appelle la modification de Tits, et dans \S\ref{derives et coproduits} on aborde une procédure d'élargissement central. Le \S\ref{subsection theorie de bt pseudo-reductive} s'agit d'une rétrospective de la théorie de Bruhat--Tits visant le cas pseudo-réductif. Au \S\ref{donnees radicielles schematiques}, nous donnons une nouvelle approche des données radicielles schématiques de Bruhat--Tits, particulièrement utile sur les bases de dimension supérieure. Dans \S\ref{subsection groupes de Tits} les $\ZZ[t^{\pm 1}]$-groupes de Tits $\underline{G}$ sont construits et puis dans \S\ref{subsection groupes immobiliers modelant tits} il en est de même de leurs $\ZZ[t]$-modèles immobiliers $\underline{\Gscr_{\Omega}}$. Nous démontrons notamment qu'ils sont lisses, affines et connexes.

Au \S\ref{grassmanniennes affines}, on commence par présenter les définitions et propriétés basiques des grassmanniennes affines. Le \S\ref{subsection cordes et fibres en droites} établit certains résultats structuraux sur les groupes d'arcs $L^+\underline{\Gscr_{\fbf}}$ parahoriques, ainsi que leurs opposés, les groupes de cordes $L^-\underline{\Gscr_{\fbf}}$, qui sont exploités pour introduire les cycles de Schubert. Dans \S\ref{normalite des varietes de schubert} nous démontrons le théorème de normalité sur les schémas de Schubert. À partir du \S\ref{grassmannienne de beilison-drinfeld et ind-projectivite} on se penche sur la déformation globale de Beilinson--Drinfeld $\on{GR}_{\underline{\Gscr_{\fbf}}}$ et l'on montre le théorème de la cohérence sur les schémas de Schubert globaux au \S\ref{normalite des varietes de schubert globales et leurs composantes irreductibles}. Le \S\ref{subsection theorie de kac-moody} sert de briève introduction aux notions de la théorie de Kac--Moody, tandis qu'au \S\ref{subsection comparaison de kac-moody} le lien entre les deux théories des groupes est établi. 

\subsection{Remerciements} Premièrement, il faut évidemment que je remercie mon directeur de thèse Prof. Dr. P.~Scholze, qui m'a suggéré d'étudier les modèles locaux des variétés de Shimura et m'a aidé constamment avec des conseils précieux. Ensuite, je tiens à remercier vivement T.~Haines et T.~Richarz, qui ont toujours montré un grand intérêt pour mon travail, qui ont partagé très généreusement leurs idées mathématiques avec moi et qui m'ont appris beaucoup, notamment pendant l'écriture de \cite{HLR}. Je veux remercier aussi le Prof. Dr. G.~Faltings en tant que deuxième rapporteur de ma dissertation \cite{LouDiss}. B.~Conrad mérite aussi une mention distinguée pour sa façon d'échanger de courriels très énergique et à la vitesse de l'éclair. C'est avec grand plaisir que je remercie J.~Anschütz, G.~Pappas, M.~Santos et X.~Zhu pour des discussions fructueuses autour de ce sujet et/ou des commentaires pertinents sur le texte. De plus, je tiens à remercier A.-C.~le Bras, qui a corrigé plusieurs fautes de français. Enfin, un remerciement spécial est adressé au rapporteur, qui n'a pas seulement fait des suggestions très pertinentes et attentives, mais qui m'a aussi forcé à renoncer à mes aspirations littéraires et ainsi à faire l'écriture de cet article correctement. Du point de vue historique, je suis immensément redevable à J.~Tits, pour la façon parfois inattendue dont son travail a imprégné tout ce que j'ai fait, et sans lequel je n'aurais jamais pu résoudre le cas non réduit. 

Ce travail repose sur d'idées obtenues pendant mon doctorat à l'Université de Bonn et la première écriture de l'article a été réalisée avec le soutien financier du SFB/TR 45 de la Deutsche Forschungsgemeinschaft et du prix Leibniz attribué au Prof. Dr. P.~Scholze. Durant la réécriture du texte, j'ai été post-doctorant à l'Imperial College de Londres et soutenu par le Conseil européen de la recherche, au moyen de la bourse de démarrage numéro d'identification 804176 attribuée à la Prof. A.~Caraiani dans le cadre du programme Horizon 2020 de l'Union européenne.

\subsection{Notations et conventions} Tous les anneaux considérés seront unitaires et commutatifs. Le foncteur de restriction des scalaires le long d'une extension plate et finie $A \rightarrow B$ d'anneaux sera noté $\Res_{B/A}$. On dit qu'un $A$-groupe $G$ est connexe si toutes ses fibres sont connexes. Dans le présent article, on ne s'occupera que des formes tordues quasi-déployées. Sauf mention contraire, les immeubles ici considérés sont les réduits et pas les élargis. Tous les ind-schémas sont stricts.

\section{Automorphismes de Dynkin et leurs formes tordues}\label{rappels et preparatifs}

\subsection{Systèmes de Chevalley--Steinberg}\label{epinglages et systemes de Chevalley-Steinberg}

Donnons-nous un schéma en groupes de Chevalley $H$ sur $\ZZ$ muni d'un épinglage, c'est-à-dire de la donnée d'un couple de Killing $T_H \subseteq B_H \subseteq H$ constitué par un tore maximal et un Borel, la base correspondante $\Delta_H$ du système de racines $\Phi_H$ de $H$ suivant $T_H$ et des isomorphismes $y_\alpha:\GG_a \rightarrow U_\alpha$ pour toute racine simple $\alpha \in \Delta_H$. La définition suivante d'une base de Chevalley est tirée directement de \cite[3.2.1, 3.2.2]{BTII}.

\begin{defn}\label{definition d'un systeme de Chevalley} On dit qu'une famille d'isomorphismes $y_\alpha:\GG_a \rightarrow U_\alpha$ pour toute racine $\alpha \in \Phi_H$ forme un système de Chevalley du groupe épinglé $H$, si
	\begin{enumerate}
		\item les $y_\alpha$ pour les racines positives simples $\alpha \in \Delta_H$ sont les épinglages donnés,
		\item l'élément $n_\alpha=y_\alpha(1)y_{-\alpha}(1)y_\alpha(1)$ normalise $T_H$ pour toute racine $\alpha\in \Phi_H$,
		\item et l'on a  $n_\alpha y_\beta(r)n_\alpha^{-1}=y_{s_\alpha(b)}(\epsilon_{\alpha,\beta} r)$, pour un certain signe $\epsilon_{\alpha,\beta} \in \{\pm 1\}$, où $\alpha, \beta \in \Phi_H$ et $s_\alpha$ désigne la réflexion dans le groupe de Weyl $W_H$ par rapport à $\alpha$.
	\end{enumerate}
\end{defn}
	
 Il faut noter qu'il n'est pas possible en général de se débarrasser de ces signes méchants $\epsilon_{\alpha,\beta}$. Par conséquent, le choix du prolongement d'épinglage dans un système de Chevalley n'est pas unique, parce qu'on se permet toujours de remplacer $y_\alpha$ et $y_{-\alpha}$ par leurs inverses, pour $\alpha \in \Phi_H \setminus \pm \Delta_H$. Notre convention pour le signe de la racine opposée $-\alpha$ diffère de celle de \cite[exp. XXIII, 1.2]{SGA3} comme le lecteur le verra dans l'exemple ci-dessous :
 
 \begin{exmp}\label{systeme de chevalley pour sl2}
 	Soit $H$ le groupe de rang un $\on{SL}_2$ épinglé de la manière habituelle : $T_H$ est le tore diagonal, $B_H$ le groupe des matrices triangulaires supérieures et $y_\alpha: \GG_a \rightarrow U_\alpha$ est donné par:
 	\begin{equation} r \mapsto \begin{pmatrix}
 	 1 & r\\
 	 0 & 1
 	\end{pmatrix}.\end{equation}
 	Alors, le seul système de Chevalley de ce groupe épinglé satisfait à la formule
 	\begin{equation} y_{-\alpha}(r)= \begin{pmatrix}
 	1 & 0\\
 	-r & 1
 	\end{pmatrix},\end{equation}
puisque des calculs fournissent $n_\alpha= n_{-\alpha}=	\begin{pmatrix}
0 & 1\\
-1 & 0
\end{pmatrix}$.  
 \end{exmp}  

Alors, \cite[exp. XXIII, prop. 6.2]{SGA3} affirme qu'il existe toujours un système de Chevalley pour un tel groupe épinglé $H$ :

\begin{propn}[Chevalley]
	Il y a un et un seul système de Chevalley au signe près de $H$ prolongeant l'épinglage donné.
\end{propn}

\begin{proof}
	Écrivons une racine arbitraire $\alpha \in \Phi_H$ comme la transformée $w(\beta)$ d'une racine positive simple $\beta \in \Delta_H$. L'élément $w$ admet une écriture $s_1\dots s_k$ comme produit de réflexions $s_i$ associées aux racines simples. On a ainsi \begin{equation}y_\alpha(r)=\on{int}(n_{\alpha_1}\dots n_{\alpha_k})(y_\beta( \pm r)),\end{equation} ce qui nous permet aussi de définir une famille d'isomorphismes $y_\alpha$ en prenant un signe quelconque. On n'a qu'à vérifier que telle construction, malgré ses choix arbitraires, préserve l'épinglage donné lorsque appliquée à une racine simple. Ceci est le contenu du \cite[exp. XXIII, lem. 6.3]{SGA3}, lequel repose sur des calculs en rang $2$.
\end{proof}

Considérons maintenant le groupe $\Xi_H$ d'automorphismes de la donnée radicielle épinglée correspondante à $\Delta_H \subseteq \Phi_H \subseteq X^*(T_H)$ et identifions-le au groupe des automorphismes de $H$ fixant l'épinglage donné, d'après le \cite[exp. XXIV, th. 1.3]{SGA3}. On peut associer à ces données un nouveau système de racines $\Phi_G $ dans $X^*(T_H)^{\Xi_H}$ formé par les moyennes
\begin{equation} a=\lvert \Xi_H \rvert^{-1}\sum_{\sigma \in \Xi_H} \sigma(\alpha)\end{equation}
des orbites finies de $\Phi_H$ pour l'action de $\Xi_H$ (ceci conserve son sens même si $\Xi_H$ est infini), cf. \cite[lem. 4.2]{Hai15}. Ce système de racines vient muni d'une application évidente de restriction $\Phi_H \rightarrow \Phi_G$. On remarque que la définition suivante d'un système de Chevalley--Steinberg coïncide avec celle de \cite[4.1.3]{BTII} et de \cite[déf. 4.2]{LndvCpc}, sauf que les auteurs y travaillent avec les formes tordues quasi-déployées qu'on étudiera au \S\ref{formes tordues quasi-deployees}.

\begin{defn}\label{systeme de chevalley-steinberg}
	Un système de Chevalley s'appelle de Chevalley--Steinberg si, pour tout automorphisme $\sigma \in \Xi_H$, l'on a $\sigma (y_\alpha(r))= y_{\sigma(\alpha)}(\epsilon_{\alpha,\sigma} r)$ où $\epsilon_{\alpha,\sigma} \in \{\pm 1\}$ et est même égal à $1$ si la moyenne $a \in \Phi_{G}^{\on{nd}}$ n'est pas divisible.
\end{defn} 

En particulier, une condition suffisante pour que le signe $\epsilon_{\alpha,\sigma}$ soit toujours égal à $1$ est que $\Phi_G$ soit réduit, c'est-à-dire que $H^{\on{sc}}$ ne contienne aucun facteur isomorphe à $\on{SL}_{2n+1}$. L'exemple suivant montre que cette condition est aussi nécessaire :

\begin{exmp}\label{exemple chevalley-steinberg pour sl3}
	Soit $H$ le groupe $\on{SL}_3$ épinglé de la façon canonique, c'est-à-dire $T_H$ est le tore diagonal, $B_H$ le sous-groupes des matrices triangulaires unipotentes et les sous-groupes radiciels associés aux racines positives simples $\alpha$ et $\beta$ sont donnés par \begin{equation}y_\alpha: r\mapsto \begin{pmatrix}
	1& r & 0\\
	0 & 1 & 0\\
	0 & 0 & 1
	\end{pmatrix}, y_\beta: r\mapsto \begin{pmatrix}
	1& 0 & 0\\
	0 & 1 & r\\
	0 & 0 & 1
	\end{pmatrix}.\end{equation} D'après \cite[ex. 7.1.5]{Co14}, on sait que le seul automorphisme $\sigma$ du groupe épinglé $H$ est :
	\begin{equation} \begin{pmatrix}
	a& b & c\\
	d & e & f\\
	g & h & i
	\end{pmatrix}\mapsto \begin{pmatrix}
	ei-fh & af-cd & bf-ce\\
	ah-bg & ai-cg & bi-ch\\
	dh-eg & di-fg & ae-bd
	\end{pmatrix},\end{equation}
	ce qui entraîne $\epsilon_{\alpha+\beta,\sigma}=-1$ comme voulu.
	Par alternative, nous pourrions avoir calculé directement l'effet de $\sigma$ sur $\on{Lie}H$, utilisant le crochet.
\end{exmp}

Ce sont justement ces signes méchants, qui nous forceront à traiter de façon spéciale les groupes unitaires de dimension impaire, voir l'ex. \ref{exemple su3} et surtout le \S\ref{facteurs de 2 dans le cas non reduit}. Finissons-en par énoncer l'existence des systèmes de Chevalley--Steinberg, dont la démonstration remonte essentiellement à l'\oe{}uvre \cite{St1} de Steinberg.

\begin{propn}[Steinberg]
	Il existe toujours un système de Chevalley--Steinberg.
\end{propn}

\begin{proof}
Vu que $\sigma$ est un automorphisme et que les $\ZZ$-automorphismes de $\GG_a$ sont déterminés par les signes $\{\pm 1\}$, on déduit que $\sigma (y_\alpha(r))= y_{\sigma(\alpha)}(\epsilon_{\alpha,\sigma} r)$ avec $\epsilon_{\alpha,\sigma} \in \{\pm 1\}$. Fixé un système de Chevalley quelconque du groupe épinglé $H$, on doit montrer que, quitte à faire des changements de signes, on arrive à la condition $\epsilon_{\alpha,\sigma}=1$ pour toutes les racines $\alpha \in \Phi_H$ telles que leurs moyennes $a\in \Phi_G^{\on{nd}}$ soient non divisibles.

On se ramène aisément au cas où $H$ est adjoint simple. Le cas des involutions fut traité dans \cite[lem. 3.2]{St1} et le cas du groupe symétrique $\Sf_3$ en trois éléments fut vérifié dans la preuve de \cite[prop. 4.4]{LndvCpc}. 
\end{proof}

\subsection{Formes tordues et quasi-systèmes de Chevalley} \label{formes tordues quasi-deployees} Soient $k$ un corps et $G$ un $k$-groupe réductif et connexe. Rappelons que $G$ s'appelle quasi-déployé si le centralisateur de tout sous-$k$-tore déployé maximal $S$ est un tore maximal $T$. En particulier, on peut toujours choisir un sous-groupe de Borel $B\supset T\supset S$ correspondant à une base $\Delta_G$ du système de racines $\Phi_G$ de $G$ relativement à $S$. Soit $l/k$ une extension galoisienne à groupe $\Gamma$ déployant $G$. Le lemme suivant est du folklore :

\begin{lem}\label{lemme forme deployee epinglage}
	Soit $H$ la $k$-forme déployée de $G$ munie d'un épinglage. Alors, quitte à conjuguer par $H^{\emph{ad}}(l)$, il y a un et un seul prolongement de l'isomorphisme donné $G_l\simeq H_l$ à un isomorphisme de triples $(G,B,T)\otimes_k l \simeq (H,B_H,T_H)\otimes_k l$ tel que l'action de $\Gamma$ transportée au membre de droite préserve l'épinglage.
\end{lem}

 Il en résulte un certain homomorphisme $\tau: \Gamma \rightarrow \Xi_H$ vers le groupe d'automorphismes épinglés, tel que l'opération de $\Gamma$ sur $G_l$ transportée à $H_l$ devienne égale à $\gamma \mapsto \tau(\gamma)\otimes \gamma$.

\begin{proof}
	L'unicité est évidente en vertu du \cite[exp. XXIV, th. 1.3]{SGA3}. Grâce aux théorèmes de conjugaison pour les tores maximaux et les paraboliques minimaux qui les contiennent, cf. par exemple \cite[th. C.2.3, th. C.2.5]{CGP}, on peut choisir $h \in H(l)$, tel que $\varphi: G_l \simeq H_l$ envoie $(B,T)$ sur $(B_H,T_H)$. Pour préserver l'épinglage, il suffit de conjuguer par un élément de $T^{\on{ad}}(l)$.
\end{proof}

Par conséquent, tous les $k$-formes $G$ quasi-déployées et $l$-déployées du $k$-groupe épinglé $H$ s'obtiennent par une procédure très explicite de tordument comme suit : 
\begin{cons}\label{construction groupe quasi deploye}
Soient $\Gamma$ le groupe de Galois de $l/k$ et $\tau: \Gamma \rightarrow \Xi_H$ un homomorphisme de groupes continu. Comme $\Gamma$ commute aux $k$-automorphismes de $H$, l'homomorphisme $\tau:\Gamma \rightarrow \Xi_H$ induit une classe $c_\tau \in H^1(k,\on{Aut}_{l}(H))$ et l'on note $G$ le $k$-groupe réductif, connexe et quasi-déployé obtenu en tordant $H$ par $c_\tau$, comp. avec \cite[lem. 7.1.1, prop. 7.1.6]{Co14}. Plus explicitement, si $\Gamma$ est fini, considérons la restriction des scalaires $\Res_{l/k} H$ munie de l'opération de $\Gamma$ qui fait correspondre à $\gamma$ l'automorphisme $\tau(\gamma) \otimes \gamma$. Si l'on pose alors $G=(\Res_{l/k} H)^{\Gamma}$, on obtient la $k$-forme quasi-déployée cherchée de $H$.
\end{cons}

La définition classique d'un quasi-épinglage d'un groupe quasi-déployé $G$ consiste de la donnée d'un couple de Killing $B\supset T$ et d'un élément $X \in \on{Lie} B$, tel que son image dans $\on{Lie}B_H$ coïncide avec la somme des vecteurs générateurs $Y_\alpha=y_\alpha(1) \in \on{Lie}U_\alpha$ pour chaque racine simple $\alpha \in \Delta_H$, cf. \cite[exp. XXIV, 3.9]{SGA3}. Bien que cela est raisonnable, nous aurons besoin de recourir plutôt à d'isomorphismes décrivant les sous-groupes radiciels $U_a$ pour les racines relatives $a \in \Phi_G$, où ce symbole-ci désigne le système de racines relatif\footnote{Il n'y a aucun conflit entre cette notation et celle utilisée dans \S\ref{epinglages et systemes de Chevalley-Steinberg}, grâce au résultat de Haines qu'on avait déjà cité, cf. \cite[lem. 4.2]{Hai15}.} de $G$ par rapport à $S$. 	
    
  Ensuite, on étudie la structure des sous-groupes radiciels $U_a$ lorsque $a \in \Phi_G^{\on{nd}}$ n'est pas divisible.
  
  \begin{exmp}[non-multipliable]\label{exemple sl2}
  	Supposons que $H=\prod_{I} \on{SL}_2$, épinglé comme d'habitude et où $I$ est un ensemble fini sur lequel $\Gamma$ agit de façon transitive. Alors, le groupe $\Gamma$ y opère canoniquement par d'automorphismes de Dynkin qui échangent les différents facteurs simples de rang $1$, selon l'action naturelle de $\Gamma$ sur l'espace homogène $I$. On veut calculer le groupe $G=(\Res_{l/k} H)^\Gamma$. Mais si $(h_{\gamma})\in \on{SL}_2(R\otimes_kl)^{I}$ est invariant par $\Gamma$, alors $\sigma(h_{i})=h_{\sigma i}$. Par conséquent, on arrive à la conclusion que $G \simeq \Res_{m_0/k}  \on{SL}_2$, où $m_0$ désigne le corps fixe du stabilisateurs $ \Gamma_0 \subset \Gamma$ de $i_0 \in I$, et que l'inclusion dans $\Res_{l/k} H$ est donnée par la diagonale tordue $(\gamma)_{\gamma \in \Gamma/\Gamma_0}$.
  \end{exmp}

Revenons alors au cas général d'une racine relative ni divisible ni multipliable $a \in \Phi_G^{\on{nd,nm}}$ obtenue en moyennant $\alpha\in \Phi_H$. En termes d'un système de Chevalley--Steinberg de notre groupe épinglé $H$, on obtient un isomorphisme
\begin{equation}x_a: \Res_{l_a/k}\GG_a \rightarrow  U_a, \end{equation}\begin{equation} r\mapsto \prod_{\gamma\alpha} y_{\gamma\alpha}(\gamma(r)), \end{equation}
où l'on note $l_a$ le corps fixe du stabilisateur de $\alpha$ pour l'action de $\Gamma$, car $\epsilon_{\alpha,\gamma}=1$ pour tout $\gamma \in \Gamma$, comp. aussi avec \cite[4.1.5]{BTII}.
Si l'on avait choisi soit un autre système de Chevalley, soit une racine conjuguée de $\alpha$, on n'aurait changé $x_a$ qu'au signe et à la conjugaison près, cf. aussi \cite[4.1.7]{BTII}. On va dire que deux tels isomorphismes sont similaires et donc que $x_a$ ne dépend que de $a$ à similitude près.

\begin{exmp}[multipliable]\label{exemple su3}
	Considérons maintenant $H=\on{SL}_3$ muni de l'épinglage usuel et supposons que $l/k$ soit une extension quadratique séparable. Dans l'exemple \ref{exemple chevalley-steinberg pour sl3}, nous avons déterminé l'involution de Dynkin $\sigma$ de $\on{SL}_3$. Le groupe échéant $G=(\Res_{l/k}H)^\Gamma$ sera toujours noté $\on{SU}_{3,l/k}$ et regardé en tant que fermé de $\Res_{l/k}\on{SL}_3$ ; il s'agit même du groupe unitaire attaché à la forme quadratique $xy+yx-z^2$. Son système de racines est $\Phi_G=\{\pm a, \pm 2a\}$ et on souhaiterait d'avoir une description explicite du groupe $U_a$. Grâce à l'exemple sus-mentionné, on tire
	\begin{equation}\label{equation matrices unipotentes su3} U_a =\Bigg\{ \begin{pmatrix}
	1& r & s\\
	0 & 1 & \sigma(r)\\
	0 & 0 & 1  \end{pmatrix} : s+\sigma(s)=r\sigma(r) \Bigg\}. \end{equation}
	Notant $x_a(r,s)$ la matrice ci-dessus, on peut regarder $x_a$ comme un isomorphisme entre un certain groupe matriciel et le groupe algébrique abstrait $U_a$. En termes du système de Chevalley--Steinberg de $\on{SL}_3$, il s'écrit de la façon suivante :
	\begin{equation}x_a(r,s)=y_{\alpha}(r)y_{\alpha+\beta}(-\sigma(s))y_\beta(\sigma(r))=y_\beta(\sigma(r))y_{\alpha+\beta}(s)y_\alpha(r).\end{equation}
\end{exmp}

Cet exemple nous amène à introduire de la notation pour ce groupe matriciel unipotent, cf. aussi \cite[4.1.9]{BTII}
\begin{defn}\label{definition groupe pluriel}
	Soit $l/k$ une extension quadratique séparable. Le groupe unipotent pluriel $\GG_{p,l/k}$ est le sous-groupe des matrices triangulaires supérieures unipotentes de $\Res_{l/k}\on{SL}_3$ de la forme de (\ref{equation matrices unipotentes su3}).
\end{defn}
  
Soient $a \in \Phi_G^{\on{m}}$ une racine relative multipliable et $\alpha,\beta\in \Phi_H$ deux racines conjuguées dont la somme est une racine et dont les moyennes sont égales à $a$. En termes d'un système de Chevalley--Steinberg commode de notre groupe épinglé $H$, on obtient un isomorphisme
\begin{equation}
x_a: \Res_{l_{2a}/k}\GG_{p,l_a/l_{2a}} \rightarrow  U_a, \end{equation} \begin{equation} (r,s)\mapsto \prod_{\{\gamma\alpha,\gamma\beta\}} y_{\gamma\beta}(\gamma\sigma(r))y_{\gamma(\alpha+\beta)}(\gamma(s))y_{\gamma\alpha}(\gamma(r)), 
\end{equation}
où l'on note $l_a$ le corps fixe du stabilisateur de $\alpha$ pour l'action de $\Gamma$ et $l_{2a}$ celui de $\alpha +\beta$, et $\sigma$ l'involution de $l_a/l_{2a}$, voir \cite[4.1.9]{BTII}. Notons que les ensembles $\{\gamma\alpha,\gamma\beta\}$ doivent être regardés comme étant ordonnés et $\gamma$ doit préserver cet ordre, pour que cette écriture ait du sens, donc on dira que $\gamma \alpha$ resp. $\gamma\beta$ est un relèvement de premier type resp. second type de $a$. Finalement, l'isomorphisme $x_a$ ne dépend que de $a$ à similitude près, cf. \cite[4.1.13]{BTII} et voir \cite[A.4]{BTII} pour les opérations permises dans ce cas.

L'ensemble des isomorphismes $x_a$ pour $a \in \Delta_G$ resp. $a \in \Phi_G$ ainsi obtenus est exactement ce que nous aurions envie d'appeler un quasi-épinglage resp. quasi-système de Chevalley. La seule nuisance est qu'une telle notion nous semble à l'instant extrinsèque et dépendante de la réalisation de $G$ comme $\Gamma$-invariants d'une restriction des scalaires de $G$. L'approche idéale pour s'y référer intrinsèquement est celle de \cite[2.4]{LouBT}, dont le truc principal est de considérer plutôt les homomorphismes naturels \begin{equation}\zeta_a:\Res_{l_a/k} \on{SL}_2 \rightarrow G, \end{equation} \begin{equation}
\zeta_a:\Res_{l_{2a}/k} \on{SU}_{3,l_a/l_{2a}} \rightarrow G ,\end{equation} prolongeant les isomorphismes $x_a$, selon $a$ appartienne à $\Phi_G^{\on{nd,nm}}$ ou à $\Phi_G^{\on{m}}$. En effet, leur existence découle de la description des groupes quasi-déployés de rang $1$ conformément aux ex. \ref{exemple sl2} et \ref{exemple su3}. Voici la \cite[déf. 2.4]{LouBT} :

\begin{defn}\label{definition quasi-systemes de chevalley}
	Soit $G \supset B \supset T \supset S$ un groupe quasi-déployé. Un quasi-épinglage de $G$ est une collection d'homomorphismes $\zeta_a$ de $\Res_{l_a/k} \on{SL}_2$ vers $G$ si $a \in \Delta_G^{\on{nd,nm}}$ (resp. de $\Res_{l_{2a}/k} \on{SU}_{3,l_a/l_{2a}}$ si $a \in \Delta_G^{\on{m}}$) tels que leurs restrictions aux tores diagonaux déployés soient égales à $a^{\vee}:\GG_m \rightarrow G$ et que les matrices triangulaires supérieures unipotentes s'envoient sur $U_a$. Un quasi-système de Chevalley de $G$ quasi-épinglé est une extension de la collection $\zeta_a$ aux racines non divisibles $a \in\Phi_G^{\on{nd}}$ de telle sorte que $\on{int}(m_a) \circ \zeta_b$ soit semblable à $\zeta_{s_a(b)}$.
\end{defn}

 Les $m_a$ s'écrivent comme $\prod_{\gamma \alpha} n_{\gamma \alpha}$ dans le cas non multipliable (resp. $\prod_{\gamma (\alpha+\beta)} n_{\gamma (\alpha+\beta)}$ dans le cas multipliable), voir la déf. \ref{definition d'un systeme de Chevalley} pour le cas déployé. Ceux-ci se généralisent\footnote{En fait, lorsque $a$ est multipliable, l'élément $m_a$ ne s'écrit comme $m_a(r,s)$ pour un certain couple $(r,s)$ que si $1$ est une norme de $l$. Mais cela ne change presque rien.} en des éléments \begin{equation}
 m_a(r)=x_{a}(r^{-1})x_{-a}(r)x_{a}(r^{-1})
 \end{equation} pour toute racine $a\in \Phi_G^{\on{nd},\on{nm}}$ ni divisible ni multipliable, et \begin{equation}m_a(r,s)= x_a(rs^{-1},\sigma(s)^{-1})x_{-a}(r,s)x_a(r\sigma(s)^{-1},\sigma(s)^{-1})\end{equation} pour toute racine multipliable $a \in \Phi_G^{\on{m}}$, voir \cite[4.1.5, 4.1.11]{BTII}.
 
 \begin{propn}\label{proposition existence et conjugaison des quasi systemes de chevalley}
 	Les classes de similitude des quasi-systèmes de Chevalley de $G$ forment un torseur trivial sous $T^{\on{ad}}(k)/T^{\on{ad}}[2](k)$.
 \end{propn}

\begin{proof}
	On renvoie à \cite[prop. 2.5]{LouBT} pour plus de détails. Ce que deux classes de similitude de quasi-systèmes de Chevalley ne diffèrent qu'en conjuguant par $T^{\on{ad}}(k)/T^{\on{ad}}[2](k)$ est presque évidente. La méthode de preuve dans le cas déployé permet d'étendre un quasi-épinglage donné à une famille d'homomorphismes $\zeta_a$. Pour vérifier qu'ils satisfont à la contrainte de similitude, on utilise les formules explicites pour les $x_a$, $m_a$, etc. en termes des $y_\alpha$, $n_\alpha$, etc. pour se ramener au cas déployé.
\end{proof}

\begin{rem}
	Outre les résultats de conjugaison, qui ne marchent qu'au-dessus des corps, le genre de constructions entreprises dans ce numéro, voir par exemple la constr. \ref{construction groupe quasi deploye}, s'étendent facilement aux bases générales. On en laisse le soin au lecteur.
\end{rem}
\subsection{La modification de Tits}\label{facteurs de 2 dans le cas non reduit}

Malheureusement, le prolongement naturel du groupe $\GG_{p,l/k}$, voir la déf. \ref{definition groupe pluriel}, au cas où les coefficients sont entiers (ceci doit être pris au sens du \S\ref{groupes sur Z[t]}) n'est plus lisse en caractéristique résiduelle égale à $2$. Par conséquent, on va introduire un certain $k$-groupe unipotent $\GG^{\on{T}}_{p,l/k}$ explicite à la suite de \cite[annexe 2]{TitsOber}, qui s'identifie encore, lorsque $2$ est inversible dans $k$, au groupe unipotent pluriel de la déf. \ref{definition groupe pluriel} noté $\GG^{\on{CS}}_{p,l/k}$ pendant tout ce numéro. L'avantage de la variante $\GG^{\on{T}}_{p,l/k}$ de Tits est qu'elle admet une description plus commode en termes du sous-espace vectoriel $l^0 \subset l$ des éléments de trace nulle, qui s'étendra sans peine de manière lisse au cas où les coefficients sont entiers, voir la déf. \ref{definition drs groupe de tits}.

\begin{defn}\label{definition pluriel tits}
	Soit $k$ un corps de caractéristique différente de $2$ et $l/k$ une extension quadratique. Alors, $\GG^{\on{T}}_{p,l/k}$ est la variété $\Res_{l/k}\AAA^1_l \times \AAA^1_{l_0}$ munie de la loi de groupe
	\begin{equation}\label{equation loi de groupe pluriel} (u,v)\cdot (u',v')=(u+u',v+v'+\sigma(u)u'-u\sigma(u')) .\end{equation}
	Ici, l'on regarde $\AAA^1_{l_0}$ comme la droite affine sur $k$ au moyen d'un isomorphisme quelconque de $k$-espaces vectoriels $k \simeq l_0$. 
\end{defn}

Notons que cette définition se ressemble à celle de \cite[4.1.15]{BTII}, sauf en ce qui concerne le choix d'un élément $\lambda$ de trace $1$, bien que nous avons choisi $\lambda=1$. Nous en tirons l'isomorphisme
\begin{equation}
\GG_{p,l/k}^{\on{T}} \rightarrow \GG_{p,l/k}^{\on{CS}},
\end{equation} 
\begin{equation} 
(u,v)\mapsto (u,2v-u\sigma(u)),\end{equation}
\begin{equation}
(r,r/2+s\sigma(s)/2) \mapsfrom (r,s), 
\end{equation}
lequel profite du fait que $2$ est inversible dans $k$. Avant de continuer, étudions la situation limite en caractéristique $2$, qui est implicitement liée avec toutes les modification que nous sommes en train d'introduire :

\begin{rem}\label{remarque limite inseparable pluriel}
	Supposons pour l'instant que $l/k$ est une extension quadratique inséparable de caractéristique $2$, c'est-à-dire $\sigma$ est trivial. Alors la déf. \ref{definition groupe pluriel} préserve son sens et l'on obtient un isomorphisme $\GG^{\on{CS}}_{p,l/k}=\Res_{l/k}\alpha_2 \times \Res_{l/k}\GG_a $, donc le groupe commutatif dont il s'agit n'est pas lisse. Par contre, si l'on remplace $l^0$ dans la déf. \ref{definition pluriel tits} par un choix arbitraire de complément linéaire de $k$ dans $l$, alors on obtient que $\GG^{\on{T}}_{p,l/k}$ s'identifie au groupe commutatif lisse $\Res_{l/k}\GG_a \times \GG_a$.
\end{rem}

Moyennant l'isomorphisme donné $\GG^{\on{T}}_{p,l/k} \simeq \GG^{\on{CS}}_{p,l/k}$, on peut donc considérer les isomorphismes du quasi-système de Chevalley de $G$ introduit précédemment, voir la déf. \ref{definition quasi-systemes de chevalley}, comme partant du membre de gauche. On obtient alors l'écriture
\begin{equation}x_a(u,v)=y_\beta(\sigma(u))y_{\alpha+\beta}(v/2+N(u)/2)y_\alpha(u) \end{equation}
dans le cas où $\Phi_H$ est irréductible (en général, il faut considérer un produit à plusieurs termes). Cet état de choses est indésirable en vertu du \S\ref{groupes sur Z[t]}, puisque le morphisme vers $H$, pour qu'il soit prolongeable aux coefficients entiers, cf. prop. \ref{proposition plongement restrictions des scalaires groupe de tits}, ne devrait pas impliquer des divisions par $2$. On souhaiterait donc modifier les isomorphismes du système de Chevalley--Steinberg de $H$, de manière que $x_a$ s'exprime en fonction de $u$, $v$ et de leurs conjugués à coefficients entiers. Ceci sera achevé au lem. \ref{lemme coordonnees de la modification de Tits}.

\begin{exmp}\label{exemple modification de tits rang un}
	Reprenons les exemples \ref{exemple chevalley-steinberg pour sl3} et \ref{exemple su3}, alors $H=\on{SL}_3$ et $G=\on{SU}_{3,l/k}$. D'après la prop. \ref{proposition existence et conjugaison des quasi systemes de chevalley}, on se permet de conjuguer le système de Chevalley de $H$ par un élément de $T_H^{\on{ad}}(k)$. Choisissons d'abord $\varpi_{\alpha}^\vee(2)$. En notant le quasi-système de Chevalley et le système de Chevalley--Steinberg par $x_{\bullet}^{\on{CS}}$ et $y_{\bullet}^{\on{CS}}$, posons alors $y_{\bullet}^{\on{T}}=\on{int}(\varpi_{\alpha}^\vee(2)) \circ y_{\bullet}^{\on{CS}}$. On en tire les relations suivantes :
	\begin{equation}x_a^{\on{CS}}(u,v)=y_{\alpha}^{\on{T}}(u/2)y_{\alpha+\beta}^{\on{T}}(-\sigma(v)/4-N(u)/4)y_\beta^{\on{T}}(\sigma(u)) \end{equation}
	\begin{equation}
	x_{-a}^{\on{CS}}(u,v)=y_{-\alpha}^{\on{T}}(2u)y_{-\alpha-\beta}^{\on{T}}(-\sigma(v)-N(u))y_{-\beta}^{\on{T}}(\sigma(u)).
	\end{equation}
	En particulier, il suffit de prendre $x_a^{\on{T}}(u,v)=x_a^{\on{CS}}(2u,4v)$ et $x_{-a}^{\on{T}}(u,v)=x_{-a}^{\on{CS}}(u,v)$, pour que la contrainte d'intégralité soit satisfaite.
\end{exmp}

Cependant, il faut typiquement modifier aussi les isomorphismes associés aux racines ni divisibles ni multipliables. La définition suivante généralise de façon naturelle la modification de Tits dans le cas où $H$ est simplement connexe et simple, voir \cite[pp. 216-217]{TitsOber} :

\begin{defn}[Tits]\label{definition quasi-systeme de Tits}
	Soient $c$ la somme des racines simples multipliables de $G$ et $\gamma$ la somme de leurs relèvements de premier type. Alors, on pose
	\begin{equation}
	y_\alpha^{\text{T}}(r)=y_a^{\text{CS}}(2^{\langle \varpi_{\gamma}^{\vee}, \alpha\rangle}r), \end{equation} \begin{equation}
	x_a^{\text{T}}(r)=x_a^{\text{CS}}(2^{\langle \varpi_{2c}^{\vee}, a\rangle}r),\end{equation} \begin{equation}
	x_a^{\text{T}}(u,v)=x_a^{\text{CS}}(2^{\chi_{+}(a)}u,4^{\chi_{+}(a)}v),
	\end{equation}
	et l'on dit que les quasi-systèmes correspondants sont de Tits. Ici, $\chi_{+}$ désigne la fonction caractéristique de $\Phi_G^+$ dans $\Phi_G$.
\end{defn}

Nous sommes amené à définir les \begin{equation}\label{equation elements m du normalisateur singulier tits}m_a^{\on{T}}(r)=x^{\on{T}}_{a}(r^{-1})x^{\on{T}}_{-a}(r)x^{\on{T}}_{a}(r^{-1})\end{equation} pour toute racine $a\in \Phi_G^{\on{nd},\on{nm}}$ ni divisible ni multipliable, et \begin{equation}\label{equation elements m du normalisateur pluriel tits}
m_a^{\on{T}}(u,v)= x^{\on{T}}_{a}\Big(\frac{u}{s(u,v)},\frac{1}{s(u,\sigma(v))}\Big)x^{\on{T}}_{-a}(u,v)x^{\on{T}}_{a}\Big(\frac{u}{s(u,\sigma(v))},\frac{1}{s(u,\sigma(v))}\Big)
\end{equation} pour toute racine multipliable $a \in \Phi_G^{\on{m}}$, où l'on pose $s(u,v)=v+u\sigma(u)$. Nous posons aussi $m_a^{\on{T}}=m_a^{\on{T}}(1)$ si $a$ n'est ni divisible ni multipliable, et $m_a^{\on{T}}=m_a^{\on{T}}(1,0)$ si $a$ est multipliable. 

Considérons l'action de $\Gamma$ par automorphismes de Dynkin sur le groupe épinglé $(H,y_\alpha^{\on{CS}})$ selon le système de Chevalley-Steinberg originel et transportons-la au groupe épinglé $(H,y_\alpha^{\on{T}})$. Il est important de remarquer que ces automorphismes ne préservent pas le système de Tits non plus, car ils introduisent des facteurs de $2$ supplémentaires. Plus tard, il faudra en faire usage et ces automorphismes seront dénommés désormais de Dynkin--Tits.

Enregistrons la propriété la plus significative de cette modification.
\begin{lem}\label{lemme coordonnees de la modification de Tits}
	Les coordonnées des $x_a^{\on{T}}$ par rapport aux $y_\alpha^{\on{T}}$ sont des polynômes entiers dans les variables $\gamma(u)$, $\gamma(v)$ et les $m_b^{\on{T}}$ permutent les $x_a^{\on{T}}$ à similitude près.
\end{lem}

\begin{proof}
	On laisse le soin au lecteur de vérifier que les facteurs de $2$ concernés s'annulent, puisqu'il serait fastidieux de reproduire tous les calculs, comp. aussi avec l'ex. \ref{exemple modification de tits rang un}.
\end{proof}

Sauf mention contraire, on travaillera toujours dans la suite avec de quasi-systèmes de Tits et pas de Chevalley--Steinberg, donc on se permettra d'omettre l'exposant $\on{T}$, lorsqu'aucune confusion n'est à craindre, cf. \S\ref{subsection groupes de Tits} en particulier.
\subsection{Des extensions centrales} \label{derives et coproduits}

Dans cet article, on travaillera souvent en même temps avec une classe de groupes tordus à groupe adjoint donné et l'une des tâches les plus courantes consistera à ramener le cas général au simplement connexe dès que possible, voir les preuves de la prop. \ref{proposition fibre generique quasi-reductive} et du th. \ref{theoreme normalite des schemas de Schubert}. Pour cette raison, nous décrivons maintenant une manière très sympathique, à la suite de \cite[\S1.4]{CGP}, de modifier un $\ZZ$-groupe déployé et épinglé $H$ en son tore maximal $T_H$.

Soit $\widetilde{T}_H$ un $\ZZ$-tore déployé muni des morphismes $T_H \rightarrow \widetilde{T}_H \rightarrow T_H^{\on{ad}}$, dont le composé est l'application naturelle. Comme $T_H^{\on{ad}}$ s'identifie canoniquement au groupe des automorphismes de $H$ laissant $T_H$ stable point par point, voir par exemple \cite[th. 1.3.9, th. A.4.6]{CGP}, ceci équivaut à se donner une opération de $\widetilde{T}_H$ sur $H$ ainsi qu'un morphisme équivariant $T_H \rightarrow \widetilde{T}_H$. En particulier, l'application anti-diagonale
\begin{equation}\label{equation plongement anti-diagonal}T_H \rightarrow H \rtimes \widetilde{T}_H, 
\end{equation}
donnée par $t \mapsto (t^{-1}, t)$, identifie $T_H$ à un sous-groupe fermé central de $H \rtimes \widetilde{T}_H$.

\begin{propn}\label{proposition representabilite modifie central deploye}
	Le faisceau en groupes quotient $(H \rtimes \widetilde{T}_H)/T_H$ de l'application (\ref{equation plongement anti-diagonal}) est représentable par un $\ZZ$-groupe déployé $\widetilde{H}$ à tore maximal $\widetilde{T}_H$ et à groupe adjoint $H^{\emph{ad}}$.
\end{propn} 

\begin{proof}
	La représentabilité d'un quotient plat sur une base régulière de dimension $1$ est un résultat d'Anantharaman, cf. \cite[\S4]{Anan}. Les propriétés restantes se vérifient au-dessus de chaque corps algébriquement clos, voir \cite[prop. 1.4.3]{CGP}.	
\end{proof}

 En particulier, on en tire une suite exacte courte \begin{equation}\label{equation z-extension canonique deploye}1 \rightarrow T_H^{\on{sc}} \rightarrow H^{\on{sc}}\rtimes T_H \rightarrow H \rightarrow 1, \end{equation}
 car le tore maximal $T_H$ opère sur $H$ en préservant le groupe dérivé, d'où le relèvement de l'opération au revêtement simplement connexe $H^{\on{sc}}$. Vu que $T_H^{\on{sc}}$ est induit et que le groupe dérivé de $H^{\on{sc}}\rtimes T_H$ est simplement connexe, cette extension centrale de $H$ est habituellement appelée une $z$-extension, cf. \cite[prop. 3.1]{MS82}. Puisque sa formation ne dépend que du choix du tore maximal $T_H$, nous parlerons souvent de la $z$-extension canonique du groupe épinglé $H$.

Traitons enfin le cas des formes tordues quasi-déployées comme dans \S\ref{formes tordues quasi-deployees}. Soient alors $l/k$ une extension galoisienne à groupe de Galois $\Gamma$, $H$ un $k$-groupe épinglé muni d'une action de $\Gamma$ par automorphismes de Dynkin (ou bien de Dynkin--Tits en dehors de la caractéristique $2$, cf. \S\ref{facteurs de 2 dans le cas non reduit}) et $G$ le groupe obtenu par torsion comme dans la constr. \ref{construction groupe quasi deploye}. En se donnant de nouveau un $k$-tore déployé $\widetilde{T}_H$ opéré par $\Gamma$ et des $k$-morphismes $T_H \rightarrow \widetilde{T}_H\rightarrow T_H^{\on{ad}}$ invariants par $\Gamma$, on arrive aux applications de tores tordus $T \rightarrow \widetilde{T} \rightarrow T^{\on{ad}}$, ce qu'on utilise pour définir le modifié central
\begin{equation}\label{equation modifie central tordu} \widetilde{G}=(G\rtimes \widetilde{T})/T\end{equation}
de $G$, lequel est réductif connexe à tore maximal égal à $T$, cf. prop. \ref{proposition representabilite modifie central deploye}. D'autre part, l'action de $\Gamma$ sur $H$ se prolonge naturellement en une action sur la modification centrale $\widetilde{H}$.

\begin{lem}
	Le groupe $\widetilde{G}$ s'identifie canoniquement à la forme tordue quasi-déployée associée à $\widetilde{H}$ muni de l'action naturelle de $\Gamma$.
\end{lem}

Il résulte de même la $z$-extension canonique, comp. avec (\ref{equation z-extension canonique deploye}) :
\begin{equation}\label{equation z-extension canonique tordu} 1\rightarrow T^{\on{sc}}\rightarrow  G^{\on{sc}}\rtimes T \rightarrow G \rightarrow 1,\end{equation}
dont l'avantage principal n'est pas seulement sa naturalité, mais aussi que plusieurs propriétés convenables de $G$ soient préservées, comme nous verrons plus tard, cf. par exemple l'hyp. \ref{hypothese groupe dans la fibre generique}.

\section{Familles entières de groupes parahoriques}\label{groupes sur Z[t]}

Dans cette partie, on construit au \S\ref{subsection groupes immobiliers modelant tits} les $\AAA^1_\ZZ$-modèles en groupes \og immobiliers \fg{} $\underline{\Gscr_{\Omega}}$ des groupes de Tits $\underline{G}$ sur $\GG_{m,\ZZ}$ introduits au \S\ref{subsection groupes de Tits} et l'on étudie leurs propriétés, démontrant notamment leur affinité, cf. th. \ref{theoreme affinite modeles immobiliers}.

\subsection{Théorie de Bruhat--Tits pour les groupes pseudo-réductifs}\label{subsection theorie de bt pseudo-reductive}

Comme il a été expliqué dans l'introduction, le but principal de cet article nous amène à travailler avec des groupes pseudo-réductifs en tant que spécialisations en caractéristiques petites et/ou mauvaises de certains groupes réductifs, voir la prop. \ref{proposition fibre generique quasi-reductive}.

Commençons par rappeler cette notion pour ceux qui ne la connaissent pas bien :

\begin{defn}
	Soient $k$ un corps et $G$ un $k$-groupe lisse, affine et connexe. On dit que $G$ est pseudo-réductif s'il ne contient aucun sous-$k$-groupe invariant, unipotent, lisse et connexe (autrement dit, si son $k$-radical unipotent $R_{u,k}G$ s'annule).
\end{defn} 

Cette classe des groupes pseudo-réductifs contient nettement celle des groupes réductifs connexes et est stable par extensions des scalaires séparables, voir \cite[prop. 1.1.9]{CGP}. Voici l'exemple prototypique, cf. \cite[prop. 1.1.10]{CGP}, d'un groupe pseudo-réductif qui n'est pas réductif :

\begin{exmp}
	Soient $l/k$ une extension finie non-triviale et purement inséparable de corps et $H$ un $l$-groupe réductif. Alors, $G=\Res_{l/k}H$ est pseudo-réductif, vu que, si l'on note $U$ le $k$-radical unipotent de $G$, alors l'image de $U_l \subset G_l$ dans $H$ est lisse, connexe, distinguée et unipotente, donc $U=1$ se réduit à l'unité lui-même. D'autre part, $G$ n'est pas réductif, car la $p$-torsion du centre de $G(k^{\on{alg}})=H(l\otimes_k k^{\on{alg}})$ est un groupe infini, comp. avec \cite[lem. 1.2.1]{CGP}, d'après un calcul élémentaire appliquant que $l\otimes_k k^{\on{alg}}$ est une $k^{\on{alg}}$-algèbre non réduite.
\end{exmp}

Malgré cet exemple, on pourrait craindre qu'il ne soit typiquement pas possible d'apprivoiser ces groupes pseudo-réductifs et qu'ils n'admettent aucun théorème raisonnable, soit de structure, soit de classification. Au contraire, Borel--Tits avaient déjà découvert dans \cite{BoT2} que leur structure générale ressemble beaucoup à celle des groupes réductifs et connexes : 

\begin{thm}[Borel--Tits]\label{theoreme borel-tits structure}
	Tous les tores déployés maximaux de $G$ sont conjugués par $G(k)$. Étant donné un tore déployé maximal $S$ de $G$, alors l'ensemble des poids $\Phi_G$ de l'action de $S$ sur l'algèbre de Lie $\gf$ de $G$ se prolonge naturellement en une donnée radicielle $(X_*(S),X^*(S),\Phi_G^\vee,\Phi_G)$. Les espaces de poids $\gf_a\oplus \gf_{2a}$ sont sous-jacents à l'un et un seul sous-$k$-groupe lisse, unipotent et connexe $U_a$ de $G$, tel que l'application produit
	\begin{equation}\label{equation immersion ouverte grosse cellule}\prod_{a \in \Phi_G^-} U_{-a} \times Z_G(S) \times \prod_{a \in \Phi_G^+} U_a \rightarrow G,\end{equation}
	soit une immersion ouverte, quel que soit le choix de racines positives et quel que soit l'ordre en lequel le produit est mis. Ces données ne dépendent de $S$ qu'à $G(k)$-conjugaison près.
\end{thm}

\begin{proof}
	Toutes les affirmations avaient été annoncées en \cite{BoT2}, dont les preuves n'étaient qu'esquissées. Puis Conrad--Gabber--Prasad ont fourni dans leur \oe{}uvre \cite{CGP} des arguments complets, voir \cite[prop. 2.1.8, th. C.2.3, th. C.2.15, prop. C.2.26]{CGP}.
\end{proof}

L'autre question qu'on s'est naturellement posée était de savoir s'il est loisible de classifier tous les groupes pseudo-réductifs. Tits avait déjà remarqué dans l'importance de l'application canonique
\begin{equation}\label{equation application naturelle restriction des scalaires}
i_G: G \rightarrow \Res_{k'/k} G',
\end{equation} où $k'$ est le corps de définition du radical unipotent géométrique $R_u G_{k^{\on{alg}}}$ et $G'$ le $k^{\on{alg}}$-quotient réductif de $G_{k^{\on{alg}}}$. On ne peut pas s'attendre typiquement à ce que $i_G$ soit injectif, soit surjectif, parce que $G$ peut être modifié centralement par la procédure de modification centrale de (\ref{equation modifie central tordu}), et vu que les groupes pseudo-réductifs commutatifs sont habituellement n'importe quoi, cf. \cite[ex. 1.6.3, ex. 1.6.4, \S11.3]{CGP}.
À ces phénomènes de Cartan près, une classification exhaustive a été trouvée par Conrad--Gabber--Prasad, cf. \cite[th. 5.1.1, th. 10.2.1]{CGP}.

\begin{thm}[Conrad--Gabber--Prasad]\label{theoreme classification cgp}
	Soit $k$ un corps de caractéristique $p$ satisfaisant à $[k:k^p]\leq p$ dès que $p$ est positif. Alors tous les groupes pseudo-réductifs s'obtiennent naturellement à partir des groupes simplement connexes, des exotiques et des barcelonais en combinant les opérations produit, restriction des scalaires et modification centrale.
\end{thm}

Les groupes exotiques, voir \cite[déf. 7.2.6]{CGP}, et les barcelonais, cf. \cite[déf. 10.1.2]{CGP}, n'existent qu'en caractéristiques petites $p=2$ ou $3$ et sont reliés à certains phénomènes de $p$-divisibilité du système de racines $\Phi_G$. Dans la suite, nous allons éclaircir leur aspect dans le cas quasi-déployé.

\begin{rem}\label{remarque classification cp}
	Le th. \ref{theoreme classification cgp} vaut plus généralement lorsque $p \neq 2$ ou $p=2$ et $[k:k^2]\leq 2$. Le cas restant où $p=2$ et $[k:k^2]>2$ fut traité par Conrad--Prasad dans \cite{CP}, en supposant que $G$ est localement de type minimal, voir \cite[\S1.6, th. de str.]{CP}.
\end{rem}

Penchons-nous d'abord sur le cas exotique :

\begin{exmp}[exotique]\label{exemple exotique basique}
	Soient $k$ un corps imparfait de caractéristique $p$ satisfaisant à $[k:k^p]= p$ et $G$ un groupe exotique basique et quasi-déployé, voir \cite[déf. 7.2.6]{CGP}. Alors, on a $p=2,3$, $\Phi_G$ est irréductible et possède une $p$-arête, donc on peut appliquer \cite[lem. C.2.2]{CP} pour déduire que $G$ est pseudo-déployé. D'après \cite[th. 7.2.3]{CGP}, le $k^{1/p}$-quotient réductif $G'$ de $G$ est déployé et simplement connexe à système de racines $\Phi_G$, tandis que l'application canonique $ i_G :G \rightarrow \Res_{k^{1/p}/k} G'$ de (\ref{equation application naturelle restriction des scalaires}) est une immersion fermée. Soit \begin{equation}\pi': G' \rightarrow \overline{G}'\end{equation} l'isogénie très spéciale, voir \cite[déf. 7.1.3]{CGP}, définie sur $k^{1/p}$, où $\overline{G}'$ est un groupe déployé et simplement connexe de type dual. Alors, l'image de $G$ par la restriction des scalaires \begin{equation} \pi : \Res_{k^{1/p}/k}  G' \rightarrow\Res_{k^{1/p}/k} \overline{G}'\end{equation} de l'isogénie très spéciale $\pi'$ est un $k$-Levi $\overline{G} $ du membre de droite, c'est-à-dire un $k$-descendu de $\overline{G}'$, et $G$ s'identifie à l'image réciproque par $\pi$ de ce Levi $\overline{G}$ dans $\Res_{k^{1/p}/k}  G'$, voir \cite[déf. 7.2.6]{CGP}.
	
	On peut trouver un $k^{1/p}$-système de Chevalley, cf. déf. \ref{definition d'un systeme de Chevalley}, de $G'$ (resp. $k$-système de Chevalley de $\overline{G}$) tels que l'isogénie très spéciale s'identifie à
	\begin{equation}\pi: x'_a(r) \mapsto \begin{cases} \overline{x}_{\overline{a}}(r^p), a \in \Phi_{G}^<\\
	\overline{x}_{\overline{a}}(r), a \in \Phi_G^>
	\end{cases}\end{equation}
	restreinte aux groupes radiciels relativement à $S'$ (resp. $\overline{S}$) et à
	\begin{equation}\pi:a^\vee(r) \mapsto \begin{cases} \overline{a}^\vee(r^p),  a \in \Phi_G^<\\
	\overline{a}^\vee(r), a \in \Phi_G^>
	\end{cases}\end{equation}
	restreinte aux facteurs multiplicatifs associés aux coracines, comp. avec \cite[prop. 7.1.5]{CGP}, où les exposantes $<,>$ notent les racines courtes, resp. longues de $\Phi_G$.
	
	Ainsi, on obtient des isomorphismes
	\begin{equation} x_a:\Res_{k_a/k}\GG_a \rightarrow U_a  \end{equation}\begin{equation}
	a^\vee: \Res_{k_a/k}\GG_a \rightarrow Z^a, 
	\end{equation}
	où $k_a=k^{1/p}$ si $a \in \Phi_G^<$, $k_a=k$ si $a \in \Phi_G^>$ et $Z^a$ désigne l'intersection $Z \cap \langle U_a, U_{-a}\rangle$, cf. \cite[rem. 7.2.8]{CGP}.
\end{exmp}

Le groupe pseudo-réductif suivant est aussi pseudo-déployé, mais son système de racines absolu n'est pas réduit, c'est-à-dire de type $BC_n$.\footnote{Cela justifie que ces groupes soient appelés barcelonais ci-après. La raison pour laquelle nous avons préféré cette terminologie au mot \og non réduit \fg{} utilisé dans \cite{CGP} est que ce dernier entre en conflit avec son sens en géométrie algébrique, ce qui fut très importante afin de motiver la modification de Tits, voir la rem. \ref{remarque limite inseparable pluriel} et la prop. \ref{proposition fibre generique quasi-reductive}.}

\begin{exmp}[barcelonais]\label{exemple absolument non reduit}
	Soient $k$ un corps imparfait de caractéristique $2$ tel que $[k:k^2]= 2$ et $G$ un groupe pseudo-réductif barcelonais basique, cf. \cite[déf. 10.1.2]{CGP}. Alors, on sait déjà que $G$ est pseudo-déployé et que son $k^{1/2}$-quotient réductif $G'$ est déployé et simplement connexe à système de racines $\Phi_G^{\on{nm}}$, cf. \cite[th. 2.3.10]{CGP}. L'application correspondante $i_G: G \rightarrow \Res_{k^{1/2}/k}G'$, voir (\ref{equation application naturelle restriction des scalaires}), induit des isomorphismes
	\begin{equation}Z \rightarrow \Res_{k^{1/2}/k}Z', \end{equation} \begin{equation} U_a \rightarrow \Res_{k^{1/2}/k}U'_a, a \in \Phi_G^{\on{nd,nm}}\end{equation} 
	entre les Cartans et les sous-groupes radiciels associés aux racines ni divisibles ni multipliables, voir \cite[th. 9.4.7]{CGP}. Cependant, d'après \cite[prop. 9.5.2, prop. 9.6.8]{CGP} le morphisme $U_a \rightarrow \Res_{k^{1/2}/k} U'_{2a}$ pour $a \in \Phi_G^m$ peut être identifié avec
	\begin{equation}\GG_{p,k^{1/2}/k} \rightarrow \Res_{k^{1/2}/k}\GG_a\end{equation}
	donné par $(u,v)\mapsto u^2+v$, où l'on note $\GG_{p,k^{1/2}/k}$ le groupe unipotent pluriel modifié selon la définition limite de la rem. \ref{remarque limite inseparable pluriel} (ceci ne dépend pas à isomorphisme près du choix d'un complément $k$-linéaire $k_0= k\subset k^{1/2}$ pour $\alpha \in k^{1/2}\setminus k$). On signale que la construction de \cite[\S\S9.6-9.7]{CGP} utilise les lois birationnelles strictes par réduction au cas de $\Res_{k^{1/2}/k}G'$, tandis que Tits a réalisé $G$ explicitement comme une sorte de groupe unitaire de dimension impaire dégénéré en caractéristique $2$, voir \cite[pp. 218-219]{TitsOber}. Par alternative, voir les props. \ref{proposition fibre generique quasi-reductive} et \ref{proposition plongement restrictions des scalaires groupe de tits}.
\end{exmp}

Motivés par l'étude des groupes algébriques sur les corps valués, Bruhat--Tits ont crée leur théorie des immeubles et des modèles en groupes immobiliers, cf. \cite{BTI} et \cite{BTII}. Les exemples précédents montrent que les points rationnels des groupes pseudo-réductifs ne s'éloignent beaucoup de ceux des groupes réductifs. Solleveld s'y est appuyé dans \cite{Sol} pour parvenir à construire des immeubles et, dans \cite[déf. 3.5, th. 4.9]{LouBT}, nous avons construit les schémas en groupes.

\begin{thm}[Solleveld, L.]\label{theoreme theorie de bruhat-tits pseudo-reductive}
	La théorie de Bruhat--Tits s'étend aux groupes pseudo-réductifs sur les corps complets et résiduellement parfaits.
\end{thm}

\begin{rem}
	Cet énoncé devrait être valable sur les corps discrètement valués et henséliens $k$ et pour les $k$-groupes $\widehat{k}$-pseudo-réductifs et $k^{\on{nr}}$-quasi-déployés. Dans \cite{LouBT}, on suppose encore que $\widehat{k}/k$ est séparable et que le corps résiduel est parfait. Ce dernier a été éliminé dans \cite[IV, \S\S3.1-3.2]{LouDiss} et Conrad m'a convaincu que le premier devrait être superflue.
\end{rem}

Notre construction des immeubles de Bruhat--Tits dans \cite{LouBT} est pareille à celle de \cite{BTII}. Supposons encore que $G$ soit quasi-déployé. À chaque sous-tore déployé maximal $S$ de $G$, on associe un appartement $\Ascr(G,S,k)$, en vérifiant les axiomes de \cite[déf. 6.1.1, déf. 6.2.1]{BTI} pour appliquer le \cite[th. 6.5]{BTI}, cf. aussi \cite[déf. 7.4.2]{BTI}. Le seul point compliqué est donc de trouver une valuation $\varphi$ de la donnée radicielle abstraite $(Z(k),U_a(k))$, pour lequel nous recourons à un quasi-système de Chevalley de $G$, cf. déf. \ref{definition quasi-systemes de chevalley} et \cite[prop. 2.5]{LouBT} pour le cas pseudo-réductif. Alors, on pose
\begin{equation}\label{equation valuation singulier} \varphi_a(x_a^{\on{CS}}(r))=\omega(r) \end{equation}\begin{equation}\label{equation valuation pluriel unitaire}
 \varphi_a(x_a^{\on{CS}}(r,s))=\omega(s)/2, \end{equation}\begin{equation}\label{equation valuation pluriel barcelonais}
 \varphi_a(x_a^{\on{CS}}(u,v))=\omega(u^2+v)/2 
\end{equation}
selon que $a\in \Phi_G^{\on{nd,nm}}$ est une racine ni divisible ni multipliable, ou que $a \in \Phi_G^{\on{m}}$ est multipliable et $G^a$ réductif, ou que $a \in \Phi_G^{\on{m}}$ est multipliable et $G^a$ non standard. Ici, l'on note $\omega$ la valuation additive de $k$ étendue à sa clôture algébrique et satisfaisant à $\omega(k)=\ZZ$.

Disons un mot sur les modèles entiers de $G$ associés aux parties non vides et bornées $\Omega$ de l'appartement $\Ascr(G,S,k)\subset \Iscr(G,k)$, où $G$ est encore supposé quasi-déployé. D'un côté, il faut construire le modèle de Néron connexe $\Zscr$ d'un groupe pseudo-réductif commutatif $Z$, ce qui fut achevé par Bosch--Lütkebohmert--Raynaud dans \cite[th. 10.2.2]{BLR}. De l'autre côté, on a besoin de trouver les modèles entiers $\Uscr_{a,\Omega}$ stabilisant $\Omega$ pour toute racine non divisible $a \in \Phi_G^{\on{nd}}$, ne s'agissant que des calculs simples, voir \cite[\S4.3]{BTII} et \cite[lem. 4.4]{LouBT}. Enfin, on recolle ces données en appliquant la théorie des lois birationnelles, voir \cite[\S4.1, th. 4.2]{LouBT}, ce qu'on reprendra d'ailleurs toute de suite au \S\ref{donnees radicielles schematiques}, cf. ths. \ref{theoreme existence de solutions} et \ref{theoreme d'existence et unicite des groupes attaches aux drs}.
\begin{rem}
	Presque tous les résultats que nous avons mentionnés s'étendent aux groupes quasi-réductifs, dont le $k$-radical unipotent n'est peu-être pas non plus trivial, mais bien ployé, c'est-à-dire sans sous-groupe unipotent additif, voir \cite[1.1.12]{BTII} et \cite[déf. B.2.1]{CGP}. Ceci a été caché dans cette section afin de ne pas effrayer le lecteur sans raison.
\end{rem}
\subsection{Données radicielles schématiques} \label{donnees radicielles schematiques}
Dans ce paragraphe, on discute le concept des données radicielles schématiques dû à Bruhat--Tits, cf. \cite[déf. 3.1.1]{BTII}. Leur but, comme le lecteur bientôt comprendra, est de construire des $A$-modèles en groupes $\Gscr$ de $G$ avec de bonnes propriétés par \og recollement fermé \fg{} de $A$-modèles en groupes correspondants pour chaque terme individuel $Z$ et $U_a$ de la grosse cellule, où $a \in \Phi_G^{\on{nd}}$ décrit l'ensemble des racines non divisibles.

\begin{defn}[Bruhat--Tits, à peu près]\label{definition donnees radicielles schematiques}
Soient $A$ un anneau noethérien et intègre, $K$ son corps de fractions, $G$ un $K$-groupe pseudo-réductif et $S$ un sous-tore déployé maximal. Une donnée radicielle dans $G$ par rapport à $S$ au-dessus de $A$ est la donnée de $A$-modèles en groupes lisses, affines et connexes $\Zscr$ de $Z=Z_G(S)$ (resp. $\Uscr_a$ du sous-groupe radiciel $U_a$ attaché à la racine non divisible $a \in \Phi_G^{\on{nd}}$) tels que :
\begin{itemize}[leftmargin=1.7cm]
\item[(DRS 0)] l'injection de $S$ dans $Z$ se prolonge en un morphisme $\Sscr \rightarrow \Zscr$, où $\Sscr$ désigne le seul groupe diagonalisable à isomorphisme unique près modelant $S$.
\item[(DRS 1)] l'application conjugaison $Z \times U_a \rightarrow U_a$ se prolonge en une application $\Zscr \times \Uscr_a \rightarrow \Uscr_a$.
\item[(DRS 2)] si $b \neq -a$, l'application commutateur $U_a \times U_b \rightarrow \prod_{c\in ] a, b[} U_c$ se prolonge en un morphisme $\Uscr_a \times \Uscr_b \rightarrow \prod_{c\in ] a, b[} \Uscr_c$, où l'on note $] a, b[$ la partie de $\Phi_{G}^{\on{nd}}$ contenue dans $\QQ_{> 0}a+\QQ_{ >0}b$.
\item[(DRS 3)] si on note $W_a$ l'image réciproque par l'immersion fermée $U_a \times U_{-a} \hookrightarrow G$ du fermé local $ C_a:=U_{-a} \times Z \times U_a\subset G$, cf. (\ref{equation immersion ouverte grosse cellule}), alors il existe un ouvert $\Uscr_{\pm a} \subset\Wscr_a\subset \Uscr_a \times \Uscr_{-a}$ modelant $W_a$ tel que l'inclusion $W_a \subset C_a$ se prolonge en une application $\Wscr_a \rightarrow \Cscr_a:=\Uscr_{-a} \times \Zscr \times \Uscr_a$.
\end{itemize}
\end{defn} 

\begin{rem}\label{remarque difference definition drs}
Il y a deux différences entre notre définition et celle de Bruhat--Tits, comp. avec \cite[déf. 3.1.1]{BTII}: on ne demande pas que l'application $\Sscr \rightarrow \Zscr$ soit une immersion fermée, puisque cela est superflue, voir \cite[exp. IX, cor. 2.5, th. 6.8]{SGA3} ; notre voisinage ouvert $\Wscr_a$ contient les fermés $\Uscr_{\pm a}$ et pas seulement la section unité, pour que la loi birationnelle sur $\Cscr:=\Uscr_-\times \Zscr \times \Uscr_+$ soit automatiquement stricte, cf. déf. \ref{definition lois birationnelles} et th. \ref{theoreme d'existence et unicite des groupes attaches aux drs}.
\end{rem}

 Notre démonstration du théorème concernant l'existence du recollé $\Gscr$ fera usage des lois birationnelles de groupes.
 
 \begin{defn}\label{definition lois birationnelles}
 	Soient $A$ un anneau noethérien et $X$ un $A$-schéma séparé de type fini. Alors, une loi birationnelle de groupe sur $X$ est la donnée d'un morphisme $A$-rationnel $m: X \times X \rightarrow X$ tel qu'il soit associatif et que les applications de translation universelles à la gauche $(\on{id},m)$ resp. à la droite $(m, \on{id})$ soient $A$-birationnelles. On dit que la loi birationnelle est stricte si les applications de translation universelles sont $X$-birationnelles relativement aux deux projections $X\times X \rightarrow X$.
 \end{defn}

On suit la terminologie de \cite{BLR} en dénommant un morphisme de présentation finie $A$-rationnel si son lieu ouvert de définition est à $A$-fibres denses. Étant donnée une loi birationnelle de groupe sur $X$, la question qui se pose est de la prolonger birationnellement à un schéma en groupes $\overline{X}$. 

\begin{thm}[Existence de solutions]\label{theoreme existence de solutions}
	Gardons les notations de la déf. \ref{definition lois birationnelles}. Alors, il existe un et un seul $A$-groupe séparé de type fini $\overline{X}$ muni d'un morphisme $A$-birationnel $X \dashrightarrow \overline{X}$ préservant les lois birationnelles de groupe. Si la loi birationnelle de $X$ est stricte ou si $A$ est normal et $X$ lisse, alors l'application rationnelle $X \dashrightarrow \overline{X}$ se prolonge en une immersion ouverte définie partout.
\end{thm} 

\begin{proof}
	On ne fera que résumer la chronologie du degré de généralité dont le théorème d'existence de solutions de lois birationnelles en groupes a été démontré. Alors, au-dessus des corps ce résultat remonte à Weil. Puis il fut amélioré par Artin, voir \cite[exp. XVIII, th. 3.7]{SGA3}, sur des bases assez générales, mais la solution n'était donnée qu'en espaces algébriques. Finalement, Bosch--Lütkebohmert--Raynaud, voir \cite[th. 6.6.1]{BLR}, montrèrent que la solution $\overline{X}$ est un schéma. Observons qu'on s'est toujours ramené au cas où la loi birationnelle est stricte, d'après \cite[prop. 5.2.2]{BLR}. Enfin, si la base est normale et $X$ lisse, alors on n'a pas besoin de remplacer la grosse cellule par un ouvert $A$-schématiquement dense grâce au \cite[th. 5.1.5]{BLR}.
\end{proof}
 
 \begin{thm}[Bruhat--Tits, à peu près]\label{theoreme d'existence et unicite des groupes attaches aux drs}
Reprenons les notations de la déf. \ref{definition donnees radicielles schematiques}. Alors, il existe un et un seul $A$-groupe séparé, lisse et connexe $\Gscr$ modelant $G$ tel que :
\begin{enumerate}
\item les inclusions $Z \rightarrow G$ et $U_a \rightarrow G$ pour tout $a \in \Phi_G^{\on{nd}}$ se prolongent en des isomorphismes de $\Zscr$ et $\Uscr_a$ sur des sous-$A$-groupes fermés de $\Gscr$.
\item pour tout système de racines positives $\Phi_G^+$ et tout ordre mis sur $\Phi_G^+$, le morphisme $\prod_{a \in \Phi_G^{\on{nd},+}} \Uscr_a \rightarrow \Gscr$ soit un isomorphisme sur un sous-$A$-groupe fermé $\Uscr_+$ de $\Gscr$.
\item l'application produit $\Uscr_-\times \Zscr \times \Uscr_+ \rightarrow \Gscr$ définisse un ouvert du membre de droite.
\end{enumerate}
\end{thm}

Avant de initier la preuve en tant que telle, résumons son histoire complexe.

\begin{rem}\label{remarque histoire construction du recolle ferme}
	 Dans \cite[3.1.3]{BTII}, Bruhat--Tits évoquent leur envie d'appliquer les lois birationnelles pour répondre à la question, mais \cite[exp. XVIII, th. 3.7]{SGA3} ne fournissait que des solutions en espaces algébriques et non pas en schémas -- le livre \cite{BLR} n'avait pas encore été écrit, cf. th. \ref{theoreme existence de solutions}. Par conséquent, ils consacrent 30 pages à son argument alternatif avec la théorie des représentations, cf. \cite[\S3.8, 3.9.4]{BTII}. 
	Sur un anneau de valuation discrète, Landvogt donnera une démonstration appliquant les lois birationnelles, voir \cite[\S\S5--6]{LndvCpc}, tout en empruntant des idées superflues de \cite[\S3]{BTII}.
\end{rem}
\begin{proof}
Lorsqu'on se donne un ordre grignotant sur $\Phi^+$, cf. \cite[3.1.2]{BTII}, alors \cite[\S3.3]{BTII} fournit une loi de groupe dans le produit $\Uscr_+=\prod_{a \in \Phi^{+,\on{nd}}}\Uscr_{a}$, en argumentant par récurrence sur la longueur des parties positivement closes $\Psi$ de $\Phi^{+,\on{nd}}$ et en appliquant la règle (DRS 2). Comme les groupes trouvés sont toujours lisses et connexes, le fait que $\Uscr_+$ s'écrit comme produit de ses sous-groupes fermés $\Uscr_{a}$ rangés dans un ordre quelconque résulte du théorème principal de Zariski et de la même indépendance de l'ordre au-dessus des corps, ce qui est déjà classique, voir par exemple \cite[prop. C.2.26]{CGP}. Notons également que $\Zscr$ opère sur $\Uscr_+$ canoniquement d'après la règle (DRS 1), donnant lieu à un $A$-groupe noté parfois $\Zscr\Uscr_+:=\Zscr \ltimes \Uscr$.

Pour construire $\Gscr$ satisfaisant à la troisième propriété, il suffit de construire une loi birationnelle stricte sur $\Cscr_{+}:=\Uscr_-\times \Zscr \times\Uscr_+$ pour un choix quelconque $\Phi^+$ de racines positives, d'après le th. \ref{theoreme d'existence et unicite des groupes attaches aux drs}. Il reste à définir une application birationnelle $\Cscr_{+} \dashrightarrow \Cscr_{-}$ prolongeant l'identité de $G$ dans la fibre générique et contenant la section unité, car cela permettrait de définir des applications rationnelles de produit et d'inversion dans $\Cscr_{+}$, en commutant tous les facteurs et en multipliant ou inversant lorsqu'il le faut. 

Montrons plus généralement la même assertion pour chaque paire de systèmes de racines positives $\Phi^+$ et $\Psi^+$ par récurrence sur le cardinal des rayons de $ \Phi^+ \cap \Psi^- $. Si cette intersection ne contient que le rayon radiciel $\QQ_{>0}a \cap \Phi$ engendré par la racine non divisible $a \in \Phi$, alors l'axiome (DRS 3) en fournit un morphisme rationnel \begin{equation}\label{equation lois birationnelles echange}\Cscr_{\Phi^+}=\Vscr_-\times \Uscr_{-a} \times \Zscr \times \Uscr_a \times \Vscr_+ \dashrightarrow \Vscr_-\times \Uscr_{a} \times \Zscr \times \Uscr_{-a}\times  \Vscr_+=\Cscr_{\Psi^+}\end{equation}
défini dans un voisinage ouvert de la section unité, où $\Vscr^{\pm}:= \Uscr_{\Phi^{\pm,\on{nd}}\setminus\{\pm a\}}$. La construction montre aussi que la loi birationnelle est indépendante du système de racines positives $\Phi^+$ choisi, donc il en sera de même du groupe $\Gscr$ obtenu à la fin. 

Pour que la loi birationnelle esquissée ci-dessus soit stricte, il faut et il suffit que le morphisme multiplication $\Cscr_+ \times \Cscr_+ \dashrightarrow \Cscr_+$ soit un épimorphisme, vu que la grosse cellule $\Cscr_+$ est géométriquement connexe, cf. \cite[rem. 9.6.3]{CGP}. Or, considérant le fermé $\Uscr_- \times \Zh\Uscr_+ \subset \Cscr_+ \times \Cscr_{+}$ dont les points sont de la forme $(a,1,1,b)$, on déduit aisément que ceux-ci appartiennent au lieu de définition du morphisme multiplication, compte tenu de ce que les morphismes d'échange (\ref{equation lois birationnelles echange}) soient définis autour de $\Zscr \Uscr_{\pm}$.

Maintenant, on dispose d'un $A$-groupe lisse, séparé et connexe $\Gscr$ qui modèle $G$ et qui contient les différents $\Cscr_{\Phi^+}$ en tant qu'ouverts, et il faut encore voir que les sous-groupes localement fermés $\Zscr$ et $\Uscr_+$ sont effectivement fermés. Or l'image réciproque de $\Zscr\Uscr_+$ par le revêtement fidèlement plat $\Cscr_{\Phi^-} \times \Cscr_{\Phi^+} \rightarrow \Gscr$ est déterminée par la condition fermée $b_1b_2=1$, où $(a_1,b_1, b_2, a_2) \in \Zscr\Uscr_+ \times (\Uscr_-)^2 \times \Zscr \Uscr_+$, ce qui achève la vérification de l'affirmation par descente fidèlement plate.
\end{proof}

En principe, on voudrait que le groupe $\Gscr$ soit affine, mais cela ne découle ni de la construction avec des lois birationnelles ni de celle de \cite[\S3.8, 3.9.4]{BTII} : en effet, \cite[\S VII.3]{Ray} fournit un exemple d'un groupe quasi-affine, lisse et connexe qui n'est pas affine ; remarquons néanmoins que, par contre, celui de \cite[3.2.15]{BTII} s'avère affine. Toutefois, l'on a :

\begin{propn}[Raynaud, à peu près]\label{proposition enveloppe affine}
Soient $A$ un anneau noethérien et régulier de dimension inférieure ou égale à $2$ et $\Gscr$ un $A$-modèle en groupes lisse, connexe et séparé de $G$. Alors, la $A$-algèbre $\Gamma(\Gscr, \Oh_{\Gscr})$ est munie d'une structure canonique de $A$-algèbre de Hopf plate de type fini. De plus, le $A$-groupe résultant $\Gscr^{\emph{af}}:=\spec \Gamma(\Gscr, \Oh_{\Gscr})$ est lisse, affine et sa composante neutre s'identifie à $\Gscr$ par l'application naturelle.
\end{propn}

\begin{proof}
Le groupe $\Gscr$ est tout d'abord quasi-affine grâce au \cite[cor. VII.2.2]{Ray}, ce qu'on peut appliquer car $A$ est normal et $\Gscr$ est lisse, séparé et géométriquement connexe. Raynaud avait montré en \cite[prop. VII.3.1]{Ray}, que, lorsqu'on se trouve dans les circonstances énoncées, $\Gamma(\Gscr, \Oh_{\Gscr})$ était fidèlement plate sur $A$ et, par suite, que $\Gscr^{\on{af}}$ était un $A$-groupe plat admettant $\Gscr$ comme sous-groupe ouvert par le morphisme canonique $\Gscr \rightarrow \Gscr^{\on{af}}$. Pour aider le lecteur, nous allons expliquer brièvement la démonstration de la platitude, comp. avec \cite[lem. VII.3.2]{Ray}. 

Soit $B$ un $A$-module fini contenu dans $\Gamma(\Gscr,\Oh_\Gscr)$. Alors, $B$ induit un fibré vectoriel au-dessus d'un ouvert cofini $U$ de $S=\spec A$, qui s'étend donc en un $A$-module libre et fini $\tilde{B} = \Gamma(U,B\otimes \Oh_S)$, d'après \cite[th. 5.10.5, prop. 5.11.1]{EGAIV}. Mais par changement de base plat, on déduit l'égalité $\Gamma(U,\pi_*\Oh_\Gscr)=\Gamma(S,\pi_*\Oh_\Gscr)$ du deuxième lemme d'extension de Riemann $\Gamma(U,\Oh_S)=A$. Alors, ceci entrâine que $\tilde{B} \subseteq \Gamma(\Gscr,\Oh_\Gscr)$, d'où une écriture de cette dernière $A$-algèbre comme la colimite filtrée de ses sous-$A$-modules plats. Ce que $\Gscr^{\on{af}}$ soit maintenant un $A$-groupe affine et (fidèlement) plat est trivial, cf. \cite[exp. VIB, lem. 11.1]{SGA3}.

Malheureusement, il n'est pas clair que $\Gamma(\Gscr, \Oh_{\Gscr})$ soit une $A$-algèbre de type fini. Néanmoins, d'après \cite[prop. 2.3.1]{Anan} ou \cite[exp. VIB, prop. 12.9]{SGA3}, on peut trouver un ouvert cofini $U$ de $\spec A$ tel que $\Gscr_U$ soit affine sur $U$, d'où en particulier la présentation finie de $\Gscr^{\on{af}}$ au-dessus de $U$. Par descente fidèlement plate et quasi-compacte, on se ramène au cas où $A$ est un anneau local d'idéal maximal $\mf$, $\mf$-adiquement séparé et complet, et à corps résiduel $k$ algébriquement clos.

Nous affirmons d'abord que $\Gscr^{\on{af}}$ est de type fini au-dessus de $A/\mf^n$ pour tout entier positif $n$, ce qui résulte immédiatement du cas $n=1$ grâce au lemme de Nakayama nilpotent. Alors, le monomorphisme $\Gscr_k \rightarrow \Gscr^{\on{af}}_k$ est ouvert par hypothèse et fermé d'après le \cite[exp. VIB, lem. 11.18.1]{SGA3}, identifiant donc le membre de gauche à la composante neutre de $\Gscr^{\on{af}}_k$. Vu que le faisceau quotient $\pi_0(\Gscr_k^{\on{af}}):=\Gscr^{\on{af}}_k/\Gscr_k$ est représentable par un groupe proétale, cf. \cite[III, \S3, 7.7]{DG}, dont la section unité est ouverte, on en tire que $\pi_0(\Gscr_k^{\on{af}})$ est un groupe étale fini et que le $k$-groupe $\Gscr_k^{\on{af}}$ est forcément de type fini.

En sachant que le schéma affine $\Gscr^{\on{af}}$ est lisse au-dessus de chaque quotient artinien de $A$, on peut relever par récurrence tout point de $\Gscr^{\on{af}}$ à valeurs dans $k$ en un $A$-point correspondant (de façon non unique évidemment) et par translation de $\Gscr$ dans $\Gscr^{\on{af}}$ par ces points, on dérive la lissité de ce dernier.
\end{proof}
\subsection{Les groupes de Tits sur $\GG_{m,\ZZ}$}\label{subsection groupes de Tits}
Dans le présent paragraphe, nous allons construire les $\GG_{m,\ZZ}$-groupes de Tits à la suite de \cite[annexe 2]{TitsOber}, mission pour laquelle les notations du \S\ref{rappels et preparatifs} seront reprises. Néanmoins, il faudra que des hypothèses additionnelles soient faites :

\begin{hypt}\label{hypothese groupe dans la fibre generique}
	Sauf mention contraire, $G$ sera un $\QQ(t)$-groupe réductif connexe, quasi-déployé, $\QQ(\zeta_e,t^{1/e})$-déployé pour un certain entier positif $e$. Étant donné un tore déployé maximal $S \subset G$, on fixe un couple de Killing $Z_G(S)=T \subset B$ et un quasi-système de Tits $(x_a)_{a \in \Phi_G}$ de Tits, voir la déf. \ref{definition quasi-systeme de Tits}. On suppose de plus que $T$ est induit et que $G_{\QQ\rpot{t}}$ est résiduellement déployé.
	\end{hypt}

Les notions correspondantes pour la forme déployée $H$ de $G$ seront notées $T_H$, $B_H$, $(y_\alpha)_{\alpha \in \Phi_H}$ et $\Gamma$ désignera le groupe de Galois de $\QQ(\zeta_e,t^{1/e})/\QQ(t)$ opérant sur $H$ par automorphismes de Dynkin--Tits, cf. \S\S\ref{formes tordues quasi-deployees}-\ref{facteurs de 2 dans le cas non reduit}.

Tout d'abord signalons que le déploiement résiduel de $G_{\QQ\rpot{t}}$ dans l'hyp. \ref{hypothese groupe dans la fibre generique} sert à que les extensions de $\QQ(t)$ associées aux racines $a \in \Phi_G$ soient conjuguées de $\QQ(t_a)$, où $t_a=t^{1/e_a}$ et le nombre naturel $e_a$ divise $e$, c'est-à-dire que l'on ait des quasi-épinglages
\begin{equation}\label{equation isomorphisme groupe radiciel singulier}
x_a: U_a \xrightarrow{\sim} \Res_{\QQ(t_a)/\QQ(t)}\GG_a
\end{equation}
si $a \in \Phi^{\on{nd,nm}}_G$ n'est ni divisible ni multipliable ou
\begin{equation}\label{equation isomorphisme groupe radiciel pluriel}
x_a: U_a \xrightarrow{\sim} \Res_{\QQ(t_{2a})/\QQ(t)}\GG_{p,\QQ(t_a)/\QQ(t_{2a})}
\end{equation}
si $a \in \Phi_G^{\on{m}}$ est multipliable. D'autre part, le fait que $T$ soit induit entraîne que le tore maximal s'exprime comme produit de facteurs de la forme \begin{equation}\label{equation tore decompose} T\simeq\prod_{j\in J}\Res_{\QQ(t^{1/d_j})/\QQ(t)} \GG_m \end{equation} avec $d_j$ des diviseurs naturels de $e$. 

Pour conclure, notons que, si $G$ remplit l'hyp. \ref{hypothese groupe dans la fibre generique}, alors il en est de même de son revêtement simplement connexe $G^{\on{sc}}$, de son groupe adjoint $G^{\on{ad}}$, voir \cite[prop. 4.4.16]{BTII}, et de la $z$-extension canonique associée $\widetilde{G}$ au sens de (\ref{equation modifie central tordu}), d'où la richesse véritable de la classe d'exemples ici considérés. 

\begin{rem}Notons que, dans \cite[p. 215]{TitsOber}, le groupe $G$ est supposé simplement connexe et sa forme déployée $H$ simple, ce qui s'approche plus de l'hypothèse régnante de la version précédente de cet article : d'un côté, cela nous permet de supposer $e \leq 3$ et nous ramène à travailler avec un nombre fini de familles de diagrammes de Dynkin irréductibles ; de l'autre côté, presque tous les affirmations valables sous l'hyp. \ref{hypothese groupe dans la fibre generique} seront typiquement déduites de celles-ci.
\end{rem}
Ensuite, on se penche sur la construction d'un certain schéma en groupes $\underline{G}$ sur $\GG_{m, \ZZ}=\spec \ZZ[t^{\pm 1}]$ prolongeant $G$ à la suite de Tits, cf. \cite[p. 217]{TitsOber}. Son observation cruciale consiste à remarquer que le tore maximal $T$ et les sous-groupes radiciels $U_a$ pour $a \in \Phi_G^{\on{nd}}$ admettent des $\ZZ[t^{\pm 1}]$-modèles naturels.

\begin{defn}[Tits]\label{definition drs groupe de tits}
	Soit \begin{equation}
	\underline{T}:=\prod_{j \in J}\Res_{\ZZ[t^{\pm 1/d_j}]/\ZZ[t^{\pm 1}]}\GG_{m,\ZZ[t^{\pm 1/d_j}]}
	\end{equation} le $\ZZ[t^{\pm 1}]$-modèle en groupes de $T$ défini par transport de structure le long de l'isomorphisme (\ref{equation tore decompose}). Pour les racines $a \in \Phi^{\on{nd,nm}}$ aux rayons singuliers, soient 
	\begin{equation}
	\underline{U_a}:=\Res_{\ZZ[t_a^{\pm 1}]/\ZZ[t^{\pm 1}]}\GG_{a,\ZZ[t_a^{\pm 1}]}
	\end{equation} les $\ZZ[t^{\pm 1}]$-modèles en groupes de $U_a$ définis par transport de structure le long de $x_a$, voir (\ref{equation isomorphisme groupe radiciel singulier}). Pour chaque racine multipliable $a \in \Phi_G^{\on{m}}$, soit 
		\begin{equation}
	\underline{U_a}:=\Res_{\ZZ[t_{2a}^{\pm 1}]/\ZZ[t^{\pm 1}]}\GG_{p,\ZZ[t_a^{\pm 1}]/\ZZ[t_{2a}^{\pm 1}]}
	\end{equation} le $\ZZ[t^{\pm 1}]$-modèle en groupes de $U_a$ définis par transport de structure le long de $x_a$, voir (\ref{equation isomorphisme groupe radiciel pluriel}), où l'on note $\GG_{p,\ZZ[t_a^{\pm 1}]/\ZZ[t_{2a}^{\pm 1}]}$ le $\ZZ[t_{a}^{\pm 1}]$-schéma $\Res_{\ZZ[t_a^{\pm 1}]/\ZZ[t_{2a}^{\pm 1}]}\AAA^1_{\ZZ[t_a^{\pm 1}]} \times  t_a\AAA^1_{\ZZ[t_{2a}^{\pm 1}]}$ muni de la loi de groupe induite par (\ref{equation loi de groupe pluriel}).
\end{defn}

Observons que $\underline{T}$ ne dépend pas du choix de l'isomorphisme de (\ref{equation tore decompose}), en vertu de son identification au modèle de Néron connexe de $T$ sur $\ZZ[t^{\pm 1}]$, cf. \cite[4.4.8]{BTII}, \cite[prop. 10.1.4]{BLR} et \cite[prop. 3.8]{LouDiss}. Nous montrons ci-dessous que la donnée $(\underline{T},\underline{U_a})_{a \in \Phi_G^{\on{nd}}}$ est radicielle schématique, cf. \cite[p. 217]{TitsOber}.

\begin{propn}[Tits]\label{proposition drs groupe de tits}
	La donnée $(\underline{T},\underline{U_a})_{a \in \Phi_G^{\on{nd}}}$ satisfait aux axiomes (DRS0-3) de la déf. \ref{definition donnees radicielles schematiques}.
\end{propn}

\begin{proof}
	Ceci découle essentiellement du fait que l'application conjugaison \begin{equation}T \times U_a \rightarrow U_a, \end{equation} l'application commutateur \begin{equation} U_a \times U_b \rightarrow \prod_{c\in ]a,b[} U_c, \end{equation} et l'application rationnelle d'échange \begin{equation} U_a \times U_{-a} \dashrightarrow U_{-a} \times T \times U_a\end{equation} préservent leur sens avec $\ZZ[t^{\pm 1}]$-coefficients, c'est-à-dire que se prolongent nettement en des morphismes rationnels entre les produits des modèles $\underline{T}$ et $\underline{U_a}$ pour toute racine $a \in \Phi_G^{\on{nd}}$ non divisible. L'effet de la similitude n'est pas à craindre, restant visible dans la fibre générique.
	
	En effet, pour l'action de $\underline{T}$ sur $\underline{U_a}$, on peut se ramener, par naturalité des modèles de Néron connexes, au cas où $G$ est adjoint de rang relatif $1$. Alors, on voit par construction que $\underline{T}$ s'identifie à $\Res_{\ZZ[t_a^{\pm 1}]/\ZZ[t^{\pm 1}]}\GG_m$ par un certain morphisme $a$, choisi de telle sorte que les formules
	\begin{equation} a(s)\cdot x_a(r)\sim x_a(sr) \end{equation}\begin{equation}
	a(s) \cdot x_a(u,v) \sim x_a(su,s\sigma(s)v)\end{equation}
	décrivent l'opération de $\underline{T}$ sur $\underline{U_a}$ à similitude près, voir aussi \cite[prop. 4.4.19]{BTII}.
	
	La règle (DRS 2) concerne les applications commutateur \begin{equation}\underline{U_a} \times \underline{U_b} \rightarrow \prod_{c\in ]a,b[} \underline{U_c}\end{equation} pour des racines non divisibles $a,b$ engendrant des $\QQ$-droites différentes. Leurs fibres génériques s'écrivent à similitude près comme dans \cite[A.6]{BTII} et on voit aisément que les applications y décrites se prolongent aux coefficients entiers. Pour être exhaustif, nous allons écrire les applications commutateur lorsque l'une des racines $a$, $b$ ou $c$ est multipliable, vu qu'elles ont été affectées par la modification de Tits, voir la déf. \ref{definition quasi-systeme de Tits} :
	\begin{equation}\label{equation commutateur pluriel un}
	[x_a(r), x_b(r')] \sim x_{a/2+b/2}(0,\sigma(rr')-rr') ; \end{equation}\begin{equation}\label{equation commutateur pluriel deux}
	[x_a(u,v), x_b(u',v')] \sim x_{a+b}(2uu') ; \end{equation}\begin{equation}\label{equation commutateur pluriel trois}
	[x_a(u,v), x_b(r)] \sim x_{a+b}(ru,N(r)\sigma(v))x_{2a+b}(s(u,v)r),
	\end{equation}
	où $s(u,v)$ désigne la somme quadratique $N(u)+v$.
	En particulier, les sous-groupes radiciels multipliables commutent les uns avec les autres en caractéristique $2$, ce qui pourrait surprendre le lecteur, mais qui répondre à nos attentes, compte tenu de la classification des groupes pseudo-réductifs, voir l'ex. \ref{exemple absolument non reduit} et le th. \ref{proposition fibre generique quasi-reductive}.
	
	Enfin, pour que l'axiome (DRS 3) soit rempli, il faut se donner une certaine section globale \begin{equation}d_a:\underline{U_a}\times \underline{U_{-a}} \rightarrow \AAA^1_{\ZZ[t^{\pm 1}]}.\end{equation} Posant $\underline{W_a}$ l'ouvert distingué de $\underline{U_a}\times \underline{U_{-a}}$ où $d_a$ ne s'annule pas, qui doit contenir les fermés $\underline{U_{\pm a}}$, on a besoin encore d'une application \begin{equation}\beta_a: \underline{W_a} \rightarrow \underline{U_{-a}}\times \underline{T} \times \underline{U_{a}}.\end{equation} Dans la fibre générique, on trouve l'écriture de ces applications à similitude près dans \cite[4.1.6, 4.1.12]{BTII}. 
	
	Enregistrons toutefois leurs prolongements entiers pour référence future. Si $a \in \Phi_G^{\on{nd,nm}}$ n'est ni divisible ni multipliable, alors
	\begin{equation}\label{equation diviseur sl2} d_a(x_a(r),x_{-a}(r'))=\on{Norme}_{\ZZ[t_a^{\pm 1}]/\ZZ[t^{\pm 1}]}(1-rr')\end{equation}
	et, en posant $t(r,r')=1-rr'$ et $\underline{W_a}$ l'ouvert spécial déterminé par $d_a$, on en tire
	\begin{equation}\label{equation echange sl2} \beta_a(x_a(r),x_{-a}(r'))=(x_{-a}(r'/t(r,r')),a^{\vee}(t(r,r')),x_a(r/t(r,r'))),\end{equation}
	où $a^{\vee}$ désigne l'application composée d'un certain isomorphisme $\Res_{\ZZ[t_a^{\pm 1}]/\ZZ[t^{\pm 1}]}\GG_m \simeq \underline{T}^{\on{sc},a} $ unique à conjugaison près avec le morphisme naturel $\underline{T}^{\on{sc},a} \rightarrow \underline{T}$. Si $a \in \Phi_G^{\on{m}}$ est multipliable, il s'avère par des calculs impliquant la modification de Tits que l'on a
	\begin{equation}\label{equation diviseur su3} d_a(x_a(u,v),x_{-a}(u',v'))=\on{Norme}_{\ZZ[t_{a}^{\pm 1}]/\ZZ[t^{\pm 1}]}t(u,v,u',v'),\end{equation}
	où l'on note $t(u,v,u',v'):=1-2\sigma(u)u'+s(u,v)s(u',v')$, et on en déduit aussi que la première coordonnée de $\beta_a(x_a(u,v),x_{-a}(u',v'))$ s'écrit comme le monstre suivant
	\begin{equation}\label{equation echange su3} x_{-a}\Big(\frac{u'-us(u',v')}{t(u,v,u',v')},\frac{s(u',v')}{t(u,v,u',v')}-N\Big(\frac{u'-us(u',v')}{t(u,v,u',v')}\Big)\Big),
	\end{equation} mais sans aucun facteur de $2$ non souhaitable. On voit de même que sa dernière coordonnée est pareille à la première, quitte à échanger quelques primes et à conjuguer, et que la deuxième est égale à $(2a)^{\vee}(t(u,v,u',v'))$, où l'on note $(2a)^\vee$ l'application\footnote{L'apparition de ce facteur de $2$ est liée au fait que cela induit le copoids $(2a)^\vee$ sur $\underline{S}$ et non pas $a^\vee$.} composée $\Res_{\ZZ[t_a^{\pm 1}]/\ZZ[t^{\pm 1}]}\GG_m \simeq \underline{T}^{\on{sc},a}\rightarrow \underline{T}$.	
\end{proof}

Grâce à cette proposition, nous pouvons appliquer le th. \ref{theoreme d'existence et unicite des groupes attaches aux drs} pour obtenir un $\ZZ[t^{\pm 1}]$-groupe lisse, connexe et séparé $\underline{G}$.

\begin{defn}[Tits]\label{definition groupe de Tits}
	Le groupe de Tits $\underline{G}$ est le $\ZZ[t^{\pm 1}]$-modèle en groupes séparé, lisse et connexe de $G$ qui découle du th. \ref{theoreme d'existence et unicite des groupes attaches aux drs} appliqué à la donnée radicielle schématique $(\underline{T},\underline{U_a})_{a\in \Phi_G^{\on{nd}}}$.
\end{defn}

Remarquons tout d'abord quelques propriétés importantes :

\begin{rem}\label{remarque faits basiques groupe de tits}Tous les faits qui suivent sans démonstration ultérieure résultent d'une comparaison des grosses cellules correspondantes, vu l'unicité des solutions aux lois birationnelles :
	\begin{enumerate} 
		\item Si $G$ est la restriction des scalaires d'un groupe déployé $H'$ le long de $\QQ(t^{1/d})/\QQ(t)$, alors $\underline{G}$ en est de même la restriction des scalaires $\Res_{\ZZ[t^{\pm 1/d}]/\ZZ[t^{\pm 1}]}H'$.
		\item Si l'on inverse le degré de ramification $e$, on obtient un $\ZZ[1/e, t^{\pm 1}]$-groupe réductif $\underline{G}[1/e]$ donné par les $\Gamma$-invariants de $\Res_{\ZZ[1/e, t^{\pm 1/d}]/\ZZ[1/e, t^{\pm 1}]}\underline{H}$, voir \cite[3.c.1]{PZh}.
		\item On verra ci-dessous, cf. prop. \ref{proposition fibre generique quasi-reductive}, que la fibre générique $G_{\eta_p}:=\underline{G}\otimes \FF_p(t)$ de la réduction modulo $p$ est toujours pseudo-réductive et quasi-déployée à système de racines $\Phi_G$. Si $\Phi_H$ est irréductible et $\Gamma$ n'y opère pas trivialement (donc on peut choisir $e=2,3$), alors le th. \ref{theoreme classification cgp} entraîne que $G_{\eta_e}^{\on{sc}}$ est soit exotique basique, soit barcelonais basique. 
		\item Nous verrons plus tard que $\underline{G}$ est même affine et non pas seulement quasi-affine, voir la prop. \ref{proposition enveloppe affine}.
	\end{enumerate}
\end{rem}

Notons $G_{\eta_p}$ la fibre générique $\underline{G}\otimes \FF_p(t)$ de la réduction $G_p:=\underline{G}\otimes \FF_p[t^{\pm 1}]$ modulo $p$ de $\underline{G}$.

\begin{propn}\label{proposition fibre generique quasi-reductive}
	Le $\FF_p(t)$-groupe lisse, affine et connexe $G_{\eta_p}$ est pseudo-réductif, quasi-déployé et son système de racines relatif s'identifie à $\Phi_G$. De plus, les $\eta_p$-fibres des isomorphismes $x_a$ définissent un quasi-système de Tits.
\end{propn}

La notion de quasi-système de Tits n'a été pas définie pour les groupes pseudo-réductifs exotiques ou barcelonais, la raison étant que pour ces groupes il n'a aucun sens a priori de distinguer entre quasi-systèmes de Tits et de Chevalley, voir \cite[déf. 2.4]{LouBT}. On pourrait alors prendre l'énoncé en petites caractéristiques comme définition même.

\begin{proof}
	Par construction avec les lois birationnelles, l'on a une suite de sous-groupes fermés $\underline{S} \subset \underline{T} \subset \underline{B} \subset \underline{G}$ lisses et connexes. On voit que les fibrés vectoriels $\on{Lie}\underline{T}$ resp. $\on{Lie}\underline{U_a}$ ne sont que les espaces propres à poids nul resp. non nul pour l'action de $\underline{S}$ sur $\on{Lie}\,\underline{G}$. En particulier, cela caractérise uniquement leurs fibres $S_{\eta_p} \subset T_{\eta_p} \subset B_{\eta_p} \subset G_{\eta_p}$ comme sous-groupes fermés de $G_{\eta_p}$ par le th. \ref{theoreme borel-tits structure} -  l'assertion cruciale reste valable plus généralement pour les groupes lisses, affines et connexes.
	
	Il suffit donc de montrer, en vertu de \cite[1.1.11]{BTII} et de \cite[th. C.2.29]{CGP}, que $T_{\eta_p}$ est pseudo-réductif (ce qui est manifestement le cas) et que les sous-groupes $G_{a,\eta_p}$ engendrés par $C_{a,\eta_p}$ sont pseudo-réductifs. Autrement dit, on peut supposer $G$ de rang relatif $1$, et même simplement connexe, en faisant usage des $z$-extensions. En effet, comme $G$ résulte de $G^{\on{sc}}$ par modification centrale, cf. (\ref{equation z-extension canonique tordu}), alors les lois birationnelles fournissent un isomorphisme de $\ZZ[t^{\pm 1}]$-groupes : 
	\begin{equation}\underline{G}:=(\underline{T}\ltimes \underline{G}^{\on{sc}})/\underline{T}^{\on{sc}},\end{equation} 
	d'où le fait que $G_{\eta_p}$ soit un modifié central de $G^{\on{sc}}_{\eta_p}$ par le groupe commutatif pseudo-réductif $T_{\eta_p}$, ramenant l'assertion au cas où $G$ est simplement connexe. 
	
	Supposons que $G=\Res_{\QQ(t_a)/\QQ(t)}\on{SL}_2$ ou que $p \neq 2$ et $G=\Res_{\QQ(t_{2a})/\QQ(t)}\on{SU}_{3,\QQ(t_a)/\QQ(t_{2a})}$. Alors, on sait déjà que $G_{\eta_p}=\Res_{\FF_p(t_a)/\FF_p(t)}\on{SL}_2$ ou que $G_{\eta_p}=\Res_{\FF_p(t_{2a})/\FF_p(t)}\on{SU}_{3,\FF_p(t_{2a})/\FF_p(t_a)}$ sont pseudo-réductifs, voir la rem. \ref{remarque faits basiques groupe de tits}. Si $p=2$ et $G=\Res_{\QQ(t_{2a})/\QQ(t)}\on{SU}_{3,\QQ(t_a)/\QQ(t_{2a})}$, alors (\ref{equation diviseur su3}) et (\ref{equation echange su3}) se simplifient en caractéristique $2$ comme suit :
	\begin{equation}1+s(u,v)s(u',v')\end{equation}
	et
	\begin{equation}\frac{s(u',v')}{1+s(u,v)s(u',v')},\end{equation}
	ce qui coïncide avec (\ref{equation diviseur sl2}) et (\ref{equation echange sl2}) pour le couple $(s(u,v),s(u',v'))$. Par suite, notre loi birationnelle dans la grosse cellule $C_{\eta_2} \subset G_{\eta_2}$ coïncide avec celle du groupe barcelonais de rang relatif $1$ et de dimension $\dim G$, voir ex. \ref{exemple absolument non reduit}.
	
	Enfin, pour que les isomorphismes donnés $x_a$ constituent un quasi-système de Tits, il faut essentiellement remarquer que les $m_b$ permutent les classes de similitude des $x_a$ pour toutes racines non divisibles $a,b \in \Phi_G^{\on{nd}}$. Cela résulte par densité de l'assertion correspondante dans la fibre générique, c'est-à-dire pour $G$, donc remplie grâce au lem. \ref{lemme coordonnees de la modification de Tits}.
\end{proof}

\begin{rem}
	Être pseudo-réductif n'est pas une condition géométrique et donc n'admet aucune formulation relative en familles qui soit pratique. En particulier, on peut montrer facilement que, si $e$ est un choix minimal d'entier tel que l'extension $\QQ(\zeta_e,t^{1/e})/\QQ(t)$ déploie $G$, alors $G_{\eta_p}$ n'est pas réductif pour tout $p$ divisant $e$, et tous les autres fibres du $\FF_p[t^{\pm 1}]$-groupe $G_p$ ont un radical unipotent non trivial sur $\FF_p$. Cependant, nous verrons dans la suite, cf. prop. \ref{lemme fibres grassmannienne globale}, que $G_p$ possède une structure arithmétique assez remarquable, étant un modèle parahorique spécial de sa fibre générique. 
\end{rem}

Parfois il sera plus commode d'avoir une description plus concise et explicite du groupe de Tits $\underline{G}$ en fonction de $\Res_{\ZZ[\zeta_e,t^{\pm 1/e}]/\ZZ[t^{\pm 1}]} H$, où $H$ est vu comme $\ZZ$-groupe épinglé au moyen des épinglages $y_\alpha$ du système de Tits associé.

\begin{propn}\label{proposition plongement restrictions des scalaires groupe de tits}
	On a une immersion fermée naturelle $\underline{G}\rightarrow \Res_{\ZZ[\zeta_e,t^{\pm 1/e}]/\ZZ[t^{\pm 1}]} H$ prolongeant l'application générique. En particulier, $\underline{G}$ est affine sur $\GG_{m,\ZZ}$.
\end{propn}

Commentons également un truc qui s'avère très utile dans certaines situations. Si on se contente de construire le schéma en groupes $\underline{G}$ au-dessus de $\ZZ[\zeta_e,t^{\pm 1}]$, alors l'extension de déploiement (générique) $\ZZ[\zeta_e,t^{\pm 1}] \rightarrow \ZZ[\zeta_e,t^{\pm 1/e}]$ devient cyclique, et on peut substituer le groupe cyclique $\Gamma_0$ à son sur-groupe $\Gamma$. Par la même argumentation ci-dessous, on arrive à réaliser $\underline{G}\otimes \ZZ[\zeta_e]$ comme sous-groupe fermé de $\Res_{\ZZ[\zeta_e, t^{\pm 1/e}]/\ZZ[\zeta_e,t^{\pm 1}]} H$ contenu dans les $\Gamma_0$-invariants.

\begin{proof}
	Montrons tout d'abord que les applications restreintes aux tores \begin{equation}\underline{T} \rightarrow \Res_{\ZZ[\zeta_e,t^{\pm 1/e}]/\ZZ[t^{\pm 1}]} T_H,\end{equation}  resp. aux sous-groupes radiciels \begin{equation}\underline{U_a} \rightarrow \prod_{\gamma\alpha}\Res_{\ZZ[\zeta_e,t^{\pm 1/e}]/\ZZ[t^{\pm 1}]}  U_{H,\gamma\alpha}\end{equation} si $a\in \Phi_G$ n'est ni divisible ni multipliable, et \begin{equation}\underline{U_a} \rightarrow  \prod_{\{\gamma\alpha,\gamma\beta\}} \Res_{\ZZ[\zeta_e,t^{\pm 1/e}]/\ZZ[t^{\pm 1}]}U_{H,\gamma\alpha} U_{H,\gamma(\alpha+\beta)}U_{H,\gamma\beta}\end{equation} si $a \in \Phi_G^{\on{m}}$ est multipliable, sont des immersions fermées. Le cas du centralisateur résulte de ce que les modèles de Néron connexes des tores induits sont leurs modèles canoniques, cf. \cite[4.4.6, 4.4.8]{BTII}. À son tour, le cas des sous-groupes radiciels est une conséquence des calculs supprimés du lem. \ref{lemme coordonnees de la modification de Tits}.
	
	Par suite, la \cite[prop. 2.2.10]{BTII} entraîne que l'application naturelle \begin{equation}\underline{G}\rightarrow \Res_{\ZZ[\zeta_e,t^{1/e}]/\ZZ[t]} H\end{equation} est une immersion localement fermée. Soit $\widehat{\ZZ[t]}_{(p)}$ le séparé complété du localisé de $\ZZ[t]$ dans l'idéal premier $(p)$ : il s'agit d'un anneau de valuation discrète à uniformisante $p$ et à corps résiduel $\FF_p(t)$. L'image schématique du groupe parahorique spécial $\underline{G} \otimes \widehat{\ZZ[t]}_{(p)}$ dans la restriction des scalaires correspondante de $ H$ est un sous-groupe fermé, affine et lisse, dont la composante neutre reste égale, d'où l'égalité par \cite[prop. 4.6.21]{BTII}. 
	
	En particulier, la propriété d'être une immersion fermée se propage à un ouvert cofini $U$ de $\spec \ZZ[t^{\pm 1}]$. Pour conclure, il ne reste qu'à appliquer le lemme ci-dessous concernant l'existence d'une opération tordue fonctorielle de $\GG_{m,\ZZ[t^{\pm 1}]}$ sur $\underline{G}$ et sur $\Res_{\ZZ[\zeta_et^{1/e}]/\ZZ[t]} H$. En effet, cette action relève celle de la $e$-ième puissance sur $\spec \ZZ[t^{\pm 1}]$, dont les translatés de $U$ recouvrent tout le spectre.
\end{proof}

Ce lemme a joué un rôle important dans la fin de la démonstration précédente.

\begin{lem}\label{lemme action de rotation t unite}
	L'action $e$-ième puissance de $\GG_{m,\ZZ}$ sur soi-même se relève en une action tordue sur le $\GG_{m,\ZZ}$-groupe $\underline{G}$.
\end{lem}

\begin{proof}
	Voyons que la $e$-puissance définit une opération $\GG_{m,\ZZ} \times \GG_{m,\ZZ} \rightarrow \GG_{m,\ZZ}$ donnée par $(r,t)\mapsto r^et$. Considérons maintenant le $\GG_{m,\ZZ}$-groupe $\Res_{\ZZ[t^{\pm 1/d}]/\ZZ[t^{\pm 1}]}\GG_{a}$ classifiant des éléments $a \in R[t^{1/d}]$. Alors, la correspondance $(r,t,a) \mapsto (r^et,a)$ en détermine un relèvement de l'opération puissance. La même écriture marche aussi pour les restrictions des scalaires du groupe multiplicatif $\GG_m$ et du groupe unipotent pluriel $\GG_p$. Finalement, nous laissons au lecteur le soin de vérifier que cette action préserve les morphismes des axiomes des données radicielles schématiques -- voir la déf. \ref{definition donnees radicielles schematiques} et les formules de la preuve de la prop. \ref{proposition drs groupe de tits} --  d'où son prolongement à $\underline{G}$ par fonctorialité du th. \ref{theoreme existence de solutions}.
\end{proof}

Dans la version précédente de l'article, nous avons démontré par un calcul presque direct que, au cas où $G$ est absolument simple et simplement connexe à système de racines $\Phi_G$ réduit, la fibre générique $G_{\eta_e}$ était pseudo-réductive exotique basique. Voici le raisonnement :

\begin{exmp}
Supposons que $G$ est absolument simple, simplement connexe et non déployé à système de racines $\Phi_G$ réduit et que $e\leq 3$. Grâce à la prop. \ref{proposition plongement restrictions des scalaires groupe de tits}, on a une immersion fermée \begin{equation}
G_{\eta_e}\hookrightarrow(\Res_{\FF_e(t^{1/e})/\FF_e(t)} H)^{\Gamma_0}.
\end{equation} Vu que la partie galoisienne de l'opération
de $\Gamma_0$ sur cette restriction des scalaires s'évapore tout simplement en caractéristique $e$, le membre de droite s'écrit comme $\Res_{\FF_e(t^{1/e})/\FF_e(t)} H^{\Gamma_0}$, où $\Gamma_0$ opère par $\FF_e$-automorphismes de Dynkin sur le $\FF_e$-groupe épinglé $H$. 

D'après Steinberg, cf. \cite[th. 8.2]{St2} et \cite[prop. 5.1]{Hai}, le groupe algébrique $H^{\Gamma_0}$ est lisse, connexe, semi-simple, absolument presque simple et simplement connexe déployé à système de racines $\Phi_G$ ayant $T^{\Gamma_0}$ comme $\FF_e$-tore maximal. Ce tore se décompose en produit de groupes multiplicatifs $\GG_m$ numérotés par les racines simples $a \in \Delta_G$ et donnés par les formules \begin{equation}a^{\vee}:\GG_m \rightarrow T^{\Gamma_0}, \end{equation} \begin{equation}r \mapsto \prod_{\gamma\alpha} (\gamma\alpha)^{\vee}(r).\end{equation} Ses groupes radiciels numérotés par $a \in \Phi_G$ sont égaux à $(\prod_{\gamma\alpha} U_{\gamma \alpha})^{\Gamma_0}$ et donnés par les épinglages \begin{equation}z_a: \GG_a \rightarrow H^{\Gamma_0},\end{equation} \begin{equation}r \mapsto \prod_{\gamma\alpha} y_{\gamma\alpha}(r),\end{equation} qui forment un $\FF_e$-système de Chevalley, comme on vérifie aisément. Alors, il suffit de remarquer que ces applications sont naturellement compatibles avec celles du quasi-système de Tits de $G$, donc le $\FF_e(t)$-groupe $G_{\eta_e}$ s'identifie au seul $\FF_e(t)$-groupe pseudo-réductif exotique basique à système de racines $\Phi_G$ de l'ex. \ref{exemple exotique basique}, compte tenu de l'unicité des solutions aux lois birationnelles, cf. th. \ref{theoreme existence de solutions}.
\end{exmp}
\subsection{Modèles en groupes immobiliers sur $\AAA_{\ZZ}^1$} \label{subsection groupes immobiliers modelant tits}

Dans ce numéro, nous allons introduire les $\AAA_{\ZZ}^1$-modèles en groupes immobiliers $\underline{\Gscr_{ \Omega}}$, voir la déf. \ref{definition modeles immobiliers du groupe de tits}, des groupes de Tits $\underline{G}$, voir la déf. \ref{definition groupe de Tits}. Au-dessus de $\ZZ[1/e]$, ces groupes avaient été découverts par Pappas--Rapoport dans le cas simplement connexe, voir \cite[7.a, (7.2)]{PR}, puis définis plus généralement par Pappas--Zhu, voir \cite[th. 4.1]{PZh}. Notre gros résultat est celui de l'affinité du modèle immobilier $\underline{\Gscr_{ \Omega}}$, voir le th. \ref{theoreme affinite modeles immobiliers}, lequel résulte d'une approche radicalement différente de celle de Pappas--Rapoport--Zhu, en faisant justement usage des propriétés combinatoires élémentaires, cf. prop. \ref{proposition combinatoire caracteristiques differentes}, et de la lissité de l'enveloppe affine, cf. prop. \ref{proposition enveloppe affine}.

Avant de définir les données radicielles schématiques, on a besoin du lemme suivant sur la combinatoire des $\FF_p\rpot{t}$-groupes déduits de $\underline{G}$ par changement de base, en variant la caractéristique $p$.

\begin{propn}\label{proposition combinatoire caracteristiques differentes}
	Il existent des bijections canoniques entre les différents appartements $\Ascr(G_{\eta_p},S_{\eta_p},\FF_p\rpot{t})$, pour toute caractéristique $p \in \PP \cup \{0\}$, invariantes par l'action des groupes d'Iwahori--Weyl.
\end{propn}

\begin{proof}
	Identifions d'abord les groupes d'Iwahori--Weyl $\widetilde{W}_p$ de $G_{\eta_p}$ par rapport à $S_{\eta_p}$ en caractéristiques différentes. Soit $\underline{N}$ le normalisateur de $\underline{S}$ dans $\underline{G}$, qui en est un $\ZZ[t^{\pm 1}]$-sous-groupe fermé affine et lisse, comme l'affirme \cite[exp. XI, cor. 5.3 bis]{SGA3}. Posons \begin{equation}\label{equation modele de neron polynomes}\underline{\Tscr}:=\prod_{j\in J}\Res_{\ZZ[t^{1/d_j}]/\ZZ[t]}\GG_{m} \end{equation} le modèle de Néron connexe de $T$ sur $\ZZ[t]$. On montre que l'homomorphisme de groupes naturel \begin{equation}\label{equation isomorphisme groupe d'iwahori-weyl caracteristiques differentes}\underline{N}(\ZZ[t^{\pm 1}])/\underline{\Tscr}(\ZZ[t]) \rightarrow \widetilde{W}_p:=N_{\eta_p}(\FF_p\rpot{t})/\Tscr_p(\FF_p\pot{t})\end{equation} est toujours bijectif, ce qui fournira par zig-zag l'isomorphisme cherché. Le lemme de la serpente nous permet de traiter distinctement les sous-groupes des translations \begin{equation}\label{equation isomorphisme translations d'iwahori-weyl}\underline{T}(\ZZ[t^{\pm 1}])/\underline{\Tscr}(\ZZ[t]) \rightarrow T_{\eta_p}(\FF_p\rpot{t})/\Tscr_p(\FF_p\pot{t})\end{equation} et les quotients des transformations vectorielles \begin{equation}\label{equation isomorphisme parties vectorielles}\underline{N}(\ZZ[t^{\pm 1}])/\underline{T}(\ZZ[t^{\pm 1}])\rightarrow W(G_{\eta_p},S_{\eta_p}).\end{equation} On voit que le cas des translations est trivial, en calculant explicitement avec la décomposition (\ref{equation modele de neron polynomes}), ce qui revient au même de dire que la valuation $t$-adique est indépendante des coefficients. 
	
	Maintenant, on se penche sur la partie des transformations vectorielles. Vu que $\underline{T} \rightarrow \underline{N}$ est une immersion fermée, l'application (\ref{equation isomorphisme parties vectorielles}) est injective en caractéristique nulle $p=0$. D'autre part, les éléments $m_a$ du normalisateur qu'on avait évoqués à (\ref{equation elements m du normalisateur singulier tits}) et (\ref{equation elements m du normalisateur pluriel tits}) sont définis sur $\ZZ[t^{\pm 1}]$, voir la preuve de la prop. \ref{proposition drs groupe de tits}, et y induisent les réflexions $s_a$ du groupe de Weyl de $\Phi_G$. Par suite, les applications de (\ref{equation isomorphisme parties vectorielles}) sont même bijectives pour tout premier $p$, comme il suit d'une comparaison d'ordres. 
	
	 Maintenant, on peut définir une valuation $\varphi_p \in \Ascr(G_{\eta_p}, S_{\eta_p}, \FF_p\rpot{t})$, pour chaque caractéristique $p \in \PP \cup \{0\}$, en termes du quasi-système de Tits naturel de $G_{\eta_p}$, cf. prop. \ref{proposition fibre generique quasi-reductive}, de la manière suivante : 
	\begin{equation} \varphi_{p,a}(x_a(r))=\omega_p(r), \end{equation}\begin{equation}
	\varphi_{p,a}(x_a(u,v))=\omega_p(u^2+v)/2,
	\end{equation} où la valuation additive $t$-adique $\omega_p: \FF_p\rpot{t}^{\on{alg}}\rightarrow \RR$ est déterminée uniquement par la contrainte $\omega_p(t)=1$, comp. avec la valuation de \cite[déf. 3.1]{LouBT} donnée par (\ref{equation valuation singulier}), (\ref{equation valuation pluriel unitaire}) et (\ref{equation valuation pluriel barcelonais}). 
	
	La prochaine étape consiste à identifier les appartements $\Ascr(G_{\eta_p},S_{\eta_p},\FF_p\rpot{t})$ de façon équivariante. En choisissant $\varphi_p$ pour origine de cet appartement, il ne devient autre que le sous-$\RR$-espace vectoriel de $X_*(S_{\eta_p}) \otimes_\ZZ{\RR}$ engendré par les coracines $\Phi_G^\vee$, d'où en particulier des isomorphismes \begin{equation}\label{equation identification des appartements}\Ascr(G_{\eta_p},S_{\eta_p},\FF_p\rpot{t}) \simeq \Ascr(G_{\eta_q},S_{\eta_q},\FF_q\rpot{t})\end{equation} pour toute paire de caractéristiques $p,q \in\PP \cup\{0\}$. Il est aisé de voir que cette identification fait correspondre les racines affines de chaque appartement ainsi que leurs échelonnages, voir \cite[déf. 3.3]{LouBT}.
		
	 Démontrons enfin l'équivariance de l'isomorphisme donné par (\ref{equation identification des appartements}), qu'il suffit de la vérifier pour une partie génératrice. Pour les translations l'assertion devient évidente, vu que $z \in T_{\eta_p}(\FF_p\rpot{t})$ opère par la translation par le vecteur $\nu(z)$ tel que $a(\nu(z))=-\omega_p(a(z)) $, voir \cite[4.2.7]{BTII}, et que la valuation $t$-adique est indépendante de la caractéristique, voir (\ref{equation isomorphisme translations d'iwahori-weyl}). Les classes des $m_a \in \underline{N}(\ZZ[t^{\pm 1}])$ à leur tour opèrent par des réflexions correspondantes $s_a$ du groupe de Weyl fixant l'origine.
	\end{proof}

	\begin{rem} \label{remarque autres changements de base groupe de tits}
		Il y a d'autres changements de base intéressants qu'on avait étudiés plus soigneusement dans la version précédente de cet article.
		\begin{enumerate}
			\item Soient $e=2,3$, $G=G^{\on{sc}}$ absolument simple et simplement connexe à système de racines $\Phi_G$ réduit. Pour tout $a \in \overline{\FF}_e^{\times}$, on obtient d'isomorphismes naturels \begin{equation}\Ascr(G_{\eta_e}, S_{\eta_e}, \overline{\FF}_e\rpot{t-a}) \simeq \Ascr(G_{\eta_e}, S_{\eta_e}, \FF_e\rpot{t})\end{equation} d'appartements, bien que d'autres structures combinatoires comme les groupes d'Iwahori--Weyl. Cela découle du fait que $G_{\eta_e}$ vienne muni d'un quasi-système de Tits global et que la ramification de $\AAA^{1/e}_{\FF_e}\rightarrow \AAA^1_{\FF_e}$ soit constante. Toutefois, si l'on essayait de faire le même jeu avec les groupes barcelonais, on se serait heurt à ce que la valuation de $t^{1/2}$ ne soit pas constante parmi les diverses places de $\AAA^1_{\FF_e}$, d'où le besoin de remplacer l'origine de $\Ascr(G_{\eta_e}, S_{\eta_e}, \FF_e\rpot{t})$ à système de racines résiduel $\Phi_G^{\on{nd}}$ par un point spécial avec $\Phi_G^{\on{nm}}$ pour système de racines résiduel.
			\item Soient $\Oh$ un anneau de valuation discrète de caractéristique mixte $(0,p)$, $\varpi$ une uniformisante de $\Oh$ et $K$ son corps de fraction. Alors, l'on tire aussi un isomorphisme équivariant
			\begin{equation}\Ascr(\underline{G}\otimes_{\ZZ[t]} K, \underline{S}\otimes_{\ZZ[t]} K,K) \simeq \Ascr(G_{\eta_p}, S_{\eta_p}, \FF_p\rpot{t})\end{equation}
			d'appartements, où $\ZZ[t]\rightarrow \Oh$ est le seul homomorphisme qui applique $t$ sur $\varpi$; en particulier, l'hypothèse de ramification modérée dans \cite[4.a.2]{PZh} est inutile. Ce changement de base de type diagonal prend un rôle proéminent dans la théorie des modèles locaux des variétés de Shimura, voir \cite[\S\S8-9]{PZh} et \cite[IV, \S4.1]{LouDiss}.
		\end{enumerate}
	\end{rem}

	Dorénavant, on fera plusieurs fois l'abus de langage d'écrire $\Ascr(\underline{G},\underline{S},\ZZ\rpot{t})$ lorsqu'on parle de tous les appartements $\Ascr(G_{\eta_p},S_{\eta_p},\FF_p\rpot{t})$ pour $p \in \PP\cup \{0\}$ en même temps, identifiés selon (\ref{equation identification des appartements}). L'information combinatoire qu'on vient de rassembler dans la prop. \ref{proposition combinatoire caracteristiques differentes} nous permettra ainsi d'introduire les $\AAA^1_{\ZZ}$-modèles immobiliers $\underline{\Gscr_{ \Omega}}$, voir la déf. \ref{definition modeles immobiliers du groupe de tits}, des $\GG_{m,\ZZ}$-groupes de Tits $\underline{G}$, où $\Omega$ désigne une partie bornée et non vide de $\Ascr(\underline{G},\underline{S},\ZZ\rpot{t})$. 
	
	On commence par rappeler la structure de certains modèles entiers lisses et connexes du groupe additif, à la suite de \cite[1.4.1]{BTII}.
	
	\begin{exmp}\label{exemple modele groupe additif}
		Soit $A$ un anneau noethérien intègre à corps de fraction $K$. A chaque idéal fraccionaire $\af \subset K$, on lui associe le modèle en groupes $\af \otimes \GG_a$ dont les points à valeurs dans $R$ sont donnés par $\af \otimes_A R$. Autrement dit, quitte à localiser pour que $\af=(a)$ soit principal, ceci n'est autre que le $A$-groupe $\GG_a$ regardé en tant que modèle de $\GG_{a,K}$ au moyen de $\GG_{a,K}\xrightarrow{\cdot a}\GG_{a,K}$.
	\end{exmp}

L'étape suivante consiste à définir les données radicielles schématiques, pour lequel il faut introduire un outil combinatoire additionnel. Soient $\Omega\subset \Ascr(\underline{G},\underline{S},\ZZ\rpot{t})$ une partie non vide et bornée et \begin{equation}\label{equation fonction associee a omega} f_{\Omega}: \Phi_G \rightarrow \RR,\end{equation}
\begin{equation}\label{equation valeurs de la fonction f associee a omega}
a \mapsto \on{min} \{ k \in \Gamma_a': a(\Omega)+k\geq 0\}
\end{equation} la fonction quasi-concave et optimale attachée à $\Omega$, voir \cite[4.6.26]{BTII}. Notons que $f_\Omega$ dépend aussi du choix de l'origine $\varphi$ de l'appartement $\Ascr(\underline{G},\underline{S},\ZZ\rpot{t})$ que nous avons introduite à la prop. \ref{proposition combinatoire caracteristiques differentes}, cf. (\ref{equation identification des appartements}).

\begin{defn}\label{definition drs modele entier groupe de tits}
	Soient $G$ un $\QQ(t)$-groupe remplissant l'hyp. \ref{hypothese groupe dans la fibre generique}, $\Omega$ une partie non vide et bornée de l'appartement $\Ascr(\underline{G},\underline{S},\ZZ\rpot{t})$ et $f_\Omega$ la fonction (\ref{equation fonction associee a omega}) définie par (\ref{equation valeurs de la fonction f associee a omega}). Définissons le modèle $\underline{\Tscr}$ de $\underline{T}$ par (\ref{equation modele de neron polynomes}). Quant aux sous-groupes radiciels, définissons-le par transport de structure le long des quasi-épinglage comme suit : \begin{equation}
	\underline{\Uscr_{a, \Omega}}:=\Res_{\ZZ[t_a]/\ZZ[t]} t^{f_\Omega(a)} \GG_{a,\ZZ[t_a]},
	\end{equation} lorsque $a\in \Phi_G^{\on{nd,nm}}$ n'est ni divisible ni multipliable, ou \begin{equation}
	\underline{\Uscr_{a, \Omega}}:=\Res_{\ZZ[t_{2a}]/\ZZ[t]} t^{f_{\Omega}(a),f_{\Omega}(2a)} \GG_{p,\ZZ[t_a]/\ZZ[t_{2a}]},
	\end{equation} où $t^{f_{\Omega}(a),f_{\Omega}(2a)} \GG_{p,\ZZ[t_a]/\ZZ[t_{2a}]}$ signifie le produit $t^{f_\Omega(a)} \GG_{a,\ZZ[t_{a}]} \times t^{f_\Omega(2a)}\GG_{a,\ZZ[t_{2a}]}$.
\end{defn}

Signalons que le $\ZZ[t]$-modèle en groupes $\underline{\Uscr_{a, \Omega}}$ est bien défini, car $t^{f_\Omega(a)}$ est une puissance entière de $t_a$. On a fait un abus de langage pareil dans l'écriture de $\underline{\Uscr_{a, \Omega}}$, lorsque $a\in \Phi_G^{\on{m}}$ est multipliable, vu que $t^{f_\Omega(2a)}$ n'appartient jamais à $\ZZ[t_{2a}]$, mais bien à son complément $t_a\ZZ[t_{2a}]$ dans $\ZZ[t_a]$. Il faudrait d'ailleurs montrer que la loi de groupe de $\GG_{p,\QQ(t_{a})/\QQ(t_{2a})}$ se prolonge à son modèle entier, ce qui sera vérifié implicitement pendant la preuve de la proposition suivante.

\begin{propn}\label{proposition drs modele immobilier groupe de tits}
Les schémas en groupes $(\underline{\Tscr}, \underline{\Uscr_{a, \Omega}})_{a \in \Phi_G^{\on{nd}}}$ forment une donnée radicielle schématique au-dessus de $\ZZ[t]$ prolongeant celle de la prop. \ref{proposition drs groupe de tits}.
\end{propn}

\begin{proof}
On peut certes démontrer que nos modèles entiers satisfont aux règles de la déf. \ref{definition donnees radicielles schematiques} par la même procédure de la prop. \ref{proposition drs groupe de tits}. Toutefois, il est plus vite d'observer le truc suivant à la Beauville--Laszlo \cite[th.]{BLDesc} : on a l'égalité \begin{equation}A=A[t^{-1}] \cap A\otimes \QQ\pot{t} \subset A\otimes \QQ\rpot{t}\end{equation} pour toute $\ZZ[t]$-algèbre plate $A$. Cela implique qu'on peut toujours construire des homomorphismes entre de tels anneaux par recollement au-dessus de $\ZZ[t^{\pm 1}]$ et puis $\QQ\pot{t}$. Par conséquent, on s'est ramené, d'un côté, au cas où $t$ est inversible, ce qui a été déjà réalisé dans la prop. \ref{proposition drs groupe de tits}, et d'autre côté, au cas des modèles immobiliers du $\QQ\rpot{t}$-groupe réductif et connexe $G_{\QQ\rpot{t}}$, grâce au cor. \ref{corollaire specialises modeles immobiliers en caracteristique p}.  Mais ce dernier cas est dû à Bruhat--Tits, cf. \cite[th. 4.5.4]{BTII}.
\end{proof}

On peut finalement définir les $\ZZ[t]$-modèles immobiliers $\underline{\Gscr_{ \Omega}}$ du groupe de Tits $\underline{G}$.

\begin{defn}\label{definition modeles immobiliers du groupe de tits}
	Soit $\underline{G}$ un $\GG_{m,\ZZ}$-groupe de Tits, voir la déf. \ref{definition groupe de Tits}, et $\Omega\subset \Ascr( \underline{G},\underline{S},\Omega)$ une partie non vide et bornée. On leur associe son $\AAA^1_\ZZ$-modèle immobilier $\underline{\Gscr_\Omega}$ déduit du th. \ref{theoreme d'existence et unicite des groupes attaches aux drs} appliqué à la donnée radicielle schématique de la prop. \ref{proposition drs modele immobilier groupe de tits}.
\end{defn}

Particularisant les schémas en groupes ci-dessus, on obtient à l'évidence ceux de la théorie de Bruhat--Tits pseudo-réductive qu'on avait mentionnée au th. \ref{theoreme theorie de bruhat-tits pseudo-reductive}.

\begin{cor}\label{corollaire specialises modeles immobiliers en caracteristique p}
Le $\FF_p\pot{t}$-modèle de $G_{\eta_p}$ déduit du $\AAA_{\ZZ}^1$-modèle immobilier $\underline{\Gscr_{ \Omega}}$ de la déf. \ref{definition modeles immobiliers du groupe de tits} par le changement de base $\ZZ[t]\rightarrow \FF_p\pot{t}$ s'identifie au $\FF_p\pot{t}$-modèle immobilier $\Gscr_{\Omega,p}$ de $G_{\eta_p}$ associé à $\Omega\subset \Ascr(G_{\eta_p},S_{\eta_p},\FF_p\rpot{t})$. 
\end{cor}
\begin{proof}
	Il ne suffit que de considérer les termes de la donnée radicielle schématique, cf. déf.\ref{definition drs modele entier groupe de tits}. Mais si l'on tue $p$ dans les coefficients entiers et complète pour la topologie $t$-adique, alors on obtient encore le modèle de Néron connexe du centralisateur, voir \cite[prop. 3.3.8]{LouDiss}, et les stabilisateurs radiciels de $\Omega$, voir \cite[4.3.2, 4.3.4]{BTII} et \cite[lem. 4.4]{LouBT}, en identifiant celui à une partie convenable de $\Ascr(G_{\eta_p},S_{\eta_p},\FF_p\rpot{t})$ selon la prop. \ref{proposition combinatoire caracteristiques differentes}.
\end{proof}

Enfin, on va relier les $\AAA^1_\ZZ$-modèles immobiliers $\underline{\Gscr_{ \Omega}}$ de $G$ avec ceux de $H$, tel que l'on avait fait pour leurs groupes de Tits, cf. prop. \ref{proposition plongement restrictions des scalaires groupe de tits}. On remarque que l'isomorphisme $G_{\QQ(\zeta_e,t^{1/e})} \simeq H$ induit un plongement \begin{equation}\label{equation inclusion d'appartements}\Ascr(G,S,\QQ\rpot{t}) \hookrightarrow \Ascr(H,T_H,\QQ(\zeta_e)\rpot{t^{1/e}})\end{equation} d'appartements, donc $\Omega$ peut être regardé comme partie du membre de droite. Le résultat ci-dessous ne sera jamais appliqué dans le reste de l'article.

\begin{propn}
	Le morphisme naturel $\underline{\Gscr_{ \Omega}} \rightarrow \Res_{\ZZ[\zeta_e,t^{1/e}]/\ZZ[t]} \underline{\Hscr_{\Omega}}$ induit par l'inclusion (\ref{equation inclusion d'appartements}) est un plongement localement fermé.
\end{propn}

\begin{proof}
	On reprendre la preuve de la prop. \ref{proposition plongement restrictions des scalaires groupe de tits} pour voir que l'adhérence schématique de $T$ et des $U_a$ dans le membre de droite s'identifie aux modèles immobiliers souhaités de la déf. \ref{definition drs modele entier groupe de tits}. D'après \cite[prop. 2.2.10]{BTII}, ceci termine la preuve de la propriété énoncée. On s'attend à ce que l'image fermée du membre de gauche soit un sous-$\ZZ[t]$-groupe plat et fermé du membre de droite (dont la composante neutre s'identifierait à $\underline{\Gscr_{ \Omega}}$), mais cela pourrait éventuellement n'être pas le cas, cf. \cite[3.2.15]{BTII}, et il semble que nos méthodes soient insuffisantes pour y parvenir.
\end{proof}

\begin{rem}
	Dans une autre direction, lorsque $G$ est un groupe classique, il serait aussi très intéressant d'interpréter les groupes parahoriques comme classifiant les automorphismes de certains chaînes périodiques de réseaux munis de formes quadratiques comme en \cite[\S5]{PZh}, ou \cite[pp. 218-219]{TitsOber}, ce qui peut être regardé comme une généralisation aux bases bidimensionnelles de \cite[\S3, annexe]{RZ}. Voir aussi \cite{BTIV}, \cite{BTV} et \cite{Kir2} pour des cas classiques non modérément ramifiés sur un anneau de valuation discrète.
\end{rem}

Enfin, on mettre un terme à nos enquêtes sur la théorie des groupes, en démontrant que les $\ZZ[t]$-groupes lisses, séparés et connexes $\underline{\Gscr_\Omega}$ sont même affines, duquel on profitera plus tard pour faire la géométrie des grassmanniennes affines, voir les \S\S\ref{section grassmanniennes affines}-\ref{section grassmanniennes beilinson-drinfeld} et en particulier les props. \ref{proposition representabilite grassmannienne affine locale} et \ref{proposition representabilite grassmannienne affine globale}.

\begin{thm}\label{theoreme affinite modeles immobiliers}
Le schéma en groupes $\underline{\Gscr_{\Omega}}$ est affine sur $\AAA^1_\ZZ$.
\end{thm}

 On va démontrer ce résultat fondamental de manière complètement différente à celle de \cite[th. 4.1]{PZh} dans le cas modérément ramifié. En effet, ce dernier article procède par deux étapes : dans le cas déployé, des techniques de dilatation sont appliquées, voir \cite[\S1]{WW} et \cite[\S3.2]{BLR} pour le cas de la dimension $1$ et \cite[\S3]{RichDil} en dimension supérieure ; dans le cas tordu, un calcul cohomologique en termes de $\Gamma$-invariants est évoqué.
 
 Dans notre cadre généralisé, même si l'on sait déjà que $\underline{G}$ est affine sur $\ZZ[t^{\pm 1}]$ d'après la prop. \ref{proposition plongement restrictions des scalaires groupe de tits}, nous ne disposons pas d'une interprétation galoisienne pure, pour qu'on puisse calculer des groupes de cohomologie. Nos arguments sont donc beaucoup plus abstraits, mais aussi élémentaires en ce qui concerne les outils techniques.
 
 \begin{proof}
 On doit prouver que l'enveloppe affine $\underline{\Gscr_{\Omega}^{\on{af}}}$ est connexe, voir la prop. \ref{proposition enveloppe affine}, et on s'est ramené à traiter les fibres géométriques en les idéaux maximaux $(p,t)$ de $\ZZ[t]$, pour $p>0$ un nombre premier. Soient $\underline{\Nscr_{\Omega}}$ (resp. $\underline{\Nscr_{\Omega}^{\on{af}}}$) le normalisateur de $\underline{\Sscr}$ dans $\underline{\Gscr_{\Omega}}$ (resp. dans l'enveloppe affine $\underline{\Gscr_{\Omega}^{\on{af}}}$), qui sont sous-$\ZZ[t]$-groupes fermés et lisses, cf. \cite[exp. XI, cor. 2.4 bis, cor. 6.11]{SGA3}. Chaque composante connexe de cette fibre géométrique peut être représentée par un point de $\underline{\Nscr_{\Omega}^{\on{af}}}$ à valeurs dans $\breve{\ZZ}_p\pot{t}$, grâce au lemme d'Hensel joint au cor. \ref{corollaire specialises modeles immobiliers en caracteristique p} et au fait correspondant en théorie de Bruhat--Tits, voir \cite[prop. 2.12]{LouBT} et \cite[cor. 4.6.21]{BTII}. Alors, la non connexité de $\underline{\Gscr_{\Omega}^{\on{af}}}$ au-dessus de $\breve{\ZZ}_p\pot{t}$ contredit à l'évidence l'affirmation selon laquelle les flèches \begin{equation}\label{equation bijection quotients d'iwahori-weyl mod p}\underline{N}(\breve{\ZZ}_p\rpot{t})/\underline{\Nscr_{\Omega}}(\breve{\ZZ}_p\pot{t}) \rightarrow  \underline{N}(\overline{\FF}_p\rpot{t})/\underline{\Nscr_{\Omega}}(\overline{\FF}_p\pot{t}) \end{equation} \begin{equation}\label{equation bijection quotients d'iwahori-weyl caracteristique nulle}
 \underline{N}(\breve{\ZZ}_p\rpot{t})/\underline{\Nscr_{\Omega}}(\breve{\ZZ}_p\pot{t}) \rightarrow  \underline{N}(\breve{\QQ}_p\rpot{t})/\underline{\Nscr_{\Omega}}(\breve{\QQ}_p\pot{t})
 \end{equation} sont bijectives, puisque la descente de Beauville--Laszlo, voir \cite[th.]{BLDesc}, en fournit l'égalité \begin{equation}\underline{\Nscr_{\Omega}^{\on{af}}}(\breve{\ZZ}_p\pot{t})=\underline{N}(\breve{\ZZ}_p\rpot{t}) \cap \underline{\Nscr_{\Omega}}(\breve{\QQ}_p\pot{t}) \end{equation} comme sous-groupes de $\underline{N}(\breve{\QQ}_p\rpot{t})$. 
 
 Maintenant, nous montrons cette assertion, c'est-à-dire que les applications (\ref{equation bijection quotients d'iwahori-weyl mod p}) et (\ref{equation bijection quotients d'iwahori-weyl caracteristique nulle}) sont bijectives. En effet, nous avons déjà vu pendant la preuve de la prop. \ref{proposition combinatoire caracteristiques differentes} que les diverses variantes des groupes d'Iwahori-Weyl sont toutes canoniquement isomorphes, quels que soient les coefficients utilisés, cf. (\ref{equation isomorphisme groupe d'iwahori-weyl caracteristiques differentes}). Par suite, il reste à avérer que les applications d'ensembles finis \begin{equation}\label{equation comparaison groupe de weyl immobilier mod p}\underline{\Nscr_{\Omega}}(\breve{\ZZ}_p\pot{t})/\underline{\Tscr}(\breve{\ZZ}_p\pot{t}) \rightarrow  \underline{\Nscr_\Omega}(\overline{\FF}_p\pot{t})/\underline{\Tscr}(\overline{\FF}_p\pot{t}), \end{equation}\begin{equation}\label{equation comparaison groupe de weyl immobilier p unite}
 \underline{\Nscr_\Omega}(\breve{\ZZ}_p\pot{t})/\underline{\Tscr}(\breve{\ZZ}_p\pot{t}) \rightarrow  \underline{\Nscr_\Omega}(\breve{\QQ}_p\pot{t})/\underline{\Tscr}(\breve{\QQ}_p\pot{t})
 \end{equation} sont bijectives. D'une part, le cardinal du membre de droite de ces applications ne dépend pas de la caractéristique des coefficients, puisqu'il s'identifie en tout cas au groupe de Weyl résiduel $W(\Phi_{\Omega})$ du sous-système de racines $\Omega$-résiduel
 \begin{equation}\label{equation sous-systeme residuel des modeles immobiliers}
  \Phi_\Omega:= \{ a \in \Phi_G : f_\Omega(a)+f_\Omega(-a)=0\},
 \end{equation} voir \cite[cor. 4.6.12]{BTII} et \cite[lem. 4.6]{LouBT} pour une preuve de cette affirmation. D'autre part, la première flèche (\ref{equation comparaison groupe de weyl immobilier mod p}) est surjective par le lemme d'Hensel et la deuxième (\ref{equation comparaison groupe de weyl immobilier p unite}) est injective, parce que $\underline{\Tscr}$ est fermé dans $\underline{\Nscr_{\Omega}}$.
\end{proof}
 
\begin{rem}
	Rappelons succinctement la stratégie de preuve de la version précédente du texte dans le cas parahorique : la première étape consistait à exploiter que les stabilisateurs soient automatiquement connexes si $G=G^{\on{sc}}$ pour montrer la connexité de l'enveloppe affine ; la deuxième étape ramenait le problème par modification centrale aux modèles de Néron des tores pas forcément induits, dès que $G^{\on{dér}}=G^{\on{sc}}$ ; la troisième étape explicitait ceux-ci grâce à la classification des tores cycliques lorsque $e \leq 3$. À la fin, cette preuve était encore plus longue que l'actuelle et ne fonctionnait que dans un cadre limitant.
\end{rem}
\section{La géométrie des grassmanniennes affines tordues entières}\label{section grassmanniennes affines}

Cette partie est consacrée à l'étude géométrique des grassmanniennes affines $\Gr_{\underline{\Gscr_\fbf}}$ attachées aux $\AAA_{\ZZ}^1$-modèles parahoriques $\underline{\Gscr_\fbf}$ des groupes de Tits $\underline{G}$ déjà introduits à la déf. \ref{definition modeles immobiliers du groupe de tits}. On démontre notamment que les schémas de Schubert $\Gr_{\underline{\Gscr_\fbf},\leq w}$, cf. déf. \ref{definition schemas de schubert}, sont géométriquement normaux et scindés en caractéristique $p>0$, cf. th. \ref{theoreme normalite des schemas de Schubert}.

\subsection{Rappels et préparatifs} \label{grassmanniennes affines}
Nous allons rappeler au lecteur les notions et résultats principales concernant les grassmanniennes affines, lesquelles seront les objets géométriques centraux de nos investigations. Voici la définition :

\begin{defn}\label{definition grassmannienne affine locale}
	Soient $A$ un anneau et $\Grm$ un schéma en groupes lisse sur $A\pot{t}$. La grassmannienne affine $\on{Gr}_{\Grm}$ est le pré-faisceau en ensembles de la catégorie des $A$-algèbres, tel que $\Gr_{\Grm}(R)$ soit l'ensemble des classes d'isomorphisme de $\Grm$-fibrés $\Eh$ sur $R\pot{t}$ munis d'une trivialisation au-dessus de $R\rpot{t}$.
\end{defn}

 Il s'agit d'un faisceau en ensembles pour la topologie plate. Le groupe de lacets $L\Grm:R \mapsto \Grm(R\rpot{t})$ de $G$, qui est représentable par un ind-schéma affine de type fini, cf. \cite[1.a]{PR}, y opère par changement de trivialisation. Notant $L^+\Grm:R \mapsto G(R\pot{t})$ le groupe d'arcs de $\Grm$, qui est naturellement contenu dans $L\Grm$ et représentable par un schéma pro-lisse, cf. \cite[prop. 1.3.2]{ZhuIntro}, on peut réaliser la grassmannienne affine comme le faisceau quotient $L\Grm/L^+\Grm$ pour la topologie étale, cf. \cite[prop. 1.3.6]{ZhuIntro}, de manière compatible à l'action à la gauche de $L\Grm$.
 
  La question de savoir quand la grassmannienne affine d'un groupe est représentable est encore subtile, mais nous avons le résultat suivant :
  
  \begin{propn}[Pappas--Zhu]\label{proposition representabilite grassmannienne affine locale}
  	Soient $A$ un anneau de Dedekind et $\Grm$ un $A\pot{t}$-groupe lisse, affine et connexe. Alors, $\Gr_{\Grm}$ est représentable par un ind-schéma quasi-projectif de type fini.
  \end{propn} 

\begin{proof}[Esquisse de la preuve]
Lorsque $\Grm=\on{GL}_n$, la grassmannienne affine admet une interprétation comme espace classifiant de $R\pot{t}$-réseaux dans $R\rpot{t}$, d'où son écriture comme colimite de fermés de grassmanniennes finies, cf. \cite[th. 1.1.3]{ZhuIntro}. Sous les contraintes énoncées, alors \cite[cor. 11.8]{PZh} affirme l'existence d'une immersion fermée $\Grm \rightarrow \on{GL}_n$ à quotient représentable et quasi-affine pour la topologie plate - cela repose essentiellement sur les travaux de Thomason, voir \cite{Thom}. Mais dans ce cas, on parvient à prouver que le morphisme induit $\Gr_{\Grm} \rightarrow \Gr_{\on{GL}_n}$ est représentable par une immersion quasi-compacte, cf. \cite[prop. 1.2.6]{ZhuIntro}.
\end{proof}

L'autre question qui se pose concerne la projectivité de $\Gr_{\Grm}$. Pour obtenir une classification complète, on se borne au cas où $A$ est un corps parfait.

\begin{propn}[Pappas--Rapoport, Richarz, L.]\label{proposition projectivite grassmannienne affine locale}
Soient $k$ un corps parfait, $\Grm$ un $k\pot{t}$-schéma en groupes lisse, affine et connexe. Pour que $\Gr_{\Grm}$ soit un ind-schéma projectif, il faut et il suffit que $\Grm$ soit parahorique au sens de \cite{LouBT}.
\end{propn}

\begin{proof}[Esquisse de la preuve] On reprend les arguments de \cite[th. 5.2]{LouBT}. Pour la suffisance, on peut supposer sans perdre de généralité que $\Grm=\Irm$ est d'Iwahori et que $\Gr_{\Irm}$ est connexe. En particulier, les $\Irm$-orbites fermées de $\Gr_{\Irm}$ sont numérotées par $w \in W^{\on{af}}$, où ce dernier symbole désigne le groupe de Weyl affine, d'après la décomposition de Bruhat de \cite[prop. 3.7]{LouBT}. Cela veut dire que $w$ peut s'exprimer dans la forme d'un produit de réflexions fondamentales ; soit $\wf=[s_1: \dots :s_n]$ le mot réduit correspondant. Considérons la variété de Demazure
	\begin{equation}\on{Dem}_\wf =L^+\Prm_{s_1} \times^{L^+\Irm}\dots \times^{L^+\Irm} L^+\Prm_{s_n}/L^+\Irm,\end{equation}  
	où l'on note $\Prm_s$ le groupe parahorique minimal associé à $s$. On peut montrer que ceci est une fibration en droites projectives localement triviale pour la topologie ouverte, donc représentable par une variété projective lisse, voir \cite[prop. 8.7, prop. 8.8]{PR} et aussi cor. \ref{corollaire projection triviale ouvert}. Son image dans $\Gr_{\Irm}$ par le morphisme multiplication est la variété de Schubert $\Gr_{\Irm,\leq w}$, qui est donc projective.
	
	Quant à la nécessité, on voit facilement que la fibre générique de $\Grm$ est sans $\GG_a$ central, en appliquant ce que $\Gr_{\GG_a}$ est l'ind-droite affine $\AAA_k^\infty$ infinie, cf. \cite[th. 5.2]{LouBT}. Enfin, le fait que $\Grm$ soit parahorique résulte d'après la preuve de \cite[th. 1.19]{Rich} d'un théorème sur les sous-groupes bornés maximaux de $\Grm(k\rpot{t})$, cf. \cite[th. 3.3.1]{BTI}, et du lien avec les sous-groupes paraboliques de la fibre spéciale, cf. \cite[th. 4.6.33]{BTII} et \cite[th. 4.7]{LouBT}, couplés au fait que les sous-groupes de congruence de $L^+\Grm$ sont toujours pro-unipotents, voir le début de \S\ref{subsection cordes et fibres en droites}, particulièrement (\ref{equation suite decroissante groupe d'arcs}).
\end{proof}

Reprenons maintenant les notations des \S\S\ref{subsection groupes de Tits}-\ref{subsection groupes immobiliers modelant tits}, c'est-à-dire : $G$ est un $\QQ(t)$-groupe quasi-déployé, $\QQ(\zeta_e,t^{1/e})$-déployé, $\QQ\rpot{t}$-résiduellement déployé et à tore maximal induit, muni d'un couple de Killing relatif $S \subset T \subset B \subset G$ et d'un quasi-système de Tits, voir l'hyp. \ref{hypothese groupe dans la fibre generique} ; $\underline{G}$ est le $\GG_{m,\ZZ}$-groupe de Tits sur $\ZZ[t^{\pm 1}]$ qui lui est associé par la déf. \ref{definition groupe de Tits} ; et les $\underline{\Gscr_\fbf}$, où $\fbf\subset \Ascr(\underline{G},\underline{S},\ZZ\rpot{t})$ est une facette, sont les $\AAA^1_\ZZ$-modèles parahoriques que nous avons construits à la déf. \ref{definition modeles immobiliers du groupe de tits} et qui sont lisses, affines et connexes sur $\ZZ[t]$, grâce au th. \ref{theoreme affinite modeles immobiliers}. 

Alors, la grassmannienne affine $\Gr_{\underline{\Gscr_\fbf}}$ sur $\spec \ZZ$ est, en quelque sorte\footnote{Ce genre d'affirmation sera précisée à l'annexe \ref{comparaison avec les varietes de demazure de kac-moody}, voir le th. \ref{theoreme comparaison kac-moody schubert}.}, une variété de drapeaux affine entière du groupe tordu $\underline{G}$. On s'intéresse aux sous-schémas intègres et invariants par $L^+\underline{\Gscr_\abf}$ de $\Gr_{\underline{\Gscr_\fbf}}$, où $\abf$ est une alcôve dont l'adhérence contient $\abf$, qu'on appelle les schémas de Schubert :

\begin{defn}\label{definition schemas de schubert}
	Soient $w \in \widetilde{W}/W_\fbf$ et $n_w \in \underline{N}(\ZZ[t^{\pm 1}])$ un représentant de $w$. Alors, le schéma de Schubert $\Gr_{\underline{\Gscr_\fbf},\leq w}$ est l'image schématique de l'application orbite $L^+\underline{\Gscr_\abf} \rightarrow \Gr_{\underline{\Gscr_\fbf}}$ donnée par $g \mapsto g\cdot n_w$.
\end{defn}

Ces schémas sont automatiquement intègres, plats et projectifs et indépendantes du choix d'un représentant $n_w$ de $w$. Les relations d'inclusion entre eux sont induites par l'ordre de Bruhat dans le double quotient $\widetilde{W}/W_\fbf$, voir \cite[prop. 2.8]{RichDipl}. Notre théorème principal concerne leur géométrie :

\begin{thm}[Faltings, Pappas--Rapoport, L.]\label{theoreme normalite des schemas de Schubert}
	Supposons que $G^{\on{dér}}=G^{\on{sc}}$. Alors, les schémas de Schubert $\Gr_{\underline{\Gscr_\fbf},\leq w}$ sont géométriquement normaux sur $\ZZ$, de Cohen--Macaulay et leurs fibres modulo $p>0$ sont scindées.
\end{thm} 

Faisons quelques remarques sur l'énoncé : 

\begin{rem}
	\begin{enumerate}
		\item Au-dessus de $\ZZ[1/e]$, le théorème est dû à Pappas--Rapoport, voir \cite[th. 8.4, 9.g]{PR}, dont la preuve remonte à Faltings, voir \cite[th. 8]{Fal}, pour les groupes déployés.
		\item L'hypothèse que le groupe dérivé soit simplement connexe ne peut être pas soulevée, comme nous l'avons démontré très récemment avec Haines--Richarz, voir \cite[th. 1.1]{HLR} et cor. \ref{corollaire non normalite varietes de schubert}. En effet, les schémas $\Gr_{\underline{\Gscr_\fbf},\leq w}$ ne seront typiquement pas à fibres réduites sans cette hypothèse, cf. \cite[th. 1.5, prop. 3.4]{HLR}.
	\end{enumerate}
\end{rem}
 
 La démonstration du th. \ref{theoreme normalite des schemas de Schubert} sera faite en diverses étapes pendant tout les numéros restants du \S\ref{section grassmanniennes affines}. Outre le traitement complet de la théorie de la grosse cellule et des diviseurs/cycles de Schubert au \S\ref{subsection cordes et fibres en droites}, notre preuve au \S\ref{normalite des varietes de schubert} contient une petite sophistication technique (ou devrait-on dire complication ?), en ce que le lem. \ref{lemme engendrement algebres de lie} de Faltings--Pappas--Rapoport est faux en caractéristique résiduelle $2$ si $\Phi_G$ n'est pas réduit et doit être remplacé par son analogue différentiel supérieur, cf. lem. \ref{lemme engendrement distributions}.

\subsection{Groupes de cordes et cycles de Schubert} \label{subsection cordes et fibres en droites}

Ce numéro est consacré à mieux comprendre la structure du groupe de lacets $L\underline{\Gscr_{\fbf}}$, de son sous-groupe d'arcs $L^+\underline{\Gscr_{\fbf}}$ et de ce qu'on appellera son sous-groupe de cordes $L^-\underline{\Gscr_{\fbf}}$, voir la déf. \ref{definition groupe de cordes}. On s'en sert pour définir quelques cycles de codimension finie dans la grassmannienne affine $\Gr_{\underline{\Gscr_{\fbf}}}$, voir le th. \ref{theoreme cycles de schubert topologie et normalite}, dont les diviseurs parmi eux engendrent le groupe de Picard, cf. cor. \ref{corollaire groupe de picard}, et jouent un rôle essentiel dans la preuve du th. \ref{theoreme normalite des schemas de Schubert}, voir en particulier la prop. \ref{proposition varietes de schubert proprietes} et le lem. \ref{lemme normalite en caracteristique nulle}.

 Le groupe d'arcs $L^+\underline{\Gscr_{\fbf}}$ possède une filtration canonique exhaustive et séparée
 \begin{equation}\label{equation suite decroissante groupe d'arcs} 1 \subset \dots \subset L^{(n+1)}\underline{\Gscr_{\fbf}}\subset L^{(n)}\underline{\Gscr_{\fbf}}\subset L^{(n-1)}\underline{\Gscr_{\fbf}} \subset \dots \subset L^{(0)}\underline{\Gscr_{\fbf}}=L^+\underline{\Gscr_{\fbf}}, \end{equation} donnée par les $n$-ièmes sous-groupes de congruence, c'est-à-dire les noyaux des applications $L^+\underline{\Gscr_{\fbf}} \rightarrow L^n\underline{\Gscr_{\fbf}}:=\Res_{\ZZ[t]/(t^n)/\ZZ}\underline{\Gscr_{\fbf}}$ de réduction modulo $t^n$. Notons que les sous-quotients successifs $L^{(n)}\underline{\Gscr_{\fbf}}/ L^{(n+1)}\underline{\Gscr_{\fbf}}$ de la filtration de congruence sont isomorphes à la fibre spéciale $\underline{\Gscr_{\fbf,s}}:=\underline{\Gscr_{\fbf}}\otimes_{\ZZ[t], t\mapsto 0} \ZZ$ si $n=0$, et au fibré vectoriel de Lie de $\underline{\Gscr_{\fbf,s}}$ comme groupe additif sinon, voir soit \cite[tome I, ch. II, \S4, nº3, prop. 3.2]{DG}, soit \cite[prop. A.9]{RS20}. Par suite, $L^{(1)} \underline{\Gscr_{\textbf{f}}}$ est un schéma en groupes pro-lisse, pro-unipotent et connexe. Il ne reste qu'à entendre la fibre spéciale entière $\underline{\Gscr_{\fbf,s}}$ comme extension d'un certain quotient réductif (qui n'existe pas forcément sur les entiers a priori) et d'un radical unipotent; on s'en occupera dans la suite, voir la prop. \ref{proposition decomposition de levi}.
 
 Avant de se concentrer sur cette décomposition de Levi, il faut discuter le rôle de l'action de rotation de $\GG_m$ sur le groupe de lacets $L\underline{G}$ déduite de l'action tordue du lem. \ref{lemme action de rotation t unite}. Cela fournit par linéarisation une opération $\ZZ$-linéaire de $\GG_m$ sur $L\underline{G}$ et on vérifie facilement que l'opération induite sur son algèbre de Lie est donnée par $t^{1/e} \mapsto at^{1/e}$. Il est facile de voir que les groupes d'arcs parahoriques $L^+\underline{\Gscr_\fbf}$ sont stables par cette opération, soit en généralisant la preuve du lem. \ref{lemme action de rotation t unite} à l'anneau de polynômes $\ZZ[t]$, soit en utilisant l'espace d'arcs de la grosse cellule. Tout cela nos permet d'introduire la terminologie suivante :
 
 \begin{defn}\label{definition racines affines}
 	Les racines affines de $\underline{G}$ sont l'ensemble $\Phi^{\on{af}}_G$ des poids non nuls de $\underline{S}\times \GG_{m,\ZZ}$ dans $L\underline{G}$, où le facteur multiplicatif y opère par rotation. 
 \end{defn} 

Bien qu'il y ait un lien entre cette notion de racines affines et celle utilisée dans \cite[1.3.3, 6.2.6]{BTI}, on avertit le lecteur qu'il y a des nouvelles racines \og imaginaires \fg{}, ainsi appelées dans la théorie des algèbres de Kac--Moody, cf. \cite[\S5.2]{Kac}, provenant de l'algèbre de Lie de $L\underline{T}$ et qui restent invisibles pour le groupe d'Iwahori--Weyl $\widetilde{W}$. 

\begin{lem}\label{lemme racines affines et sous-groupes radiciels affines}
	Le système des racines affines $\Phi^{\on{af}}_G$ se décompose en une $W^{\on{af}}$-orbite simplement transitive $\Phi^{\on{ré}}_G$ formées des $a +n$ pour $a \in \Phi_G$ et $n \in \Gamma_a'$ et une partie dénombrable $\Phi^{\on{im}}_G$ de $\widetilde{W}$-points fixes $n \in \ZZ_{\neq 0}$. Pour chaque racine réelle $\alpha \in \Phi^{\on{ré}}_G$, il y a un et un seul sous-groupe fermé additif $\underline{V_\alpha}$ de $L\underline{G}$ tel que son sous-algèbre de Lie soit l'espace de poids de dimension associé à $\alpha$.
\end{lem}

On ne tentera pas de parler d'une donnée radicielle infinie et on renvoie aux sources de la théorie de Kac--Moody pour ce genre de notions, voir \cite{Kac} pour plus de références. Puisque les racines imaginaires forment deux rayons radiciels de cardinal infini, on a du mal à entendre leurs sous-groupes radiciels.

\begin{proof}
	La détermination du système de racines affines est élémentaire, en appliquant ce que $L\underline{C}\rightarrow L\underline{G}$ est formellement étale, où $\underline{C}=\underline{U_-}\times \underline{T}\times \underline{U_+}$ désigne la grosse cellule, voir \cite[lem. 7.5]{HLR} pour un énoncé du même genre, donc les poids non nuls peuvent être calculés au moyen des quasi-épinglages des sous-groupes radiciels et du tore maximal.
	
	Quant à l'existence des sous-groupes radiciels réels, il suffit de poser $x_a(t^n\GG_{a,\ZZ}) \subset L\underline{G}$ pour voir que cela remplit les contraintes énoncées. Notons que quand $a$ engendre un rayon radiciel pluriel, on entend par cela $x_a(t^n\GG_{a,\ZZ},0)$, si $a \in \Phi_G^{\on{m}}$ est multipliable, et $x_{a/2}(0,t^n\GG_{a,\ZZ})$ si $a\in \Phi_G^{\on{d}}$ est divisible. Pour l'unicité, on prend l'intersection affine et unipotente de deux choix possibles. Vu que l'algèbre de Lie de cette intersection est non triviale, la fibre générique n'est autre que $V_{\alpha,0}$, voir \cite[IV, \S2, cor. 4.5]{DG}, d'où l'égalité par considération de l'adhérence.
	
	On dit enfin un mot sur l'opération du groupe d'Iwahori--Weyl $\widetilde{W}$ sur $\Phi^{\on{af}}_G$. Vu que $L\underline{N}$ opère sur $L\underline{G}$ par conjugaison en préservant $\underline{S} \subset L\underline{S}$ de façon compatible à l'action de rotation, ces automorphismes préservent les espaces propres de $\underline{S}\times \GG_{m,\ZZ}$. Elle est évidemment triviale dès que $\alpha \in \Phi_{G}^{\on{im}}$ est imaginaire, et coïncide avec l'action habituelle de $\widetilde{W}$ sur les racines réelles au-sens de Bruhat--Tits, en les identifiant aux racines affines de $\Ascr(\underline{G},\underline{S}, \ZZ\rpot{t})$.
\end{proof}

 Les parties parahoriques de $\Phi_G^{\on{af}}$ au sens de \cite[1.5.1]{BTI} sont aussi bien entendues :
 
 \begin{propn}\label{proposition parties parahoriques et facettes}
 	Les parties parahoriques de $\Phi_G^{\on{af}}$ sont mis en bijection au signe près avec les facettes $\fbf \subset \Ascr(\underline{G},\underline{S},\ZZ\rpot{t})$, au moyen de l'application qui fait correspondre à $\fbf$ l'ensemble $\Phi^{\on{af}}_{\fbf,\geq 0}$ des poids non nuls de $\underline{S}\times \GG_{m,\ZZ}$ dans $L^+\underline{\Gscr_{\fbf}}$.
 \end{propn}

La partie parahorique opposée de $\Phi^{\on{af}}_{\fbf,\geq 0}$ sera notée $\Phi^{\on{af}}_{\fbf,\leq 0}$ et l'on voit que leur intersection $\Phi^{\on{af}}_{\fbf,0}$ s'identifie naturellement, par restriction au sous-tore $\underline{S}$, au sous-système de racines résiduel $\Phi_\fbf \subset \Phi_G$ associé à $\fbf$ par (\ref{equation sous-systeme residuel des modeles immobiliers}).

\begin{proof}
	On peut considérer la partition symétrique de $\Phi_G^{\on{af}}$ dont les racines positives sont celles de la forme $n \in \ZZ_{>0}$, ou bien $a+n$ avec $a \in \Phi_G^+$ et $n \in \Gamma_{a}\cap \RR_{\geq 0}$ ou $a \in \Phi_G^-$ et $n \in \Gamma_{a}\cap \RR_{>0}$. Mais cela n'est autre que l'ensemble des poids non nuls de $L^+\underline{\Gscr_{\abf}}$, où $\abf$ désigne l'alcôve dominante adhérente à l'origine. Ceci étant, on s'en sert aussi pour récupérer les racines négatives simples de $\Phi_G^{\on{ré}}$ : pour chaque cloison $\fbf_i$ dans l'adhérence de $\abf$, on considère la racine réelle qui est poids de $L^+\underline{\Gscr_{\fbf_i}}$ mais pas de $L^+\underline{\Gscr_{\abf}}$, ce qui montre l'assertion dans le cas parahorique minimal. En général, les sous-groupes ainsi que les parties parahoriques sont engendrées à un sens convenable par les minimaux qui y sont contenus et la proposition est achevée. 
\end{proof}

Nos sous-groupes radiciels affines nous permettent de choisir canoniquement les facteurs de Levi promis au début de ce numéro.

  \begin{propn}\label{proposition decomposition de levi}
 	Le groupe d'arcs $L^+\underline{\Gscr_{\fbf}}$ admet un et un seul sous-$\ZZ$-groupe fermé et distingué, pro-unipotent, pro-lisse et connexe maximal $L^{++}\underline{\Gscr_{\fbf}}$, ainsi qu'un unique sous-$\ZZ$-groupe réductif $\underline{M_\fbf}$ contenant $\underline{S}$ et invariant par l'action de rotation tel que $\underline{M_\fbf} \ltimes L^{++}\underline{\Gscr_{\fbf}} \rightarrow L^{+}\underline{\Gscr_{\fbf}}$ soit un isomorphisme.
 \end{propn} 

Par suite, nous dénommons $L^{++}\underline{\Gscr_{\fbf}}$ le radical pro-unipotent, $\underline{M_\fbf}$ le facteur de Levi et l'écriture suivant
\begin{equation}\label{equation decomposition de levi arcs}\underline{M_\fbf} \ltimes L^{++}\underline{\Gscr_{\fbf}} \rightarrow L^{+}\underline{\Gscr_{\fbf}}\end{equation}
la décomposition de Levi de $L^{+}\underline{\Gscr_{\fbf}}$.
\begin{proof}
Si l'on géométrise \cite[prop. 4.6.10]{BTII} et/ou \cite[lem. 4.6]{LouBT}, on arrive à la conclusion qu'il y a une partie non vide et bornée $\Omega:=\on{tour}(\fbf)$, appelée l'entourage de $\fbf$, voir aussi \cite[IV, lem. 3.5]{LouDiss}, telle que le sous-groupe $L^+\underline{\Gscr_{\Omega}} \subset L^+\underline{\Gscr_\fbf}$ est pro-soluble et dont les sous-groupes radiciels sont ceux de ce qui s'avérera le radical pro-unipotent $L^{++}\underline{\Gscr_\fbf}$ ; autrement dit, la seule différence entre $L^+\underline{\Gscr_{\Omega}}$ et le groupe $L^{++}\underline{\Gscr_\fbf}$ est le sous-tore déployé maximal du premier. Notons que la pro-solubilité de ce groupe entraîne qu'on ait un isomorphisme de schémas $L^+\underline{\Cscr_{\Omega}}\simeq L^+\underline{\Gscr_{\Omega}}$.

En particulier, on s'est presque ramené a construire le radical pro-unipotent $L^{++}\underline{\Tscr}$ au cas où $G=T$ est un tore. Dans cet endroit, on utilise la décomposition (\ref{equation modele de neron polynomes}) comme produit de restrictions des scalaires de groupes multiplicatifs et l'on observe que
\begin{equation}L^+_t\Res_{\ZZ[t^{1/d}]/\ZZ[t]} \GG_{m,,\ZZ[t^{1/d}]} =L^+_{t^{1/d}}\GG_{m,\ZZ[t^{1/d}]} ;\end{equation}
il s'ensuit que le premier groupe de congruence par rapport à la variable $t^{1/d}$ en fournit le radical pro-unipotent cherché. 

En général, on pose 
\begin{equation}L^{++}\underline{\Gscr_{\fbf}}:= L^+\underline{\Uscr_{-,\Omega}}\times L^{++}\underline{\Tscr} \times L^+\underline{\Uscr_{+,\Omega}} \end{equation}
et l'on vérifie que ce sous-schéma fermé de $L^+\underline{\Gscr_{ \fbf}}$ possède une structure naturelle de groupe en vérifiant que les applications conjugaison, commutateur et échange des axiomes de la déf. \ref{definition donnees radicielles schematiques} y sont bien définies. Son unicité est claire, vu que les spécialisés $L^{++}\Gscr_{\fbf,p}$ en caractéristique $p \in \PP \cup \{0\}$ de ce sous-schéma en groupes redonnent le radical pro-unipotent habituel selon \cite[prop. 4.6.10]{BTII} et \cite[lem. 4.6]{LouBT}.

Enfin, on veut construire le facteur de Levi $\underline{M_\fbf}$. Son unicité découle de ce que son algèbre de Lie est le seul complément $\ZZ$-linéaire invariant par $\on{Lie}(\underline{S}\times \GG_{m,\ZZ})$ du radical pro-nilpotent de $\on{Lie}L^+\underline{\Gscr_\fbf}$, donc ses sous-groupes radiciels $\underline{V_\alpha}$ relativement à $S$ sont forcément de la forme $x_a(t^{f_\fbf(a)}\GG_a)$ pour toute racine résiduelle $a \in \Phi_\fbf$, cf. (\ref{equation sous-systeme residuel des modeles immobiliers}) et prop. \ref{proposition parties parahoriques et facettes}. Montrons d'abord que la grosse cellule du sous-groupe $\underline{M_\fbf} \subset L^+\underline{\Gscr_\fbf}$ que nous voulons construire, constituée par le produit des sous-groupes radiciels $\underline{V_\alpha}$ et de $\underline{S}$, en hérite une loi birationnelle. En effet, il s'avérera par les calculs explicites des applications des axiomes (DRS 1-3) que cette loi birationelle coïncide avec celle du groupe épinglé à tore maximal $\underline{S}$ et système de racines $\Phi_\fbf$ au moyen d'un système de Chevalley convenable.

Étudions d'abord les morphismes d'échange des facteurs. On peut voir d'après (\ref{equation diviseur sl2}) et (\ref{equation diviseur su3}) que le lieu ouvert distingué de définition de l'application échange est donné par $1-rs$. Puis (\ref{equation echange sl2}) et (\ref{equation echange su3}) entraînent que $\beta_a$ s'écrive comme
\begin{equation}\beta_a(x_a(t^{f_\fbf(a)}r),x_{-a}(t^{-f_\fbf(a)}s))=(x_{-a}\Big(\frac{t^{-f_\fbf(a)}s}{1-rs}\Big),a^{\vee}(1-rs),x_a\Big(\frac{t^{f_\fbf(a)}r}{1-rs}\Big)),\end{equation}
ce qui coïncide avec l'application voulue. Cela est légèrement plus compliqué lorsque $a \in \Phi_\fbf$ est une racine multipliable ou divisible dans $\Phi_G$, mais les écritures se simplifient nettement et il faut juste observer que $(2a)^\vee=a^\vee/2$.

Finalement, nous traitons l'application commutateur, voir la règle (DR2) de la déf. \ref{definition donnees radicielles schematiques}. Pour cela on peut supposer $G^{\on{ad}}$ absolument simple et non déployé, les morphismes les plus intéressants parmi ceux écrits dans \cite[A.6]{BTII} étant ceux dont la longueur des racines impliquées n'est pas constante. On suppose premièrement que le système de racines $\Phi_G$ est réduit et qu'il possède une arête double. L'éventualité la plus intéressante est lorsque $a,b \in \Phi_{\fbf}\cap \Phi_G^<$ sont racines courtes dont la somme est longue, cf. \cite[A.6 b)]{BTII}. Alors, on obtient
\begin{equation}[x_a(t^{f_\fbf(a)}r),x_b(t^{f_\fbf(b)}s)]\sim\begin{cases}
x_{a+b}(2t^{f_\fbf(a+b)}rs) \text{, si } a+b\in\Phi_\fbf\\ 0 \text{, sinon,}\end{cases} \end{equation}
car l'ensemble des valeurs de la valuation $\varphi_{a+b}$ est égal à $\ZZ$ et la trace de $t^{1/2}\ZZ[t]$ s'annule.
Alors, l'application commutateur est vraiment la cherchée (au signe près), d'après la règle de Chevalley, voir \cite[exp. XXIII, cor. 6.5]{SGA3}. Si nous supposons ensuite que $\Phi_G$ est réduit et qu'il possède une arête triple, alors nous nous concentrons sur le cas où $a,b \in \Phi_{\fbf}\cap \Phi_G^<$ sont courtes, $a+b \in \Phi_G^<$ est courte (automatiquement dans $\Phi_\fbf$) et $2a+b, a+2b \in \Phi_G^>$ sont longues, cf. \cite[A.6 d) et e)]{BTII}. On en déduit les formules : 
\begin{equation}[x_a(t^{f_\fbf(a)}r),x_b(t^{f_\fbf(b)}s)]\sim\begin{cases}
 \prod_{p,q} x_{c}( (p+q)t^{f_\fbf(c)}rs)\text{, si } 2a+b\in\Phi_\fbf\\x_{a+b}( t^{f_\fbf(a+b)}rs)  \text{, sinon,}\end{cases} \end{equation}
où $c=pa+qb$ décrit l'ensemble $]a,b[$. Le premier cas correspond à l'un sous-système de racines résiduel $\Phi_\fbf$ de type $G_2$ et le deuxième à l'un de type $A_2$. 

Finalement, si $\Phi_G$ est non réduit, on peut procéder de la même manière. Par exemple, si $a,b\in \Phi_\fbf \cap \Phi_G^{\on{nd,nm}}$ et $a+b \in \Phi_G^{\on{d}}$ est divisible, alors le commutateur s'écrit, cf. (\ref{equation commutateur pluriel un}), comme l'anti-trace de $rs$, ce qui s'annule si et seulement si $a+b \notin \Phi_\fbf$, et sinon a constante de Chevalley égale à $2$, comme prédit par la règle de Chevalley, cf. \cite[exp. XXIII, cor. 6.5]{SGA3} ; si $a,b\in \Phi_\fbf \cap \Phi_G^{\on{m}}$ et $a+b \in \Phi_G^{\on{nd,nm}}$, alors la constante $2$ qui est parue dans l'écriture due à Tits, cf. (\ref{equation commutateur pluriel deux}), est aussi correcte, grâce à la règle de Chevalley. En ce qui concerne le cas où $a \in \Phi_\fbf\cap \Phi_G^{\on{m}}$, $b \in \Phi_\fbf \cap \Phi_G^{\on{nd,nm}} $ et $a+b,2a+2b, 2a+b \in \Phi_G$, cf. (\ref{equation commutateur pluriel trois}), on en laisse le soin au lecteur, vu que les écritures sont assez simples.

Les paragraphes ci-dessus fournissent par fonctorialité des lois birationnelles une application $\underline{M_{\fbf}} \rightarrow  L^{+}\underline{\Gscr_\fbf}$ et ils montrent aussi que le faisceau quotient $L^+\underline{\Gscr_\fbf}/L^{++}\underline{\Gscr_\fbf}$ pour la topologie plate, lequel est représentable par un $\ZZ$-groupe lisse, affine et connexe d'après un théorème d'Anantharaman, cf. \cite[th. 4.A]{Anan}, s'identifie à $\underline{M_{\fbf}}$ au moyen du morphisme composé. Ceci n'est plus que la décomposition de Levi (\ref{equation decomposition de levi arcs}) qu'on avait cherchée.
\end{proof}

Voici une conséquence pertinente de la prop. précédente :
 
\begin{cor}\label{corollaire projection triviale ouvert}
	Soit $\abf$ une alcôve adhérente à $\fbf$. Alors, $L^+\underline{\Gscr_\fbf}/L^+\underline{\Gscr_\abf}$ est représentable par un schéma projectif lisse et la projection $L^+\underline{\Gscr_\fbf}\rightarrow L^+\underline{\Gscr_\fbf}/L^+\underline{\Gscr_\abf}$ est un $L^+\underline{\Gscr_\abf}$-torseur localement trivial pour la topologie ouverte.
\end{cor}

\begin{proof}
La première assertion résulte de \cite[th. 4.6.33]{BTII} et \cite[th. 4.7]{LouBT}, lequel fournit un sous-groupe borélien $\underline{B_{\abf,\fbf}}$ de $\underline{M_\fbf}$ contenant le tore maximal fixé $\underline{S}$ tel que l'image réciproque de $\underline{B_{\abf,\fbf}}$ dans $L^+\underline{\Gscr_\fbf}$ coïncide avec $L^+\underline{\Gscr_\abf}$. En particulier, le faisceau quotient $L^+\underline{\Gscr_\fbf}/L^+\underline{\Gscr_{\abf}}$ pour la topologie étale est représentable par le $\ZZ$-schéma projectif et lisse $\underline{M_{\fbf}}/\underline{B_{\abf,\fbf}}$ (égal à $\PP^1_\ZZ$ si $\fbf$ est un cloison) et le morphisme quotient $L^+\underline{\Gscr_{\fbf}} \rightarrow L^+\underline{\Gscr_{\fbf}}/L^+\underline{\Gscr_{\abf}}$ possède des sections localement pour la topologie ouverte, car il en est de même de la projection vers $\underline{M_{\fbf}}$ et de l'application $\underline{M_{\fbf}}\rightarrow \underline{M_{\fbf}}/\underline{B_{\abf,\fbf}}$, cf. \cite[exp. XXVI, cor. 5.9]{SGA3}. 
\end{proof}

Dans la prop. \ref{proposition decomposition de levi}, on a montré l'existence du radical pro-unipotent $ L^{++}\underline{\Gscr_{\fbf}}$ et maintenant on va voir qu'il s'identifie au produit direct de ses sous-groupes radiciels $\underline{V_{\alpha}}$, cf. lem. \ref{lemme racines affines et sous-groupes radiciels affines}, pris correctement dans la catégorie des pro-groupes.

\begin{propn}\label{proposition engendrement produit direct sous-groupes radiciels}
	Soit $I$ une partie finie de racines réelles dans $\Phi_{\fbf,>0}^{\on{af}}$ à complément saturé au sens de \cite[déf. C.2.25]{CGP}. Alors, le radical pro-unipotent $L^{++}\underline{\Gscr_\fbf}$ admet un unique sous-groupe fermé et pro-lisse $L^{++}_I\underline{\Gscr_\fbf}$ stable sous $\underline{S}\times \GG_{m,\ZZ}$ tel que les applications produit
	\begin{equation}\label{equation decomposition produit racines affines groupe d'arcs}\prod_{\alpha \in I} \underline{V_\alpha} \times L^{++}_I\underline{\Gscr_\fbf}   \rightarrow L^{++}\underline{\Gscr_\fbf}\end{equation}
	soient des isomorphismes, quel que soit l'ordre mis sur $I$. En particulier, si l'on pose $I_w=\Phi_{\abf,>0}^{\on{ré}} \cap w\Phi_{\fbf,\leq 0}^{\on{ré}}$, alors $L^{++}_{w}\underline{\Gscr_\fbf} =L^{++}\underline{\Gscr_\abf} \cap n_wL^{++}\underline{\Gscr_\fbf}n_w^{-1}$ pour tout représentant $n_w\in LN(\ZZ)$ de $w$.
\end{propn} 

\begin{proof}
 Montrons d'abord que le morphisme suivant
 \begin{equation}\label{equation decomposition produit racines affines quotients congruence}
 \big[\prod_{\alpha\in J}\underline{V_\alpha} \big] \times L^{++}\underline{\Tscr}/L^{(n)}\underline{\Tscr} \rightarrow L^{++}\underline{\Gscr_\fbf}/L^{(n)}\underline{\Gscr_\fbf}
 \end{equation} est un isomorphisme, où $\alpha$ décrit l'ensemble $J$ des racines du membre de droite par rapport à $\underline{S}\times \GG_{m, \ZZ}$ et les sous-groupes radiciels $\underline{V_\alpha}$ sont mis dans un ordre quelconque. Après les changements de base $\ZZ \rightarrow \FF_p$ à tout corps premier, ceci n'est plus que la décomposition habituelle d'un groupe unipotent lisse opéré par un tore de \cite[1.1.7]{BTII} et/ou \cite[prop. C.2.26]{CGP}. En particulier, le $\ZZ$-morphisme donné par (\ref{equation decomposition produit racines affines quotients congruence}) de $\ZZ$-schémas lisses est bijectif et birationnel, donc un isomorphisme d'après le théorème principal de Zariski. 
 
 Choisissons $n$ assez grand tel que $\underline{V_\alpha} \cap L^{(n)}\underline{\Gscr_\fbf}=1$ pour toute racine réelle $\alpha \in I$. Si l'on note $L^{++}_I \underline{\Gscr_{\fbf}}$ l'image réciproque dans $L^{++}\underline{\Gscr_{\fbf}}$ du sous-groupe de $L^{++}\underline{\Gscr_\fbf}/L^{(n)}\underline{\Gscr_\fbf}$ complémentaire de $\prod_{\alpha \in I}\underline{V_\alpha}$, cf. (\ref{equation decomposition produit racines affines quotients congruence}), on obtient ce qu'on avait souhaité. L'unicité est une conséquence de la détermination de l'algèbre de Lie et l'équivalence du foncteur Lie pour les groupes pro-unipotents en caractéristique nulle, voir \cite[th. 4.4.19]{Kum}, ce qui caractérise le modèle entier en prenant l'adhérence.

Tournons-nous vers le cas qui s'avérera fondamental dans la suite, celui de la partie $I_w= \Phi_{\abf,>0}^{\on{ré}}\cap w\Phi_{\fbf,\leq 0}^{\on{ré}}$ de cardinal $l(w)$ et de la décomposition (\ref{equation decomposition produit racines affines groupe d'arcs}) pour $I=I_w$. L'inclusion $L^{++}_w\underline{\Gscr_{\fbf}} \subset L^{++}\underline{\Gscr_{\abf}} \cap n_wL^{++}\underline{\Gscr_{\fbf}}n_w^{-1}$ est presque évidente et se démontre aisément, dont le soin est laissé au lecteur. Il reste à observer que l'intersection de $\prod_{\beta\in w^{-1}I_w} \underline{V_{\beta}}$ avec $L^{++}\underline{\Gscr_{\fbf}}$ est triviale, où l'on peut ordonner les racines réelles à goût. En effet, les racines peuvent être supposées avoir été ordonnées de telle façon que le produit soit contenu dans l'espace de lacets $L\underline{C}$ de la grosse cellule de $\underline{G}$, ainsi que le radical pro-unipotent $L^{++}\underline{\Gscr_\fbf}$, voir la prop. \ref{proposition decomposition de levi}. L'affirmation cherchée en résulte par des calculs élémentaires et grâce à ce que $\beta=w^{-1}\alpha \in \Phi_{\fbf,\leq 0}^{\on{ré}}$ pour tout $\alpha \in I_w$.
\end{proof}

Ensuite, on va mettre en valeur le fait qu'il y ait une certaine symétrie entre les parties parahoriques $\Phi_{\fbf,\geq 0}^{\on{af}}$ et leurs opposées $\Phi_{\fbf,\leq 0}^{\on{af}}$. 

\begin{lem}\label{lemme appartements jumeles}
	Il existe un et un seul isomorphisme de complexes simpliciaux \begin{equation}\label{equation opposition appartements inverser variable}\on{op}:\Ascr(\underline{G},\underline{S},\ZZ\rpot{t}) \simeq \Ascr(\underline{G},\underline{S},\ZZ\rpot{t^{-1}}),\end{equation}
	qui fait correspondre à une facette $\fbf $ du membre de gauche la seule facette $\on{op}(\fbf)$ du membre de droite telle que $\Phi_{\on{op}(\fbf),\geq 0}^{\on{af}}=\Phi_{\fbf,\leq 0}^{\on{af}}$ comme parties du groupe de caractères de $\underline{S}\times \GG_{m,\ZZ}$. En particulier, $\on{op}$ est équivariant sous l'action de $\widetilde{W}=\underline{N}(\ZZ[t^{\pm 1}])/\underline{S}(\ZZ)$.
\end{lem}

\begin{proof}
On remarque que cela est raisonnable, puisque la $\GG_m$-opération fut définie en termes de la variable $t$, d'où l'inversion de signe lorsqu'on veut la remplacer par $t^{-1}$. L'existence et unicité d'une application $\on{op}$ telle que dans (\ref{equation opposition appartements inverser variable}) avec les propriétés énoncées est un exercice élémentaire en ensembles simpliciaux et leurs réalisations géométriques, d'après la prop. \ref{proposition parties parahoriques et facettes}, ce qui donne aussi l'équivariance par le groupe d'Iwahori--Weyl. Signalons toutefois que (\ref{equation opposition appartements inverser variable}) ne coïncide pas avec l'application d'oubli des uniformisantes, car les alcôves dominantes du membre de gauche sont appliquées sur les alcôves anti-dominantes du membre de droite.
\end{proof}

 Désormais on notera la facette opposée tout simplement $\fbf$ par abus de langage. La méthode du numéro \S\ref{subsection groupes immobiliers modelant tits} peut donc être appliquée pour construire le $\ZZ[t^{-1}]$-modèle parahorique $\underline{\Gscr_\fbf}$ associé à la facette $\fbf$, cf. déf. \ref{definition modeles immobiliers du groupe de tits}. Ceci donne par recollement au-dessus de $\ZZ[t^{\pm 1}]$ le long du groupe de Tits $\underline{G}$ un $\PP_\ZZ^1$-schéma en groupes affine, lisse et connexe noté encore $\underline{\Gscr_\fbf}$, voir th. \ref{theoreme affinite modeles immobiliers}.

\begin{defn}\label{definition groupe de cordes}
	Le groupe de cordes\footnote{Cela n'a aucun rapport avec la théorie des cordes en physique et ne vise qu'éclaircir le fait que la droite affine soit beaucoup plus fine que son voisinage formel.} de $\underline{\Gscr_\fbf}$ est le faisceau en ensembles $L^-\underline{\Gscr_\fbf}$ sur les $\ZZ$-algèbres donné par la correspondance \begin{equation}R \mapsto \underline{\Gscr_\fbf}(R[t^{-1}]).\end{equation} Le faisceau en ensembles $L^{\on{gr}}\underline{G}$ qui fait correspondre \begin{equation} R \mapsto \underline{G}(R[t^{\pm1}])\end{equation} est ce qu'on appelle le groupe d'élastiques de $\underline{G}$.
\end{defn}

Par le lemme de descente de Beauville--Laszlo, voir \cite[th.]{BLDesc}, on a une égalité de faisceaux $L^-\underline{\Gscr_\fbf}=L_{t^{-1}}^+\underline{\Gscr_\fbf} \cap L^{\on{gr}}\underline{G}$, où l'intersection est prise dedans $L_{t^{-1}}\underline{G}$. On pose aussi \begin{equation}L^{--}\underline{\Gscr_\fbf}=L_{t^{-1}}^{++}\underline{\Gscr_\fbf} \cap L^{\on{gr}}\underline{G},\end{equation} qu'on regarde comme le radical ind-unipotent du groupe de cordes, dénomination laquelle n'a aucun sens précis pour l'instant, mais voir cependant (\ref{equation decomposition de levi cordes}). On note aussi $L^{--}_I\underline{\Gscr_\fbf}$ et $L^{--}_w\underline{\Gscr_\fbf}$ les sous-groupes évidentes de $L^{--}\underline{\Gscr_\fbf}$, voir la prop. \ref{proposition engendrement produit direct sous-groupes radiciels}, déduits du recollement à la Beauville--Laszlo. Cela nous permet ainsi de déduire sans peine plusieurs propriétés concernant $L^-\underline{\Gscr_\fbf}$ :

\begin{thm}\label{theoreme proprietes groupe de cordes}
	On a une égalité de sous-faisceaux
	\begin{equation}\label{equation intersection d'arcs et de cordes}\underline{M_\fbf}=L^+\underline{\Gscr_\fbf}\cap L^-\underline{\Gscr_\fbf},\end{equation}
	du groupe de lacets $L\underline{G}$, ainsi qu'une décomposition de Levi opposée
	\begin{equation}\label{equation decomposition de levi cordes}L^-\underline{\Gscr_\fbf}=\underline{M_\fbf}\ltimes L^{--}\underline{\Gscr_\fbf}\end{equation} et une décomposition radicielle pour toute partie finie $I\subset \Phi^{\on{af}}_{\fbf,< 0}$ à complément saturé
	\begin{equation}\label{equation decomposition produit racines affines groupe de cordes}L^{--}\underline{\Gscr_\fbf}=\prod_{\alpha \in I} \underline{V_\alpha} \times L^{--}_I\underline{\Gscr_\fbf}. \end{equation}
	En particulier, il en résulte, pour tout $w \in \widetilde{W}/W_\fbf$ et tout représentant $n_w \in L^{\on{gr}}N(\ZZ)$, la décomposition suivante :
	\begin{equation}\label{equation decomposition cordes conjugaison par iwahory-weyl}L^{--}\underline{\Gscr_\fbf}=(L^{--}\underline{\Gscr_\abf} \cap n_wL^{+}\underline{\Gscr_\fbf}n_w^{-1})\times L_w^{--}\underline{\Gscr_\fbf}.\end{equation}
\end{thm}

\begin{proof}
Évidemment l'application $\underline{M_\fbf} \rightarrow L\underline{G}$ se factorise à travers du fermé $L^-\underline{\Gscr_\fbf}$ du but, puisque cela est valable au-dessus d'un ouvert schématiquement dense de la source. La décomposition de Levi (\ref{equation decomposition de levi cordes}), telle que la radicielle (\ref{equation decomposition produit racines affines groupe de cordes}), découlent alors facilement de ce qu'on avait prouvé pour les groupes d'arcs (par rapport à la variable $t^{-1}$), cf. (\ref{equation decomposition de levi arcs}). On en tire au moins l'inclusion
 	\begin{equation} \underline{M_\fbf} \subset L^+\underline{\Gscr_\fbf}\cap L^-\underline{\Gscr_\fbf}\end{equation}
 la plus facile de l'égalité (\ref{equation intersection d'arcs et de cordes}).
 
 D'autre part, l'intersection contemplée dans le membre de droite de (\ref{equation intersection d'arcs et de cordes}) coïncide avec le foncteur \begin{equation}R \mapsto \underline{\Gscr_\fbf}(\PP^1_R),\end{equation} lequel est représentable par un $\ZZ$-schéma en groupes affine et lisse. En effet, l'affinité et la présentation finie se ramènent au cas du même foncteur associé aux espaces affines. D'ailleurs, pour montrer que le critère de lissité formelle soit remplit, il faut vérifier l'annulation de certaines classes d'obstruction
 \begin{equation}\label{equation annulation premier groupe de cohomologie fibre vectoriel parahorique}
 H^{1}(R,\on{Lie}\underline{\Gscr_\fbf})=1
 \end{equation} dans le premier groupe de cohomologie de l'algèbre de Lie de $\underline{\Gscr_\fbf}$, cf. \cite[exp. III, cor. 5.4]{SGA1}. Mais celle-ci admet une écriture en tant que somme d'images directes de $\Oh_{\PP_\ZZ^1}$ ou $\Oh_{\PP_\ZZ^1}(-1)$ le long de morphismes finis $\PP_\ZZ^1\rightarrow \PP_\ZZ^1$.
 
 Il suffit de montrer que le $\ZZ$-groupe affine et lisse $L^{++}\underline{\Gscr_\fbf}\cap L^-\underline{\Gscr_\fbf}$ est trivial. Mais il est aussi unipotent, puisque contenu dans le radical pro-unipotent, et étale, grâce à l'annulation de son fibré vectoriel de Lie. En caractéristique $0$, les groupes unipotents et étales ne sont autres que les triviaux, cf. \cite[IV, \S2, cor. 4.5]{DG}, donc par adhérence plate $L^{++}\underline{\Gscr_\fbf}\cap L^-\underline{\Gscr_\fbf}$ est trivial aussi.
 
 Quant à la dernière assertion, cf. (\ref{equation decomposition cordes conjugaison par iwahory-weyl}), nous l'avons déjà vu à (\ref{equation decomposition produit racines affines groupe d'arcs}), quitte à faire la descente de Beauville--Laszlo et à constater que $L^{--}\underline{\Gscr_\abf} \cap n_wL^+\underline{\Gscr_\fbf}n_w^{-1}$ est égal au produit des sous-groupes radiciels qu'il contienne. Néanmoins, si ce n'était pas le cas, alors son intersection avec le sous-groupe $L^{--}_w\underline{\Gscr_\fbf}$ ne serait pas triviale, ce qui est un absurde.
\end{proof}

\begin{rem}
	La définition du groupe de cordes n'est jamais parue dans la littérature que dans des cas particuliers. de Cataldo--Haines--Li ont introduit en \cite[\S3]{dHL18} le radical ind-unipotent pour les groupes déployés, en prenant l'intersection de ceux associés aux alcôves entourantes. Par là, on ne retrouve autre que nos groupes, en vertu de \cite[prop. 3.6.4]{dHL18}. En \cite{HLR}, nous avons étendu avec Haines--Richarz ce même radical ind-unipotent au cas tordu et modérément ramifié en prenant des $\Gamma$-invariants, voir \cite[éq. (3.12)]{HLR}, ce qui est évidemment encore compatible avec la définition actuelle. 
\end{rem}

\begin{cor}\label{corollaire immersion ouverte}
	L'application multiplication \begin{equation}\label{equation immersion ouverte multiplication arcs et cordes} L^{--}\underline{\Gscr_\fbf}\times L^+\underline{\Gscr_\fbf}\rightarrow L\underline{G}\end{equation}
	est une immersion ouverte et affine.
\end{cor}

\begin{proof}
	On prouve ce corollaire à la suite de \cite[lem. 3.6]{HLR}, observant que le morphisme (\ref{equation immersion ouverte multiplication arcs et cordes}) satisfait aux propriétés énoncées si et seulement s'il en est de même du morphisme quotient
	\begin{equation}\label{equation immersion ouverte grassmannienne affine}
	L^{--}\underline{\Gscr_\fbf} \rightarrow \Gr_{\underline{\Gscr_\fbf}},
	\end{equation} d'après la descente plate.
	La première application, cf. (\ref{equation immersion ouverte multiplication arcs et cordes}), est représentable en schémas affines, parce que $L^{--}\underline{\Gscr_\fbf}\rightarrow L\underline{G}$ est une immersion fermée et que $L^+\underline{\Gscr_\fbf}$ est représentable par un schéma affine. La deuxième, cf. (\ref{equation immersion ouverte grassmannienne affine}), est aussi de présentation finie, car ses deux membres admettent des présentations inductives par de $\ZZ$-schémas de type fini. 
	
	Mais on a déjà vu au th. \ref{theoreme proprietes groupe de cordes} que (\ref{equation immersion ouverte multiplication arcs et cordes}) détermine un monomorphisme, cf. (\ref{equation intersection d'arcs et de cordes}) et (\ref{equation decomposition de levi cordes}), donc on n'a besoin que de vérifier le critère de lissité formelle. Comme tous les foncteurs en groupes sont formellement lisses, cela se ramène à un calcul de fibrés vectoriels de Lie, lequel donne un isomorphisme.
\end{proof}

Le lemme de descente de Beauville--Laszlo, cf. \cite[th.]{BLDesc}, fournit un morphisme de $\Gr_{\underline{\Gscr_{\fbf}}}$ vers le champ algébrique $\on{Fib}_{\underline{\Gscr_{\fbf}}, \PP_\ZZ^1}$ lisse et localement de type fini qui paramètre les $\underline{\Gscr_{\fbf}}$-fibrés sur la droite projective.

\begin{cor}\label{corollaire tirant en arrière les fibres triviaux}
	La cellule ouverte $L^{--}\underline{\Gscr_{\fbf}}$ de $\Gr_{\underline{\Gscr_\fbf}}$ coïncide avec l'image réciproque du sous-champ ouvert de $\on{Fib}_{\underline{\Gscr_\fbf}, \PP_\ZZ^1}$ formé des fibrés triviaux.
\end{cor}

\begin{proof}
	L'affirmation se démontre à la suite de \cite[lem. 3.1]{HRi}, dont le point plus délicat était de constater la $H^1$-trivialité de $\on{Lie}\underline{\Gscr_\fbf}$, voir (\ref{equation annulation premier groupe de cohomologie fibre vectoriel parahorique}), ce qu'on a fait pendant la démonstration du th. \ref{theoreme proprietes groupe de cordes}.
\end{proof}

Voici une description dynamique des groupes d'arcs $L^+\underline{\Gscr_\fbf}$ et de cordes $L^-\underline{\Gscr_\fbf}$ à composantes connexes près.

\begin{propn}\label{proposition dynamique groupes d'arcs et de cordes}
	Soit $\mu: \GG_{m,\ZZ} \rightarrow \underline{S} \times \GG_{m,\ZZ}$ un copoids tel que l'ensemble des racines affines $\alpha \in \Phi_{G}^{\on{af}}$ appliquées sur la demi droite non négative soit égal à $\Phi_{\fbf,\geq 0}$. Alors, le groupe d'arcs $L^+\underline{\Gscr_\fbf}$ resp. le groupe de cordes $L^-\underline{\Gscr_\fbf}$ s'identifie à la composante neutre de l'attracteur resp. du répulseur de $L\underline{G}$ pour l'action de $\GG_{m,\ZZ}$ induite par $\mu$. 
\end{propn}

\begin{proof}
	On remarque qu'un tel copoids $\mu$ avec les propriétés désirées existe toujours, grâce à la prop. \ref{proposition parties parahoriques et facettes}. Pour plus de détails sur les $\GG_m$-opérations et leurs généralités, le lecteur est renvoyé à \cite[\S2.1]{CGP} et \cite[\S1]{Ri19a}. Ce dernier montre au \cite[th. 1.8]{Ri19a} que les actions qui stabilisent un recouvrement ouvert (ou plus généralement étale) par des affines admettent des attracteurs et des répulseurs représentables et ceci s'étend facilement au cadre des ind-schémas. 
	
	On montre tout d'abord que l'attracteur de l'espace de lacets de la grosse cellule $L\underline{C}$ est le bon. Supposons que $a$ est une racine ni multipliable ni divisible. Alors, l'action de $\GG_m$ sur $L\underline{U_a}$ induite par $\mu$ est donnée par
	\begin{equation}r \cdot \sum_{i\in \Gamma_a} a_i t^i= \sum_{i\in \Gamma_a} r^{n+im}a_i t^i\end{equation}
	en termes du quasi-système de Tits choisi. Alors, les points à valeurs dans $R$ de l'attracteur s'écrivent comme le sous-groupe additif des séries formelles dans $R\rpot{t}$ tronquées par $i=-n/m$. On voit ainsi que cela coïncide avec $L^+\underline{\Uscr_a}$, comme on le voulait. Quant au cas des racines multipliables et du tore, on en laisse le soin au lecteur. 
	
	Maintenant, on démontre que le groupe de cordes $L^-\underline{\Gscr_\fbf}$, qui est stable par $\mu$ et ind-affine, est son propre répulseur. En fait, le répulseur est représentable par un sous-$\ZZ$-ind-groupe fermé, et l'on sait qu'il contient l'espace de cordes de la grosse cellule $L^-\underline{\Cscr_\fbf}$. Par suite, l'immersion fermée du répulseur vers $L^-\underline{\Gscr_\fbf}$ est formellement étale et de présentation finie, donc un isomorphisme sur un ouvert. Néanmoins, le groupe de cordes est connexe, grâce au cor. \ref{corollaire immersion ouverte} et au lem. \ref{lemme decomposition de birkhoff}, donc l'inclusion voulue est achevée.
	
	Considérons maintenant l'attracteur de l'ouvert translaté, voir le cor. \ref{corollaire immersion ouverte},
	\begin{equation}\label{equation ouverts translates des cordes grassmannienne affine}n_wL^{--}\underline{\Gscr_\fbf}\cdot e=L^{--}_w\underline{\Gscr_\fbf} \times \Gr_{\underline{\Gscr_\fbf}, w}, \end{equation} pour chaque $w\in \widetilde{W}/W_\fbf$, où cette égalité suit d'appliquer la décomposition (\ref{equation decomposition cordes conjugaison par iwahory-weyl}).
	Son attracteur peut être calculé sans grande peine, s'avérant égal à la cellule de Schubert $\Gr_{\underline{\Gscr_\fbf},w}$. En particulier, l'attracteur de $L\underline{G}$ est un sous-groupe fermé réunion disjointe de certaines cellules de Schubert $L^+\Gscr_\abf \cdot n_w \cdot L^+\Gscr_\fbf$ fermées pour une partie $W_\mu \subset \widetilde{W}/ W_\fbf$ constitué de quelques éléments de longueur $0$, ce qui termine la démonstration.

    Le cas du groupes de cordes est traité de façon similaire, mais on a besoin de substituer la grassmannienne épaisse $L\underline{G}/L^-\underline{\Gscr_\fbf}$ à la grassmannienne affine, laquelle est représentable par un schéma de type infini. Comme la détermination précise du répulseur ne sera jamais utilisée dans la suite, mais juste le fait que nos groupes d'arcs/de cordes soient attractifs/répulsifs, voir la démonstration du th. \ref{theoreme cycles de schubert topologie et normalite}, on omettre les détails qui seront laissés au lecteur.
\end{proof}

Avant de initier l'étude géométrique des $L^-\underline{\Gscr_\fbf}$-orbites fermées de $\Gr_{\underline{\Gscr_\fbf}}$, voir le th. \ref{theoreme cycles de schubert topologie et normalite}, il va falloir déterminer quelques doubles quotients pour les numéroter. 

\begin{lem}\label{lemme decomposition de birkhoff}
	On a une bijection naturelle de type Birkhoff
	\begin{equation}\label{equation decomposition de birkhoff}\widetilde{W}/W_\fbf \simeq L^-\underline{\Gscr_\abf}(k)\backslash L\underline{G}(k)/L^+\underline{\Gscr_\fbf}(k) \end{equation}
	pour un corps quelconque $k$.
\end{lem}

\begin{proof}
L'orbite \begin{equation}\label{equation orbite groupe de cordes}\Gr_{\underline{\Gscr_\fbf}}^w:=L^-\underline{\Gscr_\abf} \cdot n_w= L^-\underline{\Gscr_\abf}/(L^-\underline{\Gscr_\abf} \cap n_wL^-\underline{\Gscr_\abf}n_w^{-1})\end{equation} est un fermé local bien défini de $\Gr_{\underline{\Gscr_\fbf}}$ et isomorphe à $L_w^{--}\underline{\Gscr_\fbf}$ en tant que ind-schéma, d'après (\ref{equation decomposition cordes conjugaison par iwahory-weyl}) et (\ref{equation ouverts translates des cordes grassmannienne affine}). L'assertion (\ref{equation decomposition de birkhoff}) revient alors au même que dire que ces orbites soient disjointes et recouvrent l'espace topologique sous-jacent à $\Gr_{\underline{\Gscr_\fbf}}$.

Par suite, le corps $k$ peut être supposé algébriquement clos, sans perdre de généralité. Afin d'établer la bijection (\ref{equation decomposition de birkhoff}) dans ce cas, on peut procéder verbatim comme dans la preuve de la \cite[prop. 1.1]{HNY}, ce qui nous ramène au cas où $G$ est simplement connexe et absolument presque simple. Évidemment on ne doit pas oublier les bizarres groupes exotiques et barcelonais, auquel cas on applique les \cite[prop. 7.3.3, prop. 9.9.2]{CGP} et \cite[prop. 3.6]{LouBT} pour ramener la question dont il s'agit à la décomposition de Birkhoff d'un groupe déployé.
\end{proof}

\begin{rem}
	Cette décomposition de Birkhoff est très bien connue aussi en théorie des groupes de Kac--Moody, voir les articles et exposés suivants \cite{TitsUniq}, \cite{TitsBour} et \cite{TitsTwin} de Tits, ainsi que le survol \cite{Mar} de Marquis. Dans son interprétation, les points à valeurs dans $k$ de nos groupes d'arcs et de cordes sont devenus des sous-groupes paraboliques opposés d'un groupe de Kac--Moody, voir l'annexe \ref{comparaison avec les varietes de demazure de kac-moody}, et forment partie d'une BN-paire raffinée, cf. \cite[prop. 3.16]{Rou16}.
\end{rem}

Comme nous avons mentionné au lemme précédent, cf. (\ref{equation orbite groupe de cordes}), l'orbite $\Gr_{\underline{\Gscr_\fbf}}^w$ associée à $w \in \widetilde{W}/W_\fbf$ admet une immersion affine dans $\Gr_{\underline{\Gscr_\fbf}}$. Ci-dessous, on parvient à saisir leurs adhérences schématiques $\Gr_{\underline{\Gscr_\fbf}}^{\geq w}$ qu'on appelle les cycles de Schubert.

\begin{thm}[Faltings, Kashiwara--Shimozono, L.]\label{theoreme cycles de schubert topologie et normalite}
 Le cycle de Schubert $\Gr_{\underline{\Gscr_\fbf}}^{\geq w}$ est un fermé plat de codimension $l(w)$, stratifié dans la réunion des $\Gr_{\underline{\Gscr_\fbf}}^{v}$, pour $v\geq w$, et à quotients locaux géométriquement normaux de Cohen--Macaulay.
\end{thm}

\begin{proof}
Le jeu à la fois combinatoire et géométrique de Faltings, cf. \cite[p. 48]{Fal}, permet de déduire que \begin{equation}\label{equation stratification de bruhat cycles de schubert}\lvert \Gr_{\underline{\Gscr_\fbf}}^{\geq w} \rvert=\bigcup_{v \geq w} \lvert \Gr_{\underline{\Gscr_\fbf}}^{v}\rvert\end{equation} et que la réunion \begin{equation}\label{equation ouverts invariants par cordes de la grassmannienne affine}\Gr_{\underline{\Gscr_\fbf}}^{u\geq}:=\bigcup_{u \geq w \in \widetilde{W}/W_\fbf} n_wL^{--}\underline{\Gscr_\fbf}\end{equation} est un ouvert stable par $L^{-}\underline{\Gscr_\abf}$. Décrivons ces raisonnements en passant pour aider le lecteur. Considérant le schéma de Richardson \begin{equation}\label{equation definition schema de richardson}\Gr_{\underline{\Gscr_\fbf}, \leq v}^{\geq w}:=\Gr_{\underline{\Gscr_\fbf}}^{\geq w} \cap \Gr_{\underline{\Gscr_\fbf},\leq v}\end{equation} donné par l'intersection d'un cycle avec un schéma de Schubert, et dont la géométrie sera étudiée plus tard au cor. \ref{corollaire schemas de richardson}, on affirme que ceci est non vide si et seulement si $w \leq v$ ; ce critère implique d'immédiat les affirmations précédentes, voir (\ref{equation stratification de bruhat cycles de schubert}) et (\ref{equation ouverts translates des cordes grassmannienne affine}). 

La suffisance de ce critère n'est plus qu'une simple application de la résolution de Demazure, qui permet de produire une droite projective dont la droite affine est contenue dans $\Gr_{\underline{\Gscr_\fbf}}^{w}$ et dont la section infinie est égale à $n_u \cdot e$, où $w < u \leq v$, d'où la conclusion par récurrence. Pour la nécessité, on peut remplacer le schéma de Richardson par une demi-cellule \begin{equation}\Gr_{\underline{\Gscr_\fbf}, v}^{\geq w}:= \Gr_{\underline{\Gscr_\fbf}}^{\geq w} \cap \Gr_{\underline{\Gscr_\fbf}, v}\end{equation} non vide sans perdre de généralité. Mais nous avons vu, cf. (\ref{equation ouverts translates des cordes grassmannienne affine}), que la cellule de Schubert $\Gr_{\underline{\Gscr_\fbf}, v}$ se plonge dans l'ouvert $n_vL^{--}\underline{\Gscr_\fbf}$. Par conséquent, celui-ci doit couper l'orbite schématiquement dense $\Gr_{\underline{\Gscr_\fbf}}^{w}$. Multipliant par un élément convenable de $L^{--}\underline{\Gscr_\abf}$ à la gauche, on revient à ce que la cellule de Richardson \begin{equation}\Gr_{\underline{\Gscr_\fbf}, v}^{w}:= \Gr_{\underline{\Gscr_\fbf}}^{w} \cap \Gr_{\underline{\Gscr_\fbf}, v}\end{equation} est non vide, voir (\ref{equation ouverts translates des cordes grassmannienne affine}). Mais on a toujours une $\GG_m$-opération pour un certain choix convenable d'un copoids dans $\underline{S} \times \GG_{m,\ZZ}$ qui fait contracter les $L^-\underline{\Gscr_\abf}$-orbites en $n_w \cdot e$, voir prop. \ref{proposition dynamique groupes d'arcs et de cordes}, et qui préserve les schémas de Schubert, d'où l'inégalité cherchée.

Soit $L^{-}\underline{\Gscr_\abf}(n)$ le sous-groupe fermé de $L^-\underline{\Gscr_\abf}$ engendré par le $n$-ième sous-groupe de congruence (par rapport à $t^{-1}$) et par $L^{--}\underline{\Tscr}$. Si l'on choisit l'entier $n$ assez grand, alors on peut supposer que $L^{--}\underline{\Gscr_\abf}(n)$ soit contenu dans $L^{--}_w\underline{\Gscr_\fbf}$ pour tout $w \leq v$. Ceci entraîne que le faisceau quotient $L^{--}\underline{\Gscr_\abf}(n)\backslash \Gr_{\underline{\Gscr_\fbf}}^{v\geq}$ pour la topologie plate est représentable par un schéma lisse, appliquant la décomposition radicielle (\ref{equation decomposition produit racines affines groupe de cordes}) joint à (\ref{equation ouverts translates des cordes grassmannienne affine}) comme d'habitude, comp. aussi avec \cite[lem. 6]{Fal}, et que la projection
\begin{equation}\label{equation projection vers quotients locaux de tf grassmannienne affine}
 \Gr_{\underline{\Gscr_\fbf}}^{v\geq} \rightarrow L^{--}\underline{\Gscr_\abf}(n)\backslash \Gr_{\underline{\Gscr_\fbf}}^{v\geq}
\end{equation} a des sections localement pour la topologie ouverte. Puis, vu que le cycle de Schubert restreint $\Gr_{\underline{\Gscr_\fbf}}^{v\geq w}$ est $L^-\underline{\Gscr_\abf}$-invariant, il se descend en un fermé représentable \begin{equation}\label{equation quotients locaux des cycles de schubert} L^{--}\underline{\Gscr_\abf}(n)\backslash \Gr_{\underline{\Gscr_\fbf}}^{v\geq w}\subset L^{--}\underline{\Gscr_\abf}(n)\backslash \Gr_{\underline{\Gscr_\fbf}}^{v},\end{equation} compte tenu de la platitude de $L^-\underline{\Gscr_\abf}(n)$, cf. lem. \ref{lemme platitude des grassmanniennes affines}, cor. \ref{corollaire immersion ouverte} et (\ref{equation decomposition produit racines affines groupe de cordes}). En grandissant $v$ arbitrairement pour recouvrir la grassmannienne toute entière, on obtient une immersion fermée $\Gr_{\underline{\Gscr_\fbf}}^{\geq w}\hookrightarrow \Gr_{\underline{\Gscr_\fbf}}$, comp. avec \cite[th. 7]{Fal}. 

Ceci nous permet d'introduire des propriétés géométriques locales pour la topologie lisse pour ces quotients des cycles de Schubert. On démontre leur normalité par récurrence sur la longueur de $w$ à la suite de Kashiwara--Shimozono, cf. \cite[prop. 3.2]{KS09}. Pour cela, on peut supposer $\fbf=\abf$ et $w \in W^{\on{af}}$. On va faire l'abus de langage de travailler avec les cycles de Schubert en lieu de leurs quotients locaux.

On commence par leurs fibres réduites sur les corps premiers $\FF_p$, $p \in \PP\cup \{0\}$. On note que le cas fondamental $w=1$ découle de ce que $\Gr_{\Gscr_{\abf,p}}$ est lisse. Par équivariance sous $L^-\Gscr_{\abf,p}$, on voit que lieu normal de $\Gr_{\Gscr_{\abf,p}}^{\geq w}$, c'est-à-dire le support du faisceau $\widetilde{\Oh}/\Oh$, est une réunion ouverte et non vide d'orbites et on choisit $v > w$ minimal tel que la restriction de $\Gr_{\Gscr_{\abf,p}}^{v}$ n'y appartienne. Soit aussi $s \in W^{\on{af}}$ une réflexion simple telle que $vs<v$ et considérons la $\PP^1_\ZZ$-fibration \begin{equation}\pi:\Gr_{\Gscr_{\abf,p}}\rightarrow \Gr_{\Gscr_{\fbf,p}},\end{equation} où $\fbf$ désigne le cloison de $\abf$ associé à $s$. Le cas $ws>w$ s'exclut aisément grâce à la minimalité de $v$, donc on tire $ws<w$ et l'on sait que $\Gr_{\Gscr_{\abf,p}}^{\geq ws} $ et $\Gr_{\Gscr_{\fbf,p}}^{\geq w} $ sont normaux par récurrence. Tirant en arrière la normalisation du faisceau structural le long de $\pi$, on en déduit que $\pi_*\Oh=\pi_*\widetilde{\Oh}$, donc leur conoyau est trivial sur $\Gr_{\Gscr_{\abf,p}}^{v} $, compte tenu de l'annulation de $R^1\pi_*\Oh$, voir \cite[lem. 3.1]{KS09}, et de ce que $\pi$ restreinte à l'orbite $\Gr_{\Gscr_{\abf,p}}^{v}$ devient un isomorphisme. 

Pour finir la preuve, il faut noter que le normalisé de $\Gr_{\underline{\Gscr_\abf}}^{\geq w}$ sur $\ZZ$ admet un ouvert lisse et à fibres denses, donc leurs fibres sont géométriquement réduites par le critère de Serre. Par suite, une application du lemme de Nakayama fournit l'affirmation cherchée. Pour montrer que ces cycles sont de Cohen--Macaulay, on renvoie le lecteur à \cite[prop. 3.4]{KS09}.
\end{proof}

\begin{rem}
	La clé de cette preuve de normalité quotiente est l'étude de l'ouverte induit par le groupe de cordes et des décompositions du th. \ref{theoreme proprietes groupe de cordes}. Ces trucs ne fonctionnent pas pour les schémas de Schubert, parce qu'on ne dispose d'aucun élément maximal, pour qu'on puisse raisonner par récurrence descendante sur la longueur. En effet, d'après les découvertes que nous avons faites avec Haines--Richarz dans \cite{HLR}, sans l'hypothèse \og le groupe dérivé est simplement connexe \fg{}, alors les variétés de Schubert en caractéristique mauvaises ne sont typiquement pas normales, voir \cite[th. 2.5]{HLR} et le cor. \ref{corollaire non normalite varietes de schubert}, et la grassmannienne affine $\Gr_{\underline{\Gscr_\fbf}}$ n'est en général pas géométriquement réduite, voir \cite[prop. 7.10]{HLR} et le cor. \ref{corollaire grassmannienne affine reduite}.
\end{rem}

En guise d'application, on obtient le théorème d'uniformisation des $\underline{\Gscr_\fbf}$-fibrés sur la droite projective, dû à Beuaville--Laszlo pour $\on{SL}_n$, cf. \cite[prop. 3.4]{BLConfBlocks}, à Drinfeld--Simpson pour les groupes déployés, cf. \cite[th. 3]{DS}, et à Heinloth pour les groupes tordus modérément ramifiés, voir \cite[th. 4, th. 5]{Hein}.
\begin{cor}
Le champ classifiant $\on{Fib}_{\underline{\Gscr_\fbf}, \PP^1}$ des $\Gscr_\fbf$-fibrés sur la droite projective s'identifie au champ quotient $[L^-\underline{\Gscr_\fbf}\backslash \Gr_{\underline{\Gscr_\fbf}}]$ pour la topologie ouverte.	
\end{cor}
\begin{proof}
La détermination dans la preuve du th. \ref{theoreme cycles de schubert topologie et normalite} des quotients de type fini des ouverts $\Gr_{\underline{\Gscr}}^{v\geq}$, cf. (\ref{equation projection vers quotients locaux de tf grassmannienne affine}) et (\ref{equation quotients locaux des cycles de schubert}), montre que le champ quotient $[L^-\underline{\Gscr_\fbf}\backslash \Gr_{\underline{\Gscr_\fbf}}]$ possède des sections localement pour la topologie ouverte et qu'il est un champ d'Artin lisse. Vu que $\on{Fib}_{\underline{\Gscr_\fbf}, \PP_\ZZ^1}$ est aussi un champ d'Artin lisse d'après \cite[prop. 1]{Hein} et que l'uniformisation \begin{equation}
[L^-\underline{\Gscr_\fbf}\backslash \Gr_{\underline{\Gscr_\fbf}}]\rightarrow\on{Fib}_{\underline{\Gscr_\fbf}, \PP_\ZZ^1}
\end{equation} est un monomorphisme formellement lisse, il suffit de montrer sa surjectivité essentielle. Mais cela résulte d'une application soigneuse du théorème de Steinberg, voir \cite[th. 4]{Hein}, qui reste valable pour les groupes pseudo-réductifs, cf. \cite[prop. 7.3.3., prop. 9.9.2]{CGP} et \cite[prop. 2.2, cor. 2.3]{LouBT}.  
\end{proof}

\begin{rem}
	Signalons que nous avons eu besoin d'éviter le calcul d'espaces tangents à la manière de \cite[lem. 6, lem. 7]{Hein}, parce que le lem. \ref{lemme engendrement algebres de lie} est insuffisant lorsque $\Phi_G$ est non réduit et la caractéristique résiduelle égale à $2$.
	Cependant, nous ne savons pas démontrer le théorème de réduction au Borel, cf. \cite[th. 2]{DS} et \cite[prop. 25]{Hein}, car le quotient $\underline{\Gscr_\fbf}/\underline{\Bscr_\fbf}$ cesse d'être projectif même sur $\GG_{m,\ZZ}$.
\end{rem}

\subsection{Normalité des schémas de Schubert} \label{normalite des varietes de schubert}

Le but de ce numéro est de finalement démontrer le théorème de normalité concernant les schémas de Schubert $\Gr_{\underline{\Gscr_\fbf}}$, cf. th. \ref{theoreme normalite des schemas de Schubert}. Signalons que le résultat dual de normalité pour les quotients locaux des cycles de Schubert du th. \ref{theoreme cycles de schubert topologie et normalite} ne nous aidera pas, vu que les schémas de Schubert ne sont pas invariants sous $L^-\underline{\Gscr_\fbf}$. En lieu de cette notion, il va falloir travailler avec la colimite filtrée des normalisations, dont la première étape consiste à montrer que cela est raisonnable, voir prop. \ref{proposition varietes de schubert proprietes} et cor. \ref{corollaire normalises des schemas de schubert}. Puis on aura besoin de relever au sens formel la cellule ouverte $L^{--}\underline{\Gscr_\fbf}$ de la grassmannienne affine à sa normalisée, voir lem. \ref{lemme engendrement distributions} et (\ref{equation isomorphisme espaces tangents normalise}). Désormais, on fixe une alcôve $\abf$ contenant $\fbf$ dans son adhérence.

\begin{propn}\label{proposition varietes de schubert proprietes}
	Les morphismes de normalisation
	\begin{equation}\label{equation normalisation varietes de schubert}
\widetilde{\Gr}_{\Gscr_{\fbf,p},\leq w}\rightarrow\Gr_{\Gscr_{\fbf,p},\leq w}
	\end{equation} sont des homéomorphismes universels, les applications de transition
	\begin{equation}\label{equation immersions fermees varietes de schubert}
	\widetilde{\Gr}_{\Gscr_{\fbf,p},\leq v}\hookrightarrow\widetilde{\Gr}_{\Gscr_{\fbf,p},\leq w}
	\end{equation}
	 entre les normalisées sont des immersions fermées et la désingularisation de Demazure se factorise à travers les normalisées
	\begin{equation}\label{equation factorisation de stein resolution de demazure}
	\on{Dem}_\wf \rightarrow \widetilde{\Gr}_{\Gscr_{\fbf,p},\leq w}\rightarrow\Gr_{\Gscr_{\fbf,p},\leq w},
	\end{equation} de sorte que la première flèche soit une désingularisation rationnelle.
	
	En caractéristique positive $p>0$, les normalisées $\widetilde{\Gr}_{\Gscr_{\fbf,p},\leq w}$ sont compatiblement scindées les unes avec les autres.
\end{propn}

Rappelons qu'un schéma $X$ de caractéristique positive $p>0$ est scindé si l'homomorphisme de $\Oh_X$-modules $\Oh_X \rightarrow F_*\Oh_X$ correspondant au frobenius absolu possède une section. On renvoie à \cite[\S1]{BK} pour toutes les propriétés basiques de cette notion fondamental pendant la démonstration ci-dessous.

\begin{proof}
	On peut suivre exactement les démonstrations de \cite[lem. 9]{Fal} et/ou de \cite[prop. 9.7]{PR}. La résolution de Demazure \begin{equation}\on{Dem}_{\wf} \rightarrow \Gr_{\Gscr_{\fbf,p}, \leq w} \end{equation}  pour $w\in W^{\on{af}}/W_\fbf $ étant birationnelle et projective, les fibres géométriques réduites s'écrivent comme fibrations itérées en droites projectives, d'où leur connexité. Autrement dit, la variété de Schubert normalisée $\widetilde{\Gr}_{\Gscr_{\fbf,p},\leq w}$ est le membre du milieu de la factorisation de Stein de la résolution de Demazure, cf. (\ref{equation factorisation de stein resolution de demazure}), et les applications de normalisation (\ref{equation normalisation varietes de schubert}) sont des homéomorphismes universels. Il s'ensuit du même raisonnement que les images directes supérieures du faisceau structural de $\on{Dem}_{\wf}$ sont triviales. 
	
	Ensuite, un scindage du frobenius $\Oh \rightarrow F_*\Oh$ de $\on{Dem}_{\wf}$ est construit, en appliquant le critère de Mehta--Ramanathan, voir \cite[prop. 1.3.11]{BK}. Sous-jacent à tout cela, se trouve le point le plus délicat, à savoir, l'existence d'un faisceau inversible ample $\Lh$ sur la grassmannienne affine $\Gr_{\Gscr_{\fbf,p}}$ à degré $1$ sur tous les droites de Schubert, comp. avec \cite[prop. 3.3.11]{GoertzI} et \cite[prop. 2.2.2]{BK}, pour qu'on puisse calculer le fibré canonique de $\on{Dem}_{\wf}$. Mais cela découle manifestement de notre construction des cycles de Schubert, voir th. \ref{theoreme cycles de schubert topologie et normalite} ainsi que le cor. \ref{corollaire groupe de picard}. Ceci comble une lacune de \cite[prop. 9.6]{PR}, dont les auteurs ont négligé que les arguments de Görtz avec des réseaux, voir \cite[\S3.3.6]{GoertzI} sont insuffisantes au cas général. Voir aussi l'erreur de \cite[10.a.1]{PR} signalée par \cite[rem. 2.1]{ZhuCoh}.
	
	Ce scindage est par construction évidemment compatible aux sous-variétés de Demazure \begin{equation}\on{Dem}_{\uf} \subset \on{Dem}_{\wf}\end{equation} et se descend donc en un scindage des variétés de Schubert normalisées. Par suite, les applications de transition (\ref{equation immersions fermees varietes de schubert}) déduites des factorisations de Stein (\ref{equation factorisation de stein resolution de demazure}) sont des immersion fermées, compte tenu de ce que les schémas scindés sont toujours faiblement normaux, cf. \cite[prop. 1.2.5]{BK}. Le cas de la caractéristique nulle découle de la semi-continuité supérieure, cf. \cite[\S1.6]{BK}.
\end{proof}

Revenons aux coefficients entiers et considérons les schémas de Schubert $\Gr_{\underline{\Gscr_\fbf}, \leq w}$ définis sur $\ZZ$. Le premier obstacle qu'on rencontre consiste à que les fibres $\Gr_{\underline{\Gscr_\fbf}, \leq w}\otimes \FF_p$ ne soient pas réduites en caractéristique positive, ce qui sera typiquement le cas quand $p$ divise le cardinal de $\pi_1(G^{\on{dér}})$, voir \cite[th. 2.5, prop. 3.4]{HLR} et le cor. \ref{corollaire non normalite varietes de schubert}. Cependant, pour les normalisés on n'a rien à craindre :

\begin{cor}\label{corollaire normalises des schemas de schubert}
	Considérons la factorisation de Stein sur les entiers 
	\begin{equation}
	\on{Dem}_\wf \rightarrow \widetilde{\Gr}_{\underline{\Gscr_\fbf}, \leq w} \rightarrow \Gr_{\underline{\Gscr_\fbf}, \leq w}
	\end{equation}
	des schémas de Schubert. Alors la deuxième flèche est le morphisme de normalisation, le membre du milieu est géométriquement normal et les morphismes de transition
	\begin{equation}
	\widetilde{\Gr}_{\underline{\Gscr_\fbf}, \leq v} \rightarrow \widetilde{\Gr}_{\underline{\Gscr_\fbf}, \leq w}
	\end{equation} sont des immersions fermées. 	
	En particulier,	leur colimite 
	\begin{equation}
	\widetilde{\Gr}_{\underline{\Gscr_\fbf}}:=\on{colim}_{w \in \widetilde{W}/W_\fbf}\widetilde{\Gr}_{\underline{\Gscr_\fbf},\leq w}
	\end{equation}
	forme un ind-schéma plat opéré par $L^+\underline{\Gscr_\fbf}$.
\end{cor}

\begin{proof}
	On doit observer que l'image directe des sections globales le long de la résolution de Demazure commute aux changements de base quelconques, ce que l'on déduit de l'annulation des images directes supérieures en caractéristique positive, voir la prop. \ref{proposition varietes de schubert proprietes}, \cite[p. 52]{Fal} et \cite[prop. 3.13]{GoertzII}. Les morphismes de transition \begin{equation}\widetilde{\Gr}_{\underline{\Gscr_\fbf}, \leq v}\rightarrow \widetilde{\Gr}_{\underline{\Gscr_\fbf}, \leq w}\end{equation} entre les schémas de Schubert normalisés sont des immersions fermées, car on le sait déjà pour les fibres spéciales. Le mot plat est appliqué pour les ind-schémas au sens de \cite[déf. 8.1]{HLR}. 
	
	Enfin, pour voir que l'opération du groupe d'arcs se relève en la colimite $\widetilde{\Gr}_{\underline{\Gscr_\fbf}}$, il suffit de remarquer que $L^+\underline{\Gscr_\fbf}$ est pro-lisse et que la normalisation faible absolue est fonctorielle, cf. \cite[0EUS]{StaProj}, et commute aux produits géométriquement réduits, cf. \cite[th. III.1, th. V.2]{Man80}.
\end{proof}

Le lemme suivant est aussi fondamental et provient de notre collaboration \cite{HLR} avec Haines--Richarz :

\begin{lem}[Haines--L.--Richarz]\label{lemme platitude des grassmanniennes affines}
La grassmannienne affine $\Gr_{\underline{\Gscr_\fbf}}$ est plate sur $\ZZ$.	
\end{lem}

\begin{proof}
	 La preuve du lemme à la suite de \cite[prop. 8.8]{HLR} se fait en deux étapes. On voit d'abord que le voisinage formel $\spf \Oh$ de la section identité est plate, ce qui peut être ramené à un calcul n'impliquant que la grosse cellule $\underline{\Cscr_\fbf}$, cf. \cite[lem. 8.4]{HLR}. 
	 
	 À ce stade, on pourrait imaginer que tout serait fini, cf. la preuve du cor. \ref{corollaire immersion ouverte}. Néanmoins, l'adhérence plate n'est pas forcément représentable, donc il faut évoquer une technique légèrement plus sophistiquée due à Faltings, cf. \cite[preuve du cor. 11]{Fal} et \cite[lem. 8.6]{HLR}. Celle-ci affirme que, pour que notre plongement ind-fermé soit un isomorphisme, il suffit d'étudier les points à valeurs dans les anneaux artiniens locaux. Comme ceux-ci peuvent être translatés vers le voisinage formel de l'unité, grâce à la décomposition de Bruhat sur les corps, voir \cite[th. 6.5.1]{BTI} et \cite[prop. 3.7]{LouBT}, on en obtient la platitude de la grassmannienne affine.
\end{proof}

Il sera essentiel pour nous de savoir en avance que les variétés de Schubert sont normales en caractéristique nulle, voir la démonstration ci-dessous du th. \ref{theoreme normalite des schemas de Schubert}. L'argument suivant est dû à Kumar, voir \cite[th. 2.16]{KumDem}, dans le cadre de Kac--Moody et a été adapté par Faltings, voir \cite[p. 52]{Fal}, aux grassmanniennes affines de groupes déployés et puis répété dans le cas tordu par Pappas--Rapoport, voir \cite[9.f]{PR}.

\begin{lem}[Kumar, Faltings, Pappas--Rapoport]\label{lemme normalite en caracteristique nulle}
Les variétés de Schubert $\Gr_{\Gscr_{\fbf,0},\leq w}$ sont normales en caractéristique nulle.	
\end{lem}

\begin{proof}
	Soit $G$ simplement connexe sans perdre de généralité, d'où le fait que la grassmannienne affine soit réduite, voir \cite[prop. 9.9]{PR} ainsi que le cor. \ref{corollaire grassmannienne affine reduite}. Soit $\Lh$ un faisceau inversible effectif sur $\Gr_{\Gscr_{\fbf,0}}$ induit par le th. \ref{theoreme cycles de schubert topologie et normalite}, cf. aussi cor. \ref{corollaire groupe de picard}. Alors, $\Lh$ est par construction sans points de base en un ouvert non vide et même partout par translation, cf. \cite[th. 7]{Fal}. En passant à la normalisée $\widetilde{\Gr}_{\Gscr_{\fbf,0}}$, alors le groupe de Picard devient engendré par ces faisceaux effectifs, voir le cor. \ref{corollaire groupe de picard}, ce qui permet de déduire l'ampleur de ces derniers sur la grassmannienne $\Gr_{\Gscr_{\fbf,0}}$ elle-même. Finalement, on remarque que l'opération naturelle de $L^-\underline{\Gscr_\abf}$ sur $\Lh$ se prolonge et se relève en une action de l'extension centrale universelle
	\begin{equation}\label{equation extension centrale rationnels}1 \rightarrow \GG_{m,\QQ} \rightarrow \widehat{LG} \rightarrow LG \rightarrow 1,\end{equation}
	par des raisonnements généraux, voir aussi le cor. \ref{corollaire charge centrale}.
	
	Considérons maintenant l'application \begin{equation}\label{equation espaces vectoriels duaux sections globales}\Gamma(\widetilde{\Gr}_{\Gscr_{\fbf,0}},\Lh)^\vee \rightarrow \Gamma(\Gr_{\Gscr_{\fbf,0}},\Lh)^\vee \end{equation} de $\QQ$-espaces vectoriels non nuls déduite de la colimite des mêmes morphismes pour les variétés de Schubert. On en déduit une action de l'algèbre de Lie de l'extension centrale $\widehat{LG}$ de (\ref{equation extension centrale rationnels}) sur les $\QQ$-espaces vectoriels non nuls de (\ref{equation espaces vectoriels duaux sections globales}). On peut même montrer que les deux $\on{Lie}\widehat{LG}$-modules sont irréductibles au même plus grand poids au sens de la théorie de Kac--Moody, donc isomorphes, cf. \cite[9.f]{PR}. Comme $\Lh$ était arbitraire, on achève facilement la conclusion voulue que les variétés de Schubert sont normales, cf. \cite[p. 52]{Fal}. Pour une preuve plus élémentaire mais moins éclairante, voir aussi \cite[annexe]{KumGr}.
\end{proof}

Récapitulons les faits qui ont été établis ci-dessus : on a un morphisme $L^+\underline{\Gscr_\fbf}$-équivariant de normalisation \begin{equation}\widetilde{\Gr}_{\underline{\Gscr_\fbf}} \rightarrow \Gr_{\underline{\Gscr_\fbf}}\end{equation} entre ind-schémas plats, voir le cor. \ref{corollaire normalises des schemas de schubert} et le lem. \ref{lemme platitude des grassmanniennes affines}, qui est un isomorphisme sur le réduit du membre de droite en caractéristique nulle, voir le lem. \ref{lemme normalite en caracteristique nulle}. Typiquement, la grassmannienne n'est pas réduite, voir \cite[th. 6.1, prop. 6.5]{PR} ainsi que le cor. \ref{corollaire grassmannienne affine reduite}, si $G$ n'est pas semi-simple, et il y a aussi des problèmes de non normalité en caractéristiques mauvaises selon notre collaboration avec Haines--Richarz, voir \cite[th. 2.5]{HLR} et le cor. \ref{corollaire non normalite varietes de schubert}. Ainsi, on suppose dorénavant que \begin{equation}G =G^{\on{sc}}\end{equation} est simplement connexe, ce qui est évidemment suffisant afin de montrer le th. \ref{theoreme normalite des schemas de Schubert}.

Comme nous avons dit auparavant, notre but est de construire l'application réciproque du morphisme naturel \begin{equation}\label{equation morphisme voisinage formel de la normalisation}\spf \widetilde{\Oh} \rightarrow \spf \Oh \end{equation} entre les voisinages formels de l'originel, voir la démonstration ci-dessous du th. \ref{theoreme normalite des schemas de Schubert}. Tous les deux membres sont des schéma en groupes formels affines de type infini, voir le cor. \ref{corollaire immersion ouverte} et appliquer le foncteur de normalisation faible absolue, cf. \cite[th. III.1, th. V.2]{Man80}. Ce raisonnement entraîne, de plus, que le morphisme (\ref{equation morphisme voisinage formel de la normalisation}) soit aussi opéré par tout groupe d'arcs parahorique $L^+\underline{\Gscr_\gbf}$ associé aux facettes $\gbf \subset \Ascr(\underline{G},\underline{S},\ZZ\rpot{t})$, et en particulier, par les groupes de lacets radiciels $L\underline{U_a}$. Ce fait sera exploité décisivement pour déduire que les algèbres de Lie à coefficients dans $\ZZ$ de ces deux groupes infinis sont isomorphes, voir les lems. \ref{lemme engendrement algebres de lie}, \ref{lemme engendrement distributions} et (\ref{equation isomorphisme espaces tangents normalise}).

Afin de motiver la discussion à venir, voici une adaptation du lemme-clef de Pappas--Rapoport, voir \cite[prop. 9.3]{PR}, que nous avons déjà mentionné ailleurs.

\begin{lem}\label{lemme engendrement algebres de lie}
	Supposons que $G=G^{\on{sc}}$ est simplement connexe et soit $R$ n'importe quel anneau. Alors, la $R$-algèbre de Lie de $L^{\on{gr}}\underline{G}$ est engendrée par celles de $L^{\on{gr}}\underline{U_\pm}$, si et seulement si $\Phi_G$ est réduit ou si $2 \in R^\times$ est inversible.
\end{lem}

\begin{proof}
	On s'est ramené au cas où $G$ est de rang relatif $1$ et absolument simple, c'est-à-dire soit $\on{SL}_2$, soit $\on{SU}_{3}$. Dans le cas non pluriel, on a $[t^ne_a,t^me_{-a}]=\pm t^{n+m}h_a$, où les $e_a$ engendrent l'algèbre de Lie de $\underline{U_a}$ et les $h_a$ de $\underline{T}$, comme d'habitude. Dans le cas pluriel, on peut choisir les générateurs $e_{\pm a}$ et $e_{\pm 2a}$ des espaces de poids de $\on{Lie}\underline{G}$, de telle manière que $[t^ne_a,t^me_{-a}]=\pm 2t^{n+m}h_a$ et $[t^ne_{2a},t^me_{-2a}]=\pm t^{n+m}h_a$, cf. (\ref{equation diviseur su3}) et (\ref{equation echange su3}). En particulier, si $2 \notin R^\times$, on voit que la sous-algèbre de Lie engendré par les espaces de poids non nuls ne contient pas les éléments de la forme $t^{n}h_a$, où $n\in \frac{1}{2}\ZZ \setminus \ZZ$ décrit l'ensemble des demi entiers non entiers. 
\end{proof}

En conclusion, le crochet de Lie ne suffit pas pour qu'on retrouve l'algèbre de Lie de $\spf \Oh$ toute entière, en n'opérant qu'avec les sous-groupes radiciels, ne soit qu'on exclut les systèmes de racines non réduits (c'était le choix de la version préliminaire de cet article) ou que la caractéristique paire soit à éviter (moins souhaitable pour nous). La solution consiste à substituer des opérateurs différentiels supérieurs aux linéaires. 

Autrement dit, on fera usage de l'algèbre des distributions $\on{Dist}\Oh$ constituée des opérateurs linéaires continus $\delta:\Oh \rightarrow \ZZ$ qui s'annulent sur une puissance de l'idéal $\Ih$ définissant la section origine. Pour la notion habituelle en géométrie algébrique et ses propriétés, on renvoie le lecteur à \cite[\S1.3]{BTII} et à \cite[I, \S7]{Jan} ; pour les ind-schémas au-dessus des corps, la discussion du concept au début de \cite[\S7.1]{HLR} est aussi recommandée.

\begin{lem}\label{lemme engendrement distributions}
	Gardons les notations du lem. \ref{lemme engendrement algebres de lie}. Alors, la $R$-algèbre de distributions de $L^{\on{gr}}\underline{G}$ est engendrée par celles de $L^{\on{gr}}\underline{U_\pm}$.
\end{lem}

\begin{proof}
	Alors, l'algèbre de distributions $\on{Dist}L^{\on{gr}}\underline{G}$ ne dépend que de l'espace d'élastiques $L^{\on{gr}}\underline{C}$ de la grosse cellule $\underline{C} \subset \underline{G}$. Comme les distributions transforment les produits directs en tensoriels, dès que les voisinages infinitésimaux de l'origine soient des $R$-modules projectifs, voir \cite[1.3.5]{BTII}, on en tire que l'algèbre $\on{Dist}L^{\on{gr}}\underline{G}$ est engendrée par celles de $L^{\on{gr}}\underline{T}$ et de $L^{\on{gr}}\underline{U_a}$, où $a$ décrit la partie $\Phi_G^{\on{nd}}$ des racines non divisibles. 
	
	Comme d'habitude, l'affirmation se réduit aux cas absolument simples de rang relatif $1$ et on commence par déduire du morphisme échange de l'axiome (DRS 3), voir les formules explicites \ref{equation echange sl2} et \ref{equation echange su3}, l'existence d'un morphisme rationnel :
	\begin{equation}\underline{U_a}\times \underline{U_{-a}}\dashrightarrow \underline{U_{-a}}\times \underline{U_a}\times \underline{U_{-a}}\times \underline{U_a} \rightarrow \underline{G},\end{equation}
	où la deuxième flèche est la multiplication et la composée se factorise à travers $\underline{T}$. Ensuite, nous prenons $L^{\on{gr}}$ et complétons autour de la section unité pour obtenir une application définie partout au sens formel ; nous allons montrer que le morphisme induit des distributions est surjectif. 
	
	Les équations qui décrivent l'échange pour le groupe $\on{SL}_2$ son si simples que le soin des calculs pertinents sera laissé au lecteur. Dans le cas de $\on{SU}_3$, on pose soit \begin{equation}\label{equation valeurs choisis calcul des distributions quadratique}u=v'=0, u'=at^n,v=bt^m,\end{equation} soit \begin{equation}\label{equation valeurs choisis calcul des distributions lineaire}u=u'=0, v=at^n,v'=bt^m.\end{equation} Si l'on identifie le sous-groupe des éléments $1-rt^n$ dans le complété de $L^{\on{gr}}\underline{T}$ à la droite affine formelle, alors on obtient un morphisme plat \begin{equation}\widehat{\AAA^2_\ZZ} \rightarrow \widehat{\AAA^1_\ZZ}\end{equation} donné lorsqu'on se trouve dans le cas (\ref{equation valeurs choisis calcul des distributions quadratique}) par \begin{equation}(a,b)\mapsto a^2b,\end{equation} soit par \begin{equation}(a,b)\mapsto ab\end{equation} lorsqu'on se trouve dans le cas (\ref{equation valeurs choisis calcul des distributions lineaire}). Les applications des modules de distributions qui en résultent sont en tout cas surjectives.	
\end{proof}

Ce lemme combiné avec l'équivariance de la normalisation entraîne que \begin{equation}\label{equation surjectif distributions normalisation}\on{Dist}\spf\widetilde{\Oh}\twoheadrightarrow \on{Dist}\spf\Oh\end{equation} est même surjectif. D'autre part, les algèbres de distributions sont toujours sans torsion et on sait déjà que le $\QQ$-localisé de l'homomorphisme (\ref{equation surjectif distributions normalisation}) est bijectif, cf. lem. \ref{lemme normalite en caracteristique nulle}, lem. \ref{lemme platitude des grassmanniennes affines} et cor. \ref{corollaire grassmannienne affine reduite}, d'où la bijectivité à coefficients entiers de (\ref{equation surjectif distributions normalisation}). 

Puis on observe sans peine que l'algèbre de Lie
\begin{equation}
\on{Lie}\spf\Oh \subset \on{Dist}\spf\Oh
\end{equation} est un sous-module saturé de l'algèbre de distributions, c'est-à-dire le conoyau est sans torsion : en effet, si $n\delta$ s'annule sur $\Ih^2$, alors la distribution $\delta$ s'annule aussi sur $\Ih^2$. Par suite, l'homomorphisme
 	\begin{equation}\label{equation isomorphisme espaces tangents normalise} \on{Lie}\spf\widetilde{\Oh}\xrightarrow{\sim} \on{Lie}\spf\Oh
\end{equation} est bijectif.
Maintenant, on est prêt à prouver le théorème de normalité des schémas de Schubert, cf. th. \ref{theoreme normalite des schemas de Schubert}.

\begin{proof}[Démonstration du th. \ref{theoreme normalite des schemas de Schubert}]
On affirme qu'il suffit de construire une section de l'homomorphisme de normalisation \begin{equation}\label{equation normalisation anneaux locaux completes schubert}\Oh_{w} \rightarrow \widetilde{\Oh}_{w} \end{equation} entre les anneaux locaux séparés complétés du schéma de Schubert associé à $w$ et de son normalisé le long de leurs sections origines.
 En effet, le support du conoyau des faisceaux structuraux serait ainsi trivial, car stable par $L^+\underline{\Gscr_\abf}$ et fermé.
 
 On procède par un argument inductif de déformation, comp. avec \cite[p. 53]{Fal} et \cite[9.g]{PR}. Grâce à la résolution de Demazure, l'anneau $\Oh_w$ peut être filtré par une suite d'idéaux décroissante et exhaustive
 \begin{equation}
  0 \subset \dots \subset I_{n+1} \subset I_n \subset I_{n-1} \subset \dots \subset \Oh_w,
 \end{equation}
  telle que leurs conoyaux $A_n$ soient plats sur $\ZZ$ ayant $I_{n-1}$ pour idéal de carré nul. Alors, le point canonique $s_n$ de $\spf \Oh_w$ à valeurs dans $A_n$ peut être relevé uniquement en un point $\widetilde{s}_n$ de $\spf \widetilde{\Oh}_w$. 
  
  En effet, comme les espaces tangents commutent aux changements de base plats, on peut appliquer le critère de lissité formelle, ainsi que l'isomorphisme (\ref{equation isomorphisme espaces tangents normalise}), pour relever uniquement $s_{n+1}$ à $\spf \widetilde{\Oh}$. Vu que ce relèvement n'est autre que l'identité en caractéristique nulle, voir le lem. \ref{lemme normalite en caracteristique nulle}, la section $\widetilde{s}_n \in \spf \widetilde{\Oh}(A_n)$ doit se factoriser à travers $\spf \widetilde{\Oh}_w$. En particulier, les $\widetilde{s}_n$ s'assemblent les unes avec les autres pour donner un morphisme $\spec \Oh_w \rightarrow \spec \widetilde{\Oh}_w$, par le lemme de Chevalley.
\end{proof}

En collaboration avec Haines--Richarz, voir \cite{HLR}, nous avons déterminé tous les groupes $G$ dont les variétés de Schubert en caractéristique positive sont toutes normales, sous des hypothèses de ramification modérée, cf. \cite[prop. 2.1, th. 2.5]{HLR}. Ceci s'étend aussi au présent cadre.

\begin{cor}[Haines--L.--Richarz]\label{corollaire non normalite varietes de schubert}
	Pour que les variétés de Schubert $\Gr_{\Gscr_{\fbf,p},\leq w}$ soient toutes normales, il faut et il suffit que $p$ ne divise pas l'ordre du groupe fondamental $\pi_1(G^{\on{dér}})$ du groupe dérivé. 
\end{cor}

\begin{proof}
	L'affirmation cruciale qu'on va prouver est que le noyau de \begin{equation}G^{\on{sc}}_{\eta_p}\rightarrow G_{\eta_p}\end{equation} est étale si et seulement si $p$ ne divise pas l'ordre $n$ du groupe fondamental du dérivé de $G$. En effet, on voit par réduction au cas déployé, cf. prop. \ref{proposition plongement restrictions des scalaires groupe de tits}, que le noyau générique modulo $p$ est contenu dans $T^{\on{sc}}_{\eta_p}[n]$. Si $(p,n)=1$, ce sous-groupe de $n$-torsion est fini étale. Sinon, on peut déduire grâce à \cite[ass. 12]{Hein} que le noyau contient toutefois un sous-groupe fini et multiplicatif d'ordre $n$, donc non étale. Armés de telle information, on peut appliquer l'argument principal de \cite[prop. 2.1]{HLR} pour déduire l'existence d'une obstruction dans les cas souhaités.
\end{proof}

Le corollaire suivant est dû à plusieurs mathématiciens qui l'ont démontré à des degrés divers de généralité.
 
\begin{cor}[Faltings, Pappas--Rapoport, Haines--L.--Richarz]\label{corollaire grassmannienne affine reduite}
La grassmannienne affine $\Gr_{\underline{\Gscr_\fbf}}$ est réduite si et seulement si $G$ est semi-simple. Pour que ses fibres soient aussi réduites, il faut et il suffit que $G$ soit simplement connexe.
\end{cor}

\begin{proof}
	On prend l'opportunité pour reprouver que la grassmannienne affine est géométriquement réduite au cas où $G$ est simplement connexe, sans recourir ni à \cite[cor. 11]{Fal} en caractéristique nulle ni au lem. \ref{lemme platitude des grassmanniennes affines}. Le voisinage formel $\spf \Oh_p$ de la grassmannienne affine $\Gr_{\Gscr_{\fbf,p}}$ modulo $p$ est contenu dans la colimite des schémas de Schubert, parce que tous les deux sont opérés par les groupes de lacets radiciels, donc les algèbres de distributions coïncident, voir le lem. \ref{lemme engendrement distributions}. Les points à valeurs dans les anneaux artiniens locaux se translatent sans peine vers ce voisinage formel, d'où notre affirmation. 
	
	En caractéristique nulle, on peut montrer aisément que $\Gr_{\Gscr_{\fbf,0}}$ est réduite si et seulement si $G$ est semi-simple, voir \cite[th. 6.1, prop. 6.5]{PR}, d'où l'assertion à coefficients entiers grâce au lem. \ref{lemme platitude des grassmanniennes affines}. Lorsque $G$ est semi-simple à groupe fondamental $\pi_1(G)$ non trivial, un nombre premier $p$ divisant le cardinal en est choisi et nous observons que le noyau de $G^{\on{sc}}_{\eta_p}\rightarrow G_{\eta_p}$ n'est pas étale, à cause de \cite[ass. 12]{Hein}. Par conséquent, il s'avère par une application soigneuse des arguments de \cite[\S7]{HLR} que $\Gr_{\Gscr_{\fbf,p}}$ n'est pas réduite. Il semble que ceci pourrait aussi découler de l'étude du morphisme d'algèbre de distributions entières, mais nous n'avons pas pu montrer que le conoyau est sans torsion divisible, comme par exemple $\QQ/\ZZ$.
\end{proof}

Cette section se conclut par donner une application à la géométrie des schémas de Richardson $\Gr_{\underline{\Gscr_\fbf}, \leq v}^{\geq w}$ sur les entiers, cf. (\ref{equation definition schema de richardson}), qui n'est parue dans la littérature des groupes de Kac--Moody affines, tant que nous le savons, cf. aussi \cite[rem. 5.10]{KS14}, mais qui est probablement connue parmi les experts, voir \cite[lem. 1]{BL01} pour les variétés de drapeaux finies et \cite[th. 3.1]{KS14} en caractéristique nulle.

\begin{cor}[Brion--Lakshimibai, Mathieu--Kumar--Schwede, L.]\label{corollaire schemas de richardson}
	Sous l'hypothèse $G^{\on{dér}}=G^{\on{sc}}$, alors les schémas de Richardson sont plats, intègres de dimension relative $l(v)-l(w)$, de Cohen--Macaulay et géométriquement normaux à fibres compatiblement scindées en caractéristique positive.
\end{cor}

\begin{proof}
	La preuve qui va suivre est reprise de \cite[lem. 1]{BL01}. On considère la fibration 
	\begin{equation}L\underline{G}\times^{L^+\underline{\Gscr_\abf}}\Gr_{\underline{\Gscr_\fbf}, \leq v} \rightarrow \Gr_{\underline{\Gscr_\fbf}}\end{equation} 
	localement triviale pour la topologie ouverte, voir cor. \ref{corollaire immersion ouverte}, dont les fibres sont isomorphes à certaines variétés de Schubert. Puis on prendre l'ind-schéma $Z$ déduit par changement de base le long de $\Gr_{\underline{\Gscr_\fbf}}^{\geq w} $ et qui est naturellement opéré par $L^-\underline{\Gscr_\abf}$. D'autre part, l'on a une application projection canonique de $Z$ vers $\Gr_{\underline{\Gscr_\abf}}$, dont le produit fibré avec la section origine s'identifie sans peine au schéma de Richardson $\Gr_{\underline{\Gscr_\fbf}, \leq v}^{\geq w}$. Évoquant la $L^-\underline{\Gscr_\abf}$-équivariance de ce morphisme, nous arrivons à ce que l'image réciproque dans $Z$ de la cellule ouverte $L^{--}\underline{\Gscr_\abf}$ de $\Gr_{\underline{\Gscr_\abf}}$ est isomorphe à son produit avec $\Gr_{\underline{\Gscr_\fbf}, \leq v}^{\geq w}$. En particulier, les propriétés géométriques qui ont été énoncées sauf le scindage découlent très vite de leurs analogues pour les schémas de Schubert, cf. th. \ref{theoreme normalite des schemas de Schubert}, ainsi que pour les quotients locaux des cycles de Schubert, cf. th. \ref{theoreme cycles de schubert topologie et normalite}. 
	
	Le part sur le scindage est un théorème non publié de Mathieu, qui fut rédigé par Kumar--Schwede en \cite[prop. 5.3]{KS14}. L'outil principal est une notion de scindage $L^+\underline{\Gscr_\abf}$-canonique au sens de \cite[déf. 4.1.1]{BK}, en termes du sous-tore maximal $\underline{S}$ et des sous-groupes radiciels $\underline{V_\alpha}$ associés aux racines réelles simples positives $\alpha \in \Delta_G^{\on{ré}}$. On s'en sert pour montrer que les cycles de Schubert en caractéristique $p$ sont compatiblement scindés, en appliquant l'opération du Borel opposé $L^-\underline{\Gscr_\abf}$, donc il en est de même des schémas de Richardson. 
\end{proof}

On obtient la description ci-dessous du groupe de Picard des schémas de Schubert $\Gr_{\underline{\Gscr_{\fbf}},\leq w}$.

\begin{cor}\label{corollaire groupe de picard}
	Si $G^{\on{dér}}=G^{\on{sc}}$, alors le foncteur de Picard rigidifié
	\begin{equation}
	\on{Pic}(\Gr_{\underline{\Gscr_\fbf}, \leq w}) \simeq  \prod_{s \leq w, l(s)=1}\ZZ \cdot\Lh_s
	\end{equation} est localement libre en les faisceaux inversibles  induits par les diviseurs de Richardson $\Gr_{\underline{\Gscr_\fbf}, \leq w}^{\geq s}$, où la longueur de $s \leq w$ est égale à $1$. 
\end{cor}

\begin{proof}
	Remarquant que les quotients locaux de (\ref{equation quotients locaux des cycles de schubert}) avec $n$ assez grand définissent des $1$-cycles intègres dans des schémas lisses, on arrive aux fibrés en droites $\Lh_s$ sur $\Gr_{\underline{\Gscr_\fbf}}$ munis d'une section globale invariante par $L^-\underline{\Gscr_\abf}$. 
	D'après le cor. \ref{corollaire schemas de richardson}, nous savons que le degré de $\Lh_s$ restreint à la droite projective $\Gr_{\underline{\Gscr_\fbf}, \leq t}$ est égal à $1$ ou $0$ selon que $s$ et $t$ sont égaux ou différents.
	
	Il reste à vérifier que les degrés d'un fibré $\Lh$ sur les droites projectives stables sous $L^+\underline{\Gscr_\abf}$ le caractérisent à isomorphisme près. Ceci n'est qu'un exercice en utilisant les résolutions de Demazure, en vertu du th. \ref{theoreme normalite des schemas de Schubert}, cf. \cite[cor. 12]{Fal} et \cite[XII, prop. 6]{Mat}. 	
\end{proof}

Typiquement l'opération de $L^{--}\underline{\Gscr_\abf}$ sur les fibrés en droites ne s'étend pas au groupe de lacets $L\underline{G}$, voir le cor. \ref{corollaire charge centrale} ci-dessous, mais juste à une extension centrale de ceci par le groupe multiplicatif. En effet, à chaque faisceau inversible $\Lh$ sur la grassmannienne affine $\Gr_{\underline{\Gscr_\fbf}}$, on lui associe le foncteur \begin{equation}\label{equation extension centrale associe au fibre en droites} \widehat{L\underline{G}}\{\Lh\}: R \mapsto (g, \alpha), \end{equation} qui paramétrise les éléments $g \in L\underline{G}(R)$ joints aux isomorphismes $g^*\Lh\xrightarrow{\alpha}\Lh$. Il s'ensuit aisément du th. \ref{theoreme normalite des schemas de Schubert} que, si $G=G^{\on{sc}}$ est simplement connexe, alors le foncteur $\widehat{L\underline{G}}\{\Lh\}$ est un $\ZZ$-ind-groupe extension centrale de $L\underline{G}$ par $\GG_{m,\ZZ}$.

\begin{cor}\label{corollaire charge centrale}
	Si $G=G^{\on{sc}}$ est presque simple, alors la correspondance \begin{equation}\Lh \mapsto \widehat{L\underline{G}}\{\Lh\}\end{equation} induite par (\ref{equation extension centrale associe au fibre en droites}) définit un homomorphisme \begin{equation}\label{equation homomorphisme charge centrale facette quelconque}c_{\fbf} :\on{Pic}(\Gr_{\underline{\Gscr_\fbf}}) \rightarrow \on{Ext}_{\on{cent}}(\GG_{m,\ZZ},L\underline{G}),\end{equation}
	dont l'image est libre de rang $1$.
\end{cor}

\begin{proof}
	Que l'extension centrale (\ref{equation extension centrale associe au fibre en droites}) soit scindée revient au même que dire que le faisceau inversible $\Lh$ admette une structure $L\underline{G}$-équivariante. Compte tenu de ce que le groupe des caractères de $L\underline{G}$ est trivial, puisque ce dernier est schématiquement engendré par $L\underline{U_\pm}$, comp. avec la preuve de prop. \ref{proposition dynamique groupes d'arcs et de cordes}, on voit que la structure équivariante est uniquement déterminée, lorsqu'elle existe, c'est-à-dire l'application
	\begin{equation}\label{equation homomorphisme groupe de picard equivariant}\on{Pic}^{L\underline{G}}(\Gr_{\underline{\Gscr_\fbf}})\hookrightarrow \on{Pic}(\Gr_{\underline{\Gscr_\fbf}})\end{equation} 
	est injective dont l'image s'identifie au noyau de l'homomorphisme (\ref{equation homomorphisme charge centrale facette quelconque}). Le raisonnement ci-dessus montre aussi que $c_\fbf$ commute aux changements de facettes.
	
	 Comme l'origine est stable par $L^+\underline{\Gscr_\fbf}$, une telle structure équivariante donne un caractère entier \begin{equation}\chi: L^+\underline{\Gscr_\fbf} \rightarrow \underline{M_\fbf} \rightarrow \GG_{m,\ZZ}\end{equation} du groupe d'arcs par restriction à la fibre de $\Lh$. Réciproquement, un $\ZZ$-caractère $\chi$ de $\underline{M_\fbf}$ fournit par inflation à $L^+\underline{\Gscr_\fbf}$ et puis par induction à $L\underline{G}$ un faisceau inversible \begin{equation}\Lh_\chi:=L\underline{G}\times^{L^+\underline{\Gscr_\fbf}}\Oh_\chi\end{equation} sur la grassmannienne affine muni d'une opération compatible par le groupe de lacets $L\underline{G}$, cf. \cite[prop. 1.5]{Iv76}. En particulier, il en résulte que le groupe de Picard équivariant
	 \begin{equation}\label{equation fibres en droite equivariants grassmannienne}
	 \on{Pic}^{L\underline{G}}(\Gr_{\underline{\Gscr_\fbf}}) \simeq X^*(\underline{M_\fbf})
	 \end{equation} s'identifie au groupe de caractères du groupe épinglé $\underline{M_\fbf}$, dont le rang coïncide avec la dimension de $\fbf$, tandis que le rang de $\on{Pic}(\Gr_{\underline{\Gscr_\fbf}})$ est strictement supérieur par une unité, cf. cor. \ref{corollaire groupe de picard}. 
	 
	 Pour voir que l'image de $c_\fbf$ est libre, on peut supposer que $\fbf=\abf$ est une alcôve. Appliquant encore la même identification (\ref{equation fibres en droite equivariants grassmannienne}) pour calculer le groupe de Picard des droites de Schubert, on remarque alors sans peine que (\ref{equation homomorphisme groupe de picard equivariant}) s'écrit au signe près comme 
	 \begin{equation}\label{equation charge centrale formule explicite coefficients kac-moody}
	 \bigoplus_s\pm a_s^{\vee} : X^*(\underline{S}) \rightarrow{} \ZZ\Lh_s,
	 \end{equation} où $a_s \in \Phi_{\fbf_s}$ en est la seule racine positive.
	 On laisse le soin au lecteur de montrer que le membre de gauche se projette isomorphiquement sur le quotient du membre de droite obtenu en oubliant le faisceau inversible $\Lh_s$ correspondant à la réflexion simple $s$ ne fixant pas un point spécial $x$ adhérente à $\abf$, dont le système de racines résiduel $\Phi_x=\Phi_G^{\on{nm}}$ ne contienne que les racines non multipliables.
\end{proof}

Désormais, pour abréger, on note $c$ le seul épimorphisme \begin{equation}\label{equation charge centrale entiers}\on{Pic}(\Gr_{\underline{\Gscr_\abf}}) \rightarrow \ZZ\end{equation} pour l'alcôve donnée $\abf$, dont le noyau coïncide avec le sous-groupe des faisceaux inversibles qui admettent une structure $L\underline{G}$-équivariante, cf. (\ref{equation homomorphisme groupe de picard equivariant}) et tel que les faisceaux en droites semi-amples soient appliqués sur $\ZZ_{\geq0}$. L'entier $c(\Lh)$ s'appelle la charge centrale de $\Lh$ et l'extension centrale $\widehat{L\underline{G}}$ de \ref{equation extension centrale associe au fibre en droites} obtenu en prenant $\Lh$ tel que $c(\Lh)=1$ s'appelle l'extension centrale universelle de $L\underline{G}$.

\section{Déformation globale sur la droite affine cyclotomique}\label{section grassmanniennes beilinson-drinfeld}

La partie présente est consacrée à l'étude des déformations à la Beilinson--Drinfeld des grassmanniennes affines associées aux groupes parahoriques $\underline{\Gscr_\fbf}$ sur la droite affine à coefficients cyclotomiques entiers, cf. prop. \ref{proposition drs modele immobilier groupe de tits}. On démontre notamment leur projectivité et aussi un théorème de cohérence à la Zhu, cf. \cite[ths. 2, 3]{ZhuCoh} à coefficients cyclotomiques entiers. Ceci sera appliqué dans un travail futur pour répondre à la conjecture de Scholze--Weinstein sur les modèles locaux, cf. \cite[conj. 21.4.1]{SchBerk}, dans des cas non modérément ramifiés, voir aussi \cite[IV, \S4.2]{LouDiss}.

\subsection{Projectivité et le lieu admissible}\label{grassmannienne de beilison-drinfeld et ind-projectivite} 

La grassmannienne affine qui fut étudiée au \S\ref{section grassmanniennes affines} était de nature locale, en ce qui concerne sa dépendance de la donnée d'une uniformisante, cf. déf. \ref{definition grassmannienne affine locale}. Si nous voulons la déformer en un équivalent globale, il faut faire varier l'uniformisante, et l'on arrive à la définition suivante due à Beilinson--Drinfeld, cf. \cite[4.3.14]{BDr} :

\begin{defn}\label{definition grassmannienne affine globale}
Soient $A$ un anneau, $X$ une courbe lisse au-dessus de $A$ et $\Grm$ un schéma en groupes sur $X$. La grassmannienne affine globale $\on{GR}_{\Grm}$ associée à cette donnée est le pré-faisceau sur $X$ qui, à chaque application $Y\rightarrow X$ quasi-compacte de $A$-schémas, lui associe les classes d'isomorphisme de $\Grm$-torseurs au-dessus de $X\otimes_A Y$ munis d'une trivialisation en dehors du graphe $\Gamma_{Y/X}$ de $Y\rightarrow X$. 
\end{defn}

On dispose aussi pour ces objets globaux des groupes d'arcs et de lacets 
\begin{equation}
L^+\Grm: R \mapsto \Grm(R\otimes_A X ),\end{equation}\begin{equation}
L \Grm: R \mapsto \Grm(R\otimes_A X \setminus \Gamma_{R/X}),
\end{equation} dont le quotient pour la topologie étales'identifie à $\on{GR}_{\Grm}$, cf. \cite[lem. 1.12]{Rich}. Le produit fibré de $\on{GR}_{\Grm}$ avec une section $x:\spec A \rightarrow X$ devient ainsi isomorphe, d'après la descente de Beauville--Laszlo, à la grassmannienne affine locale $\Gr_{\Grm_{X,x}}$ associée au $A\pot{z_x}$-groupe $\Grm_{X,x}$ déduit de $\Grm$ par changement de base, où $A\pot{z_x}$ est le voisinage formel de $X$ en la section $x$. Nous avertissons le lecteur que, si l'on applique cette identification pour étudier les fibres géométriques de $\on{GR}_{\Grm}$ sur les points génériques d'une fibre $X_k$ de $X\rightarrow \spec A$, alors il pourrait être compliqué de saisir correctement le séparé complété l'uniformisante $z_x$ résultante et que ceci a déjà provoqué d'ailleurs quelques maladresses dans la littérature, cf. \cite{RichErr}.

\begin{propn}[Pappas--Zhu]\label{proposition representabilite grassmannienne affine globale}
	Soient $A$ un anneau de Dedekind, $X$ une $A$-courbe lisse et géométriquement connexe et $\on{G}$ un $X$-groupe lisse, affine et connexe. Alors, la grassmannienne affine globale $\on{GR}_{\on{G}}$ est représentable par un ind-schéma séparé de type fini.
\end{propn}

\begin{proof}L'assertion se démontre par la même procédure de la prop. \ref{proposition representabilite grassmannienne affine locale}.
\end{proof}

On veut analyser les grassmanniennes affine globales de nos modèles parahoriques $\underline{\Gscr_\fbf}$ des groupes de Tits $\underline{G}$ sur la droite affine $\AAA^1_\ZZ$.

\begin{thm}[Pappas--Zhu, L.]\label{theoreme grassmannienne globale projective}
La grassmannienne affine globale $\on{GR}_{\underline{\Gscr_\fbf}}$ est projective.
\end{thm}

Vérifions tout d'abord que les fibres se comportent comme envisagé\footnote{Mentionnons que le lemme ci-dessous a pris la plupart des experts par surpris : en effet, le \cite[th. 1.19]{Rich} affirme que $G_{\eta_p}$ devrait être toujours réductif pour que cela se vérifie, et dont la preuve suppose erronément que cette grassmannienne est constante. On renvoie le lecteur à l'errata \cite{RichErr} qui fut originalement motivée par notre découverte et au \cite[cor. 5.3]{LouBT}, lequel est étroitement lié à la \cite[conj. 10.3.I]{BLR} sur les modèles de Néron globaux.}. 

\begin{lem}\label{lemme fibres grassmannienne globale}
	Les fibres géométriques de $\on{GR}_{\underline{\Gscr_\fbf}}$ sont projectives.
\end{lem}

\begin{proof}
On n'a besoin que de comprendre les fibres géométriques sur les points fermés du schéma de Jacobson $\AAA^1_\ZZ$, parce que la partie des fibres propres est constructible. Mais alors le lemme résulte de ce que $\underline{\Gscr_\fbf}\otimes \FF^{\on{alg}}_p\pot{u}$ soit parahorique, où $u=t-a$ et $a \in \FF^{\on{alg}}_p$, cf. prop. \ref{proposition drs modele immobilier groupe de tits} et rem. \ref{remarque autres changements de base groupe de tits}. Voir aussi \cite[prop. 6.4]{PZh}.
\end{proof}

L'étape suivante tourne autour des tores induits.

\begin{lem}\label{lemme grassmannienne globale tore}
	$\on{GR}_{\underline{\Tscr}}$ est une réunion disjointe dénombrable de schémas finis.
\end{lem}

\begin{proof}
	Cela résulte presque directement de \cite[lem. 3.8, \S4.4.2]{HRi2}. Dans une première version de l'article, nous avons suivi plutôt les démonstrations de \cite[prop. 3.4]{ZhuCoh} et \cite[prop. 4.2.8]{Lev}, en construisant des sections explicites.
\end{proof}

La fibre géométrique générique de $\on{GR}_{\underline{\Tscr}}$ en caractéristique nulle est isomorphe à la grassmannienne affine locale constante $\Gr_{T_H}$, donc numérotée par le groupe des cocaractères géométriques $X_*(T_H)$ du tore $T$, comp. avec \cite[prop. 6.5]{PZh}. Le point associé au copoids $\mu \in X_*(T_H)$ s'étend uniquement par le lemme précédent en un point \begin{equation}z^{\mu} : \AAA^{1/\infty}\rightarrow \on{GR}_{\underline{\Tscr}}\end{equation} de la grassmannienne affine globale à valeurs dans le spectre de $\ZZ[\zeta_\infty,t^{1/\infty}]$. Les spécialisés de $z^\mu$ aux corps $k$ le long de $x: \spec k \rightarrow \AAA^{1/\infty}$ sont des translations dans le groupe d'Iwahori--Weyl de $\underline{\Tscr}\otimes k \rpot{z_{x}}$, où $z_{x}$ désigne une uniformisante associée à ce point géométrique. Habituellement, ceux-ci sont interprétés en termes de l'application de Kottwitz, mais celle-ci n'est malheureusement pas disponible dans notre cadre.

\begin{defn}\label{definition schemas de schubert globaux}
	Soit $\mu \in X_*(T_H)_+$ un cocaractère géométrique dominant de $T$. Le schéma de Schubert global, noté $\on{GR}_{\underline{\Gscr_\fbf}, \leq \mu}$, est l'image schématique du morphisme orbite \begin{equation}L^+\underline{\Gscr_\fbf} \rightarrow \on{GR}_{\underline{\Gscr_\fbf}} , \end{equation} \begin{equation}
	g \mapsto gz^\mu
	\end{equation} défini au-dessus de $\AAA^{1/\infty}$.
\end{defn}

Disons un mot sur quelques fait élémentaires concernant ce schéma de Schubert global. Il est sans torsion, stable sous $L^+\underline{\Gscr_\fbf}$ et projectif, grâce à la preuve de \cite[prop. 6.5]{PZh}. Néanmoins, on ne sait pas pour l'instant s'il est plat, voir cependant le th. \ref{theoreme normalite schema de schubert global et fibres geometriques admissibles}. 

Quoiqu'il en soit, les réduits des fibres géométriques de $\on{GR}_{\underline{\Gscr_\fbf}, \leq \{\mu\}}$ sont isomorphes à la réunion de certaines variétés de Schubert de la grassmannienne affine locale $\Gr_{\underline{\Gscr_\fbf}\otimes k\pot{z_x}}$ correspondante. Ceci nous amène à la définition du lieu admissible :
\begin{equation}\label{equation definition lieu admissible}\on{Adm}_{\underline{\Gscr_\fbf}\otimes k\pot{z_x},\leq \mu}:= \bigcup_{\lambda\in W_G\cdot \mu} \Gr_{\underline{\Gscr_\fbf}\otimes k\pot{z_x}, \leq z^{\lambda}},\end{equation}
lequel est évidemment contenu dans la fibre \begin{equation}\label{equation fibre speciale schema de schubert global}\on{GR}_{\underline{\Gscr_\fbf}, \leq \mu}\otimes_{\ZZ[\zeta_\infty,t^{1/\infty}],x} k \subset \Gr_{\underline{\Gscr_\fbf}\otimes k\pot{z_x}}.\end{equation} On verra ci-dessous que ces deux sous-schémas fermés coïncident dès que le dérivé de $G$ est simplement connexe.

\begin{proof}[Démonstration du th. \ref{theoreme grassmannienne globale projective}]
La représentabilité par un ind-schéma quasi-projectif fut déjà remarquée dans la prop. \ref{proposition representabilite grassmannienne affine globale} et il suffit de montrer l'ind-propreté de $\on{GR}_{\underline{\Gscr_\fbf}}$. Leurs fibres géométriques sont projectives d'après le lem. \ref{lemme fibres grassmannienne globale}. Nous avons déjà remarqué que les schémas de Schubert globales sont projectifs, en répétant la preuve de \cite[prop. 6.5]{PZh}. 

Par suite, il reste à prouver que tout point de $\on{GR}_{\underline{\mathscr{G}_\fbf}}$ en est contenu dans la colimite de ses sous-schémas de Schubert. Mais la fibre géométrique de celle-ci sur $\overline{x}$ contienne dans toutes les variétés de Schubert de $\Gr_{\underline{\Gscr_\fbf}\otimes k\pot{z_{\overline{x}}}}$ associées aux translations du groupe d'Iwahori--Weyl, recouvrant totalement l'espace topologique de cette grassmannienne affine locale.
\end{proof}

\subsection{Le théorème de la cohérence}\label{normalite des varietes de schubert globales et leurs composantes irreductibles}

Enfin, on étudie la géométrie du schéma de Schubert globale $\on{GR}_{\underline{\Gscr_{\fbf}},\leq \mu}$ dans le cas où le dérivé $G^{\on{dér}}=G^{\on{sc}}$ est simplement connexe. Pour formuler le th. \ref{theoreme normalite schema de schubert global et fibres geometriques admissibles} ci-dessous, on a besoin de montrer d'abord sa variante cohomologique due à Pappas--Rapoport, cf. \cite[conj. 10.5]{PR}, et améliorée par Zhu, cf. \cite[th. 2]{ZhuCoh}.

\begin{thm}[Zhu]\label{theoreme de coherence cohomologie fibres amples}
Soit $\Lh$ un faisceau inversible ample sur $\on{GR}_{\underline{\Gscr_\fbf}}^{\on{réd}}$ et supposons que $G^{\on{dér}}=G^{\on{sc}}$. Alors l'entier positif \begin{equation}\dim_kH^0(\on{Adm}_{\underline{\Gscr_\fbf}\otimes k\pot{z_x}},\Lh)\end{equation} ne dépend pas du choix de $x$.	
\end{thm}

\begin{proof}
	Le premier truc que nous devons établir est que, pour un faisceau ample inversible sur la grassmannienne globale $\on{GR}_{\underline{\Gscr_\fbf},\leq \mu}$, alors la charge centrale de la restriction de $\Lh$ à $\Gr_{\underline{\Gscr_\fbf}\otimes k\pot{z_x}}$ définie à (\ref{equation charge centrale entiers}) est un invariant de $\Lh$ qui ne dépend pas du point $x$. On remarque que, lorsque le groupe d'Iwahori-Weyl $\widetilde{W}$ de $\underline{\Gscr_\fbf}\otimes k\rpot{z_x}$ n'est pas irréductible, la charge centrale doit plutôt s'entendre comme un $n$-uplet d'entiers, chacun d'eux associé à une composante irréductible du groupe d'Iwahori-Weyl ; que la charge centrale soit constante doit être entendu au sens des morphismes de spécialisation de la grassmannienne affine globale. 
	
	Ceci se trouve dans \cite[prop. 4.1]{ZhuCoh} à coefficients dans $k$ pour les groupes modérément ramifiés, mais nous allons donner une preuve explicite dans le cas où la grassmannienne affine est associée à un groupe pseudo-réductif exotique et barcelonais. Autrement dit, nous supposons que $\Phi_G$ est irréductible et que $p=e=2,3$. Or, selon les ex. \ref{exemple exotique basique} et \ref{exemple absolument non reduit}, nous disposons d'une application canonique \begin{equation}G_{\eta_e} \rightarrow \overline{G},\end{equation} où $\overline{G}$ est la restriction des scalaires purement inséparable d'un $\FF_e$-groupe déployé simplement connexe, tel que l'on ait des bijections d'immeubles \begin{equation}\Iscr(G_{\eta_e}, \FF_e\rpot{t})\simeq \Iscr(\overline{G}, \FF_e\rpot{t}),\end{equation} de points entiers et rationnels \begin{equation}G_{\eta_e}(\FF_e^{\on{alg}}\pot{t}) \simeq \overline{G}(\FF_e^{\on{alg}}\pot{t}) ,\end{equation} induisant un homéomorphisme universel \begin{equation}\Gr_{\Gscr_{\fbf,e}}\rightarrow \Gr_{\overline{\Gscr_{\fbf}}},\end{equation} cf. \cite[prop. 3.6]{LouBT}. Les morphismes induits $\PP^1_s \rightarrow \PP^1_s$ entre les droites de Schubert ont degrés soit $1$, soit $e$, bien déterminés par la combinatoire. Pour que la charge centrale soit constante, il faut alors balancer cet effet de sorte que les raisons $a_s^{\vee}/\overline{a}^\vee_s$ des coefficients de Kac--Moody, c'est-à-dire la valeur $c(\Lh_s)$ de la charge centrale $c$ dans $\Lh_s$, cf. cor. \ref{corollaire charge centrale}, associés à l'un et l'autre groupe soient égaux à $e$ ou à $1$, selon les degrés ci-dessus soient égaux à $1$ ou $e$. La relation prévue entre les coefficients peut être facilement vérifiée aux tableaux de \cite[annexe]{Car}. Ici on utilise alors la comparaison de l'annexe \ref{comparaison avec les varietes de demazure de kac-moody}, cf. (\ref{equation comparaison algebres de lie}).
	
	Finalement, nous voulons montrer des égalités \begin{equation}\dim_{\FF_p}H^0\Big(\bigcup_{i=1}^{n}\Gr_{\Gscr_{\fbf,p},\leq w_i},\Lh\Big)=\dim_{\FF_l}H^0\Big(\bigcup_{i=1}^{n}\Gr_{\Gscr_{\fbf,l},\leq w_i},\Lh\Big),\end{equation}
	où $w_i \in \widetilde{W}/W_\fbf$ sont des éléments quelconques, $p,l \in \PP \cup \{0\}$ sont des caractéristiques des corps premiers et $\Lh$ est un faisceau inversible très ample sur $\Gr_{\Gscr_\fbf}$, cf. \cite[cor. 4.4]{ZhuCoh}. Cela revient au même que savoir que la formation des réunions des schémas de Schubert commute aux changements de base quelconques. Nous observons d'abord que, grâce à un argument avec une suite exacte du type Mayer--Vietoris, on a une formule \begin{equation}\dim_{\FF_p}H^0\Big(\bigcup_{i=1}^{n}\Gr_{\Gscr_{\fbf,p},\leq w_i},\Lh\Big)= \sum_{J \subseteq I} (-1)^{j-1} \dim_{\FF_p} H^0\big(\bigcap_{j\in J}\Gr_{\Gscr_{\fbf,p},\leq w_j},\Lh\big)\end{equation} de type inclusion-exclusion, puisque les groupes de cohomologie supérieurs s'annulent. Comme les intersections des variétés de Schubert sont réduites sont réduites, puisque scindées dès que $p>0$, cf. prop. \ref{proposition varietes de schubert proprietes}, elles coïncident avec la réunion \begin{equation} \bigcap_{j\in J}\Gr_{\Gscr_\fbf,p,\leq w_j}= \bigcup_{k \in K} \Gr_{\Gscr_\fbf,p,\leq v_k}\end{equation} des variétés de Schubert contenues là, où l'ensemble des $v_k$ ne dépend pas de la caractéristique $p$, voir \cite[prop. 2.8]{RichDipl}. En particulier, on se ramène par récurrence décroissante au cas où $\lvert I \rvert=1$, lequel résulte de notre résultat sur les schémas de Schubert, voir th. \ref{theoreme normalite des schemas de Schubert}.
\end{proof}

Sur une base unidimensionnelle, le théorème précédent suffit pour démontrer que les schémas de Schubert globaux ont de fibres géométriques réduites, données par les lieux admissibles, cf. \cite[th. 3, \S4.2]{ZhuCoh} ainsi que \cite[th. 9.1]{PZh}. Dans la présente situation, il va falloir être plus soigneux avec l'argumentation géométrique.

\begin{thm}[L.]\label{theoreme normalite schema de schubert global et fibres geometriques admissibles}
Si le groupe dérivé $G^{\on{dér}}=G^{\on{sc}}$ est simplement connexe, alors le schéma de Schubert global $\on{GR}_{\underline{\Gscr_{\fbf}},\leq \mu}$ est normal, plat, géométriquement réduit et connexe, à fibres égales à $\on{Adm}_{\underline{\Gscr_\fbf}\otimes k\pot{z_x},\leq \mu}$.
\end{thm}

\begin{proof}
Nous avons déjà vu à (\ref{equation definition lieu admissible}) et (\ref{equation fibre speciale schema de schubert global}) que le lieu admissible \begin{equation} \on{Adm}_{\underline{\Gscr_\fbf}\otimes k\pot{z_x},\leq \mu}\subset \on{GR}_{\underline{\Gscr_{\fbf}},\leq \mu} \otimes_{\AAA^{1/\infty},x} k\end{equation} se trouve dedans la fibre spéciale du schéma de Schubert en tant que sous-schémas fermés de grassmannienne affine locale $\Gr_{\underline{\Gscr_\fbf}\otimes k\pot{z_x}}$. Afin d'obtenir l'égalité désirée, il faudra calculer les sections globales de fibrés en droites très amples, dont l'outil principal est le th. \ref{theoreme de coherence cohomologie fibres amples}. 

Cependant, cette utilisation aura lieu seulement quand le schéma en question est plat sur la base, pour que la caractéristique d'Euler des fibres soit localement constante. Par suite, le th. \ref{theoreme de coherence cohomologie fibres amples} entraîne que l'affirmation du présent énoncé est vraie au-dessus d'un ouvert $U \subset \AAA^{1/\infty}$ à complément pro-fini. De plus, en prenant des changements de base aux anneaux de valuation $V$ de rang $1$ sur le schéma de Schubert global dont le point fermé s'envoie sur un point générique de la fibre sur $x$, on en déduit que l'espace topologique de $\on{GR}_{\underline{\Gscr_{\fbf}},\leq \mu}$ est le bon, c'est-à-dire notre assertion sur ses fibres est valable à structures non réduites près. En particulier, le schéma de Schubert global est régulier en codimension $1$ par comparaison de cardinaux.

Considérons ainsi le schéma de Schubert global normalisé $\widetilde{\on{GR}}_{\underline{\Gscr_{\fbf}},\leq \mu}$ dans son corps des fractions. Celui-ci est plat sur la base grâce au \cite[lem. VII.3.2]{Ray}, voir aussi la prop. \ref{proposition enveloppe affine}, vu que la codimension du lieu plat du schéma de Schubert global est au moins $2$ et que celui-ci est régulier en codimension $1$. Puis, on affirme que l'application \begin{equation}\widetilde{\on{GR}}_{\underline{\Gscr_{\fbf}},\leq \mu} \otimes_{x} k \rightarrow \on{GR}_{\underline{\Gscr_{\fbf}},\leq \mu} \otimes_{x} k\end{equation} est une immersion fermée pour tout point $x$ à valeurs dans $k$. En effet, le $\AAA^{1/\infty}_k$-schéma plat \begin{equation}\widetilde{\on{GR}}_{\underline{\Gscr_{\fbf}},\leq \mu} \otimes_{\AAA^{1/\infty}_{\ZZ[\zeta_\infty]}}\AAA^{1/\infty}_k\end{equation} déduit par changement de base s'applique isomorphiquement d'après le théorème principal de Zariski sur l'adhérence plate de \begin{equation}\on{GR}_{\underline{\Gscr_{\fbf}},\leq \mu} \otimes_{\AAA^{1/\infty}_{\ZZ[\zeta_\infty]}}\AAA^{1/\infty}_k,\end{equation} car cette dernière est normale, géométriquement réduite et à fibres données par les lieux admissibles, cf. th. \ref{theoreme de coherence cohomologie fibres amples}. Enfin, le fait que l'application de normalisation
\begin{equation}
\widetilde{\on{GR}}_{\underline{\Gscr_{\fbf}},\leq \mu} \rightarrow \on{GR}_{\underline{\Gscr_{\fbf}},\leq \mu} 
\end{equation}
 soit un isomorphisme devient une conséquence du lemme de Nakayama, comp. avec la fin de la preuve de \cite[prop. 8.1]{PZh}.
\end{proof}

\begin{rem}
	Pour finir cette section, étendrons-nous un peu sur les schémas déduits du schéma de Schubert global $\on{GR}_{\underline{\Gscr_{\fbf}},\leq \mu}$ par quelques changements de base particuliers. En caractéristiques égales, les objets échéants généralisent à un certain sens les modèles locaux en caractéristiques égales au sens de Zhu, voir \cite[déf. 3.1, th. 3]{ZhuCoh}, et au sens de Richarz, voir \cite[déf. 2.5]{Rich}. En caractéristique mixte, alors on obtient en même temps les modèles locaux au sens de Pappas--Zhu, voir \cite[déf. 7.1, th. 9.1]{PZh} et de Levin, voir \cite[déf. 4.2.1, th. 4.3.2]{Lev}, mais aussi quelques-uns nouveaux en des cas sauvagement ramifiés. En guise d'application, nous avons étendu \cite[th. 2.14, cor. 2.16]{HPR} concernant la conjecture de Scholze--Weinstein sur les modèles locaux à presque tous les cas de type abélien dans \cite[IV, prop. 4.22, cor. 4.24]{LouDiss}. Récemment, on a fait des avances non négligeables sur le cas général de la conjecture, dont on rapportera dans un article futur.
\end{rem}

\begin{appendix}

	\section{Le lien avec le cadre de Kac--Moody} \label{comparaison avec les varietes de demazure de kac-moody} 

Nous arrivons maintenant au paragraphe concernant le lien entre les $\ZZ[t^{\pm1}]$-groupes de Tits $\underline{G}$ de la déf. \ref{definition groupe de Tits}, ainsi que leurs $\ZZ[t]$-modèles parahoriques $\underline{\Gscr_\fbf}$ associés aux facettes $\fbf \subset \Ascr(\underline{G},\underline{S},\ZZ\rpot{t})$ par la déf. \ref{definition modeles immobiliers du groupe de tits}, et les $\ZZ$-groupes $\Gh_\ZZ$ de (\ref{equation groupe de kac-moody entiers}), ainsi que leurs sous-groupes paraboliques $\Ph^{\pm}_\ZZ$ à (\ref{equation parabolique positif kac-moody entiers}) et (\ref{equation parabolique negatif kac-moody entiers}), construits par Mathieu dans le cadre de Kac--Moody affine, voir les cors. \ref{corollaire comparaison kac-moody parabolique} et \ref{corollaire comparaison groupes kac-moody}. Cela résulte également d'un théorème de comparaison entre nos grassmanniennes affines $\Gr_{\underline{\Gscr_\fbf}}$ et les variétés de drapeaux affines $\Fh_{\gf/\pf,\ZZ}$ de Kac-Moody, cf. th. \ref{theoreme comparaison kac-moody schubert}. À la fin, on pourra regarder les $\underline{\Gscr_\fbf}$ comme des sortes de \og délacés \fg{} des groupes de Kac--Moody affines à coefficients entiers.

\subsection{Les $\ZZ$-groupes de Kac--Moody et leurs variétés de drapeaux}\label{subsection theorie de kac-moody}
Rappelons les notions basiques sur les algèbres de Kac--Moody et les groupes qui leur sont associés. 

On dit qu'une matrice $A \in M_n(\ZZ)$ carrée d'ordre $n$ à entrées entières est de Cartan généralisée, cf. \cite[déf. 3.7]{Mar}, si les conditions suivantes sont remplies : $a_{ii}=2$, $a_{ij} \in \ZZ_{\leq 0}$ et $a_{ij} \neq 0$ si et seulement si $a_{ji} \neq 0$. On fixe désormais une matrice de Cartan généralisée $A$ de rang $r$, supposée d'ailleurs symétrisable. 

La définition suivante est une synthèse de \cite[déf. 3.10, déf. 7.9]{Mar} :

\begin{defn}Une réalisation entière de $A$ est la donnée d'un $5$-tuplet \begin{equation}\label{equation realisation entiere kac-moody}(\tf, P, P^{\vee}, \Pi, \Pi^{\vee})\end{equation} qui se compose d'un $\QQ$-espace vectoriel $\tf$ de dimension $2n-r$, des $\ZZ$-réseaux duaux \begin{equation} P \subset \tf^\vee\text{ et }P^{\vee} \subset \tf\end{equation} et des parties libres \begin{equation}\Pi^\vee=\{ \alpha_i^\vee: i=1, \dots, n\} \subset \tf\text{ et } \Pi=\{ \alpha_i: i=1, \dots, n \} \subset \tf^\vee\end{equation} à matrice d'accouplement $A$, c'est-à-dire $\langle \alpha_i^{\vee}, \alpha_j\rangle=a_{ij}$, de sorte que, lorsqu'on pose \begin{equation}Q=\bigoplus_{i=1}^n\ZZ \alpha_i \text{ et }Q^{\vee}=\bigoplus_{i=1}^n \ZZ \alpha_i^\vee,\end{equation} alors $Q \subset P$, $Q^{\vee} \subset P^\vee$ et $P^{\vee}/Q^{\vee}$ est sans torsion. 
\end{defn}
	
	Ces données sont uniquement déterminées par la matrice $A$ à isomorphisme unique près, cf. \cite[ex. 3.11, ex. 7.10]{Mar}. Les algèbres de Kac-Moody sont définies en termes des présentations comme suit, comp. avec \cite[ch. I, (I)-(VII)]{Mat} et \cite[déf. 3.17]{Mar}.

\begin{defn}
	L'algèbre de Kac--Moody $\gf$ attachée à $A$ est la $\QQ$-algèbre de Lie engendrée par $\tf$ et par des $\tf$-vecteurs propres $e_i$ et $f_i$ à valeurs propres $\pm \alpha_i$, où $i=1, \dots, n$, tels que \begin{equation} [e_i,f_j]=-\delta_{ij}\alpha_i^\vee\end{equation} et soumis aux relations de Serre \begin{equation}
	(\on{ad}e_i)^{1-a_{ij}}e_j=0,
	\end{equation}
	\begin{equation}
	(\on{ad}f_i)^{1-a_{ij}}f_j=0
	\end{equation}  pour tout $1\leq i \neq j \leq n$.
\end{defn} 

Ainsi que pour les algèbres de Lie semi-simples, il faut regarder $\tf$ comme sous-algèbre de Cartan de $\gf$ et les $e_i$ et $f_i$ comme vecteurs radiciels associés aux racines simples $\pm \alpha_i$ du système de racines $\Phi$ de $\gf$ suivant $\tf$. 

Expliquons la construction des $\QQ$-sous-groupes paraboliques de Kac--Moody à la suite de Kumar, voir \cite[ch. VI]{Kum}. Étant donnée une sous-algébre parabolique positive \begin{equation} \tf \subset \pf^+ \subset \gf\end{equation} invariante sous $\tf$, alors on a un choix unique de décomposition $\tf$-équivariante de Levi \begin{equation}\pf^+=\mf \oplus \nf^+.\end{equation} Ceci nous amène à considérer, d'un côté, le seul $\ZZ$-groupe épinglé $\Mh_\ZZ$ associé à la donnée radicielle de $\mf$ induite par la réalisation entière (\ref{equation realisation entiere kac-moody}) ci-dessus, et de l'autre côté, le seul $\QQ$-groupe pro-unipotent \begin{equation} \Uh^+_\QQ := \on{exp}(\widehat{\nf}^+)\end{equation}  associé au séparé complété de $\nf^+$ pour sa topologie naturelle, d'après \cite[th. 4.4.19]{Kum}. Puis, on voit que l'exponentielle relève l'action de l'algèbre de Lie en une opération \begin{equation}\Mh_\QQ \times \Uh^+_\QQ \rightarrow \Uh_\QQ^+,\end{equation} de sorte qu'on puisse prendre le produit semi-direct \begin{equation} \Ph_\QQ^+:=\Mh_\QQ\ltimes \Uh_\QQ^+.\end{equation}

En ce qui concerne les modèles entiers des groupes de Kac--Moody, la stratégie adoptée remonte à Kostant \cite{Kos}. L'idée est que leur géométrie devrait être codifiée dans la famille de leurs opérateurs différentiels supérieurs, cf. preuve du th. \ref{theoreme normalite des schemas de Schubert} en \S\ref{normalite des varietes de schubert}. Par suite, on a besoin de trouver un $\ZZ$-réseau convenable \begin{equation}\label{equation reseau algebre enveloppante kac-moody} U_\ZZ(\gf) \subset U(\gf) \end{equation} dedans l'algèbre enveloppante de l'algèbre de Kac--Moody $\gf$. Ceci fut mené dans le présent cadre par Tits, voir \cite{Ti13}, \cite[\S4]{TitsUniq} et \cite[déf. 7.3]{Mar}.

\begin{defn}\label{definition forme de kostant-tits}
	La $\ZZ$-forme de Kostant--Tits notée $U_{\ZZ}(\gf)$ de l'algèbre enveloppante de $\gf$ est la sous-$\ZZ$-algèbre associative 
	\begin{equation}\label{equation forme de kostant-tits algebre enveloppante}
	U_\ZZ(\gf):=\ZZ\Big[\frac{e_i^n}{n!},\frac{f_i^n}{n!},\binom{h}{n} \Big]
	\end{equation}engendrée par des puissances divisées tordues, où $h \in P^\vee$.
\end{defn}

Alors, on peut prolonger le $\QQ$-groupe pro-algébrique $\Uh^+_\QQ$ en un $\ZZ$-schéma en groupes affine et plat \begin{equation} \Uh^+_\ZZ:=\spec \Gamma(\Uh^+_\ZZ, \Oh), \end{equation} en prenant pour sections globales le dual gradué de $U_\ZZ(\nf^+)$, cf. \cite[déf. 8.41]{Mar} ; celui-ci coïncide avec la partie de $\Gamma(\Uh^+_\QQ, \Oh)$ duale au réseau de (\ref{equation reseau algebre enveloppante kac-moody}) par rapport à l'accouplement \begin{equation} \Gamma(\Uh^+_\QQ, \Oh) \times U(\nf^+) \rightarrow \QQ \end{equation} provenant du lien entre les distributions et l'algèbre enveloppante de Lie, cf. \cite[lem. 8]{Mat} et \cite[6.2]{TitsBour}. On remarque que le $\ZZ$-groupe épinglé $\Mh_\ZZ$ opère encore sur $\Uh_\ZZ^+$ d'où le modèle entier \begin{equation}\label{equation parabolique positif kac-moody entiers} \Ph_\ZZ^+:=\Mh_\ZZ \ltimes \Uh_\ZZ^+,\end{equation} voir \cite[cor. 3.10]{Rou16}.

Il y a une théorie des $\gf$-modules irréductibles intégrables $V(\lambda)$ de plus grand poids $\lambda$, où ce dernier est un caractère régulier dominant de $\Mh_\ZZ$, cf. \cite[\S9]{Kac} et \cite[\S4.1]{Mar}, déterminés en tant que quotients du module de Verma
\begin{equation}\label{equation module weyl comme quotient verma rationnels}
U(\gf)\otimes_{U(\pf)}\QQ_\lambda \twoheadrightarrow V(\lambda).
\end{equation} La $\ZZ$-forme $U_\ZZ(\gf)$ nous permet d'introduire le réseau \begin{equation} V_\ZZ(\lambda)\subset V(\lambda)\end{equation} donné par l'image de $U_\ZZ(\gf)\otimes_{U_\ZZ(\pf^+)} \ZZ_\lambda$ suivant (\ref{equation module weyl comme quotient verma rationnels}), voir \cite[ex. 7.23]{Mar}. Ceci étant, on se tourne vers les schémas de Schubert sur les entiers, cf. \cite[p. 37]{Mat2}.

\begin{defn}\label{definition kac-moody schemas de schubert}
	Soit $\lambda$ un poids régulier dominant de $\Mh_\ZZ$. Alors, à chaque $w \in W/W_\pf$, on lui associe le schéma de Schubert \begin{equation} \Sh_{w,\lambda,\ZZ}:= \overline{\Bh_\ZZ^+\cdot v_{w\lambda}} \subset \PP(V_\ZZ(\lambda)) \end{equation} donné par l'orbite fermée d'un générateur $v_{w\lambda}$ de l'espace de poids $w\lambda$ de $V_\ZZ(\lambda)$ dedans l'ind-espace projectif qui lui est associé.
\end{defn}

Heureusement, les schémas de Schubert normalisés ne dépendent que du niveau parabolique $\pf$ et pas du choix de poids $\lambda$, grâce à un argument topologique, cf. \cite[lem. 55]{Mat}, avec les schémas de Demazure 
\begin{equation}\label{equation schemas de demazure kac-moody}\Dh_{\wf,\ZZ}:=\Ph_{s_{i_1},\ZZ}^+\times^{\Bh^+_\ZZ} \dots \times^{\Bh^+_\ZZ} \Ph_{s_{i_n},\ZZ}^+/\Bh^+_\ZZ.\end{equation}
 Puis on dispose de même du théorème de normalité dans le cadre de Kac--Moody :
 \begin{thm}[Mathieu, Littelmann]\label{theoreme kac-moody normalite}
 	Les schémas de Schubert $\Sh_{w,\ZZ}$ associés à $w \in W/W_\pf$ sont géométriquement normaux.
 \end{thm}

\begin{proof}[Esquisse de la preuve]
Ceci repose sur la formule de caractère de Demazure pour les $U_\ZZ(\gf)$-modules irréductibles $V_\ZZ(\lambda)$ de plus grand poids. Mathieu recourt dans \cite[cor. 1]{Mat2} aux techniques de scindage pour les variétés de Schubert normalisées en caractéristique $p>0$, tandis que Littelmann applique sa méthodes des chemins pour y parvenir, cf. \cite[\S8]{Lit}.
\end{proof}
  La variété de drapeaux de Kac--Moody à niveau parabolique $\pf$ introduite par Mathieu est donc la limite inductive filtrée \begin{equation}\Fh_{\gf/\pf,\ZZ} :=\on{colim}_{w\in W/W_\pf}\Sh_{w,\ZZ}\end{equation} des schémas de Schubert. Il ne reste qu'à ériger le groupe complet $\Gh_\ZZ$ de Kac--Moody. 
  
  Supposons pour l'instant que $\pf^+=\bff^+$ est une sous-algèbre borélienne. On remarque d'abord que le schéma de Demazure $D_{\wf,\ZZ}$ de (\ref{equation schemas de demazure kac-moody}) possède un $\Bh_\ZZ^+$-torseur naturel \begin{equation}\Eh_{\wf,\ZZ}:=\Ph_{s_{i_1},\ZZ}^+\times^{\Bh^+_\ZZ} \dots \times^{\Bh^+_\ZZ} \Ph_{s_{i_n},\ZZ}^+. \end{equation} Il en résulte alors un $\Bh_\ZZ^+$-torseur canonique sur $\Sh_{w,\ZZ}$ noté $\Gh_{w,\ZZ}$ et déterminé en tant que l'enveloppe affine de $\Eh_{\wf,\ZZ}$, voir \cite[prop. 26]{Mat}. Ainsi, l'on note 
  \begin{equation}\label{equation groupe de kac-moody entiers}\Gh_\ZZ:=\on{colim}_{w\in W} \Gh_{w,\ZZ} \end{equation} l'ind-schéma en groupes affine sur $\ZZ$ qui est la limite directe des torseurs précédents, comp. avec \cite[p. 45]{Mat2} et \cite[\S3.6]{Rou16}. 
  
  Plus généralement, les sous-groupes paraboliques positifs $\Ph_\ZZ^+$ se plongent dans $\Gh_\ZZ$ en tant que sous-$\ZZ$-groupes fermés, tels que le quotient pour la topologie plate \begin{equation}\Gh_\ZZ/\Ph_\ZZ^+ \simeq \Fh_{\gf/\pf,\ZZ} \end{equation} s'identifie canoniquement à la variété de drapeaux partielle. 
  
  En outre, le groupe de Kac--Moody possède également des sous-groupes paraboliques négatifs ou opposés \begin{equation}\label{equation parabolique negatif kac-moody entiers} \Ph^-_\ZZ \subset \Gh_\ZZ,\end{equation} définis comme répulseurs dans $\Gh_\ZZ$ d'un copoids dominant régulier de $\Mh_\ZZ$, tels que l'on ait une décomposition de Levi
  \begin{equation}
  \Ph^-_\ZZ=\Mh_\ZZ\ltimes \Uh_\ZZ^-,
  \end{equation}
  où l'on note $\Uh_\ZZ^-$ le répulseur strict, et que l'application
  \begin{equation}
  \Uh_\ZZ^- \times \Ph_\ZZ^+ \rightarrow \Gh_\ZZ
  \end{equation} définisse une immersion ouverte et affine, voir \cite[lem. 4.1]{HLR} pour toutes ces affirmations et comparer avec la prop. \ref{proposition dynamique groupes d'arcs et de cordes}.

\subsection{La comparaison avec les grassmanniennes affines}\label{subsection comparaison de kac-moody}
Maintenant, on va considérer toutes les constructions du paragraphe précédent dans le cas particulier où $\gf$ est de type affine, c'est-à-dire $r=n-1$. 

Or, la classification des algèbres de Kac--Moody est faite en termes de diagrammes de Dynkin affines, cf. \cite[\S4.8]{Kac} et \cite[p. 77]{Mar} : on a la partie des diagrammes non tordus \begin{equation}A_n^{(1)}, B_n^{(1)}, C_n^{(1)}, D_n^{(1)}, E_6^{(1)}, E_7^{(1)}, E_8^{(1)}, F_4^{(1)}, G_2^{(1)} ;\end{equation} celle des diagrammes tordus réduits \begin{equation}A_{2n-1}^{(2)}, D_n^{(2)}, E_6^{(2)}, D_4^{(3)} ;\end{equation} et enfin le seul diagramme tordu non réduit \begin{equation} A_{2n}^{(2)}.\end{equation} 

Quoiqu'il en soit, nous disposons em même temps d'un et d'un seul $\QQ(t)$-groupe $G$ simplement connexe et absolument presque simple à isomorphisme près, satisfaisant à l'hyp. \ref{hypothese groupe dans la fibre generique}, et dont le diagramme de Dynkin affine est celui donné, qu'on fixera désormais. L'objet du présent paragraphe est de comparer les données géométriques de Kac--Moody affines décrites ci-dessus à celles du \S\ref{section grassmanniennes affines}. 

Traitons d'abord les opérateurs différentiels. On note ci-dessous $\widetilde{L\underline{G}}$ le produit semi-direct
\begin{equation}
\widetilde{L\underline{G}}:=\widehat{L\underline{G}} \rtimes \GG_{m,\ZZ}
\end{equation}
 de l'extension centrale universelle $ \widehat{L\underline{G}}$ de (\ref{equation charge centrale entiers}) par l'action de rotation du groupe multiplicatif $\GG_{m,\ZZ}$, voir le lem. \ref{lemme action de rotation t unite}. Ces définitions s'appliquent également aux groupes d'élastiques $L^{\on{gr}}\underline{G}$, donc on fera usage des mêmes notations.

\begin{lem}\label{lemme identification forme de tits et distributions}
	Il y a un isomorphisme $W^{\on{af}}$-équivariant
	\begin{equation}\label{equation identification distributions kac-moody}U_\ZZ(\gf)\simeq \on{Dist}\big(\widetilde{L^{\on{gr}}\underline{G}}\big)\end{equation}
	de $\ZZ$-algèbres associatives.
\end{lem}

\begin{proof} 
La première étape consiste à construire l'isomorphisme rationnel, ce qui ramène le problème à trouver un isomorphisme \begin{equation}\label{equation comparaison algebres de lie}\gf \simeq \on{Lie}\big(\widetilde{L^{\on{gr}}G}\big),\end{equation} compte tenu de ce que le membre de droite de (\ref{equation identification distributions kac-moody}) n'est autre que l'algèbre enveloppante de celui de (\ref{equation comparaison algebres de lie}). Cette dernière identification est toutefois une conséquence des descriptions explicites des algèbres de Kac--Moody affines $\gf$ dans \cite[ths. 7.4, 8.3]{Kac}, jointes à la détermination de la charge centrale dans (\ref{equation charge centrale formule explicite coefficients kac-moody}).

Penchons-nous sur l'égalité \begin{equation}\gf_\ZZ =\on{Lie}\big(\widetilde{L^{\on{gr}}\underline{G}}\big)\end{equation}
de réseaux entiers des algèbres de Lie. Il s'avère que la base évidente du membre de droite induite par le quasi-système de Tits est une base de Chevalley du membre de gauche au sens de \cite[\S\S2-3]{Mitz} ; on attire l'attention particulière du lecteur à la modification de \cite[(3.5.33)]{Mitz} pour le cas non réduit, qui correspond exactement à celle de la déf. \ref{definition quasi-systeme de Tits}.

 Or, comme \begin{equation}\on{Dist}(\GG_{a,\ZZ})=\ZZ\Big[\frac{X^n}{n!}\Big] \text{ et } \on{Dist}(\GG_{m,\ZZ})=\ZZ\Big[\binom{X}{n}\Big],\end{equation} cf. \cite[I, ex. 7.8]{Jan}, on en tire l'inclusion $U_\ZZ(\gf)\subset \on{Dist}(\widetilde{L^{\on{gr}}\underline{G}})$ énoncée par (\ref{equation identification distributions kac-moody}). D'autre part, le lem. \ref{lemme engendrement distributions} entraîne que cette dernière algèbre est engendrée par $\on{Dist}(L^{\on{gr}}\underline{U}^{\pm})$ et $\on{Dist}(\widetilde{S})$. Par conséquent, on achève le lemme voulu.
\end{proof}

\begin{rem}
	L'algèbre de distributions de $L^{\on{gr}}\GG_{m,\ZZ}$ peut être calculée explicitement. Celle-ci est donnée en fonction de certains polynômes en un nombre dénombrable de variables qui correspondent à une base de $\on{Lie}(L^{\on{gr}}\GG_{m,\ZZ})$, la difficulté résidant en les déterminer. Raisonnant par récurrence, on récupère les polynômes de Garland--Mitzman, voir \cite[(5.7), th. 5.8]{Gar1}, \cite[\S4.1]{Mitz} et \cite[ex. 8.60]{Mar}.
\end{rem}

Grâce à la prop. \ref{proposition parties parahoriques et facettes}, on dispose d'une correspondance bijective entre les facettes $\fbf \subset \Ascr(\underline{G},\underline{S},\ZZ\rpot{t})$ et les sous-algèbres paraboliques positives $\pf^+ \subset \gf$. Ceci se géométrise comme suit :

\begin{cor}\label{corollaire comparaison kac-moody parabolique}
 Il y a un et un seul isomorphisme \begin{equation}\Ph_\ZZ^+ \simeq \widetilde{L^+\underline{\Gscr_\fbf}}\end{equation} relevant celui des algèbres de distributions graduées.
\end{cor}

\begin{proof}
L'identification (\ref{equation comparaison algebres de lie}) concernant les algèbres de Lie fournit déjà un isomorphisme \begin{equation} \Uh_\QQ^+ \simeq L^{++}\Gscr_{\fbf,0}\end{equation} entre les radicaux pro-unipotents sur les rationnels, d'après \cite[th. 4.4.19]{Kum}. D'autre part, les données entières définissant les facteurs de Levi de (\ref{equation parabolique positif kac-moody entiers}) et (\ref{equation decomposition de levi arcs}) coïncident aux extensions près, d'où une identification
\begin{equation}
\Mh_\ZZ \simeq \widetilde{\underline{M_\fbf}}.
\end{equation} Ceci entraîne qu'on puisse identifier
\begin{equation}
\Ph_\QQ^+ \simeq \widetilde{L^{+}\Gscr_{\fbf,0}}
\end{equation} les schémas en groupes sur les rationnels comme envisagé.

Maintenant, nous voulons améliorer ce résultat en étendant les isomorphismes aux modèles entiers. On s'est ramené à construire un isomorphisme entre les radicaux unipotents. En vertu de \cite[rem. 3.5.1(1)]{BTII}, démontré dans \cite[I, prop. 10.12]{Jan}, on peut retrouver les $\ZZ$-groupes lisses unipotents et connexes en connaissant leur algèbres de distributions. En particulier, l'affirmation en résulte du lem. \ref{lemme identification forme de tits et distributions}, passant à la limite projective.
\end{proof}

\begin{thm}[Faltings, Pappas--Rapoport, L.]\label{theoreme comparaison kac-moody schubert}
	Il y a des isomorphismes équivariants naturels 
	\begin{equation}\Sh_{w,\ZZ} \simeq  \Gr_{\underline{\Gscr_\fbf},\leq w}\end{equation}
	entre les schémas de Schubert associés à $w \in W^{\on{af}}/W_\fbf$ en chacun des cadres, induisant un isomorphisme global	$\Fh_{\gf/\pf,\ZZ} \simeq  \Gr_{\underline{\Gscr_\fbf}}$.
\end{thm}

\begin{proof}
On reprend les arguments de \cite[8.d, 9.h]{PR}. Le cor. \ref{corollaire comparaison kac-moody parabolique} fournit un isomorphisme équivariant des résolutions de Demazure, d'où l'identification des schémas de Schubert d'après \cite[lem. 33]{Mat}, appliquant simultanément les théorèmes de normalité en chacun des cadres étudies, cf. th. \ref{theoreme normalite des schemas de Schubert} et th. \ref{theoreme kac-moody normalite}. 

À son tour, la démonstration implicite de \cite[p. 54]{Fal} tire profit de ce que les sections globales duales des faisceaux amples s'identifient
\begin{equation}
\Gamma(\Fh_{\gf/\pf,\ZZ}, \Lh_{-\lambda})^\vee \simeq V_\ZZ(\lambda)
\end{equation} aux modules irréductibles intégrables de l'algèbre de distributions. Identifiant les mêmes objets dans le cadre des grassmanniennes affines, on voit que tous les deux schémas de Schubert se plongent dans les mêmes espaces projectifs.
\end{proof}

Le théorème ci-dessus permet ainsi de comparer également les groupes complets, complétant le dictionnaire entre la théorie de Kac--Moody affine et les grassmanniennes affines.

\begin{cor}[Haines--L.--Richarz]\label{corollaire comparaison groupes kac-moody}
	Il y a un isomorphisme canonique \begin{equation}\Gh_\ZZ \simeq \widetilde{LG}\end{equation} induisant des identifications \begin{equation}\Ph_\ZZ^{\pm} \simeq \widetilde{L^{\pm}\underline{\Gscr_\fbf}}\end{equation} entre les sous-groupes paraboliques et les sous-groupes d'arcs et/ou de cordes.
\end{cor}

\begin{proof}
	On renvoie à \cite[prop. 4.3]{HLR} pour plus de détails. Notons que l'identification des sous-groupes paraboliques opposés aux sous-groupes de cordes résultent plus aisément de la prop. \ref{proposition dynamique groupes d'arcs et de cordes} et de \cite[lem. 4.1]{HLR}. 
\end{proof}

\end{appendix}
\bibliographystyle{amsalpha}

\end{document}